\documentclass[11pt,reqno]{amsart}

\usepackage[top=30truemm,bottom=30truemm,left=25truemm,right=25truemm]{geometry}

\usepackage{amssymb, amsmath}
\usepackage{amsthm}

\makeatletter
    
    \@addtoreset{equation}{section}
  \makeatother

\newtheorem{definition}{Definition}[section]
\newtheorem{proposition}[definition]{Proposition}
\newtheorem{theorem}[definition]{Theorem}
\newtheorem{lemma}[definition]{Lemma}

\newtheorem{remark}{Remark}[section]

\begin{document}




\title[Nonlinear damped wave equation with slowly decaying data]{
The Cauchy problem for the nonlinear damped wave equation
with slowly decaying data
}

\author[M. Ikeda]{Masahiro IKEDA}
\author[T. Inui]{Takahisa INUI}
\author[Y. Wakasugi]{Yuta WAKASUGI}
 \address[M. Ikeda and T. Inui]{Department of Mathematics,
 Graduate School of Science, Kyoto University, Kyoto 606-8502, Japan}
 \address[Y. Wakasugi]{Graduate School of Mathematics, Nagoya University,
 Furocho, Chikusaku, Nagoya 464-8602, Japan}

%
%
%

\begin{abstract}
We study the Cauchy problem for the nonlinear damped wave equation
and establish the large data local well-posedness
and small data global well-posedness
with slowly decaying initial data.
We also prove that the asymptotic profile of the global solution is
given by a solution of the corresponding parabolic problem,
which shows that the solution of the damped wave equation has the diffusion phenomena.
Moreover, we show blow-up of solution
and give the estimate of the lifespan
for a subcritical nonlinearity.
In particular, we determine the critical exponent for any space dimension.
\end{abstract}

\maketitle



%

%
%
%




\tableofcontents

\section{Introduction}
In this paper, we study the Cauchy problem for the nonlinear damped wave equation
\begin{align}
\label{nldw}
	\left\{ \begin{array}{ll}
	\displaystyle
	\partial_t^2 u - \Delta u + \partial_t u = \mathcal{N}(u),
	&(t,x) \in [0,\infty)\times \mathbb{R}^n,\\
	\displaystyle
	u(0,x)=\varepsilon u_0(x),\ \partial_tu(0,x)=\varepsilon u_1(x),
	&x\in \mathbb{R}^n,
	\end{array} \right.
\end{align}
where
$n\in \mathbb{N}$
and
$u$ is a real-valued unknown function,
$\mathcal{N}(u)$
is a power type nonlinearity,
$(u_0, u_1)$
are given data,
and
$\varepsilon > 0$ is a positive parameter,
which describes the amplitude of the initial data.

Our purpose is to establish
the large data local well-posedness
and
the small data global well-posedness
for the Cauchy problem \eqref{nldw}
with slowly decaying initial data, that is,
we treat the initial data not belonging to $L^1(\mathbb{R}^n)$ in general.
Moreover, we investigate the asymptotic behavior of
the global solution and
the estimate of the lifespan from both above and below for subcritical nonlinearities.

The equation \eqref{nldw} is firstly derived by
Oliver Heaviside as the telegrapher's equation,
which describes the current and voltage in an electrical circuit with resistance and inductance.
Cattaneo \cite{Ca58} also introduced the equation \eqref{nldw} as
a modified heat conduction equation which equips
the finite propagation speed property.
The equation \eqref{nldw} also has several background related to
biology and stochastic models
such as genetics, population dynamics \cite{DuOh86, Ha96}
and correlated random walk \cite{Go51, Ka74}.

The local and global well-posedness,
asymptotic behavior of global-in-time solutions and
blow-up of local-in-time solutions have been widely studied for a long time.
Since a pioneer work by Matsumura \cite{Ma76},
it has been well known that solutions of the damped wave equation
behaves like that of the heat equation as time tends to infinity.
Namely, he established $L^p$-$L^q$ estimates of the linear damped wave equation
(Eq. \eqref{nldw} with $\mathcal{N}(u) = 0$),
whose decay rates are the same as those of the linear heat equation
$v_t - \Delta v = 0$
(see also Racke \cite{Ra90} for more general setting).
After that, the so-called diffusion phenomena was found by
Hsiao and Liu \cite{HsLi92} for hyperbolic conservation laws with damping.
Namely, they showed that the asymptotic profile of the solution is given by
the heat kernel
(see also \cite{Ni96, Ni97, Li97, YaMi00}).
Later on, Nishihara \cite{Ni03MathZ}, Marcati and Nishihara \cite{MaNi03},
Hosono and Ogawa \cite{HoOg04}, and Narazaki \cite{Na04}
derived more precise
$L^p$--$L^q$ estimates
for the linear damped wave equation
and applied them to
semilinear equations to obtain global solutions.
Also, the diffusion phenomena for abstract damped wave equations
were studied by
\cite{Ik02, IkNi03, ChHa03, RaToYo11, RaToYo16, Nis}.

For the nonlinear damped wave equation with
the absorbing nonlinearity $\mathcal{N}(u) = - |u|^{p-1}u$,
Kawashima, Nakao and Ono \cite{KaNaOn95}
refined Matsumura's $L^p$--$L^q$ estimates and
applied them to the global well-posedness for the equation \eqref{nldw}
with arbitrarily initial data $(u_0, u_1)\in H^1\times L^2$.
Based on this result, Karch \cite{Kar00} showed the diffusion phenomena
when $p > 1+\frac{4}{n}$ and $n\le 3$.
After that,
Hayashi, Kaikina and Naumkin \cite{HaKaNa07},
Ikehata, Nishihara and Zhao \cite{IkNiZh06}
and Nishihara \cite{Ni06} treated the case
$p > 1+\frac{2}{n}$ and $n\le 4$ if the initial data belongs to
$(H^1\cap L^1) \times (L^2 \cap L^1)$.
Also, Hayashi, Kaikina and Naumkin
\cite{HaKaNa04JDE, HaKaNa06, HaKaNa07, HaKaNa07JMAA},
Hayashi and Naumkin \cite{HaNa}
and Hamza \cite{Ham10}
studied the asymptotic profile of solutions for
critical and subcritical nonlinearities $1<p\le 1+\frac{2}{n}$.

The nonlinear damped wave equation with the source term
$\mathcal{N}(u) = |u|^p$ or $|u|^{p-1}u$
has been widely studied.
In this case, Levine \cite{Le75} showed that the solution in general blows up in finite time for
large initial data.
Therefore, to obtain the global existence of solutions,
we need some smallness condition for the initial data.
Nakao and Ono \cite{NakOn93} studied the case
$\mathcal{N}(u) = |u|^{p-1}u$ with $p \ge 1+\frac{4}{n}$
and proved the global existence of solutions
by the method of modified potential well.
Li and Zhou \cite{LiZh95} found that when $n \le 2$, the critical exponent of \eqref{nldw}
is given by $p=1+\frac{2}{n}$, that is, the local-in-time solution
can be extended time-globally if $p>1+\frac{2}{n}$ and the initial data is
sufficiently small, while the finite time blow-up occurs
if $p\le 1+\frac{2}{n}$ and the initial data has positive integral value.
The number $1+\frac{2}{n}$ is well known as Fujita's critical exponent
named after his seminal work \cite{Fu66},
which is the threshold between the global existence and the blow-up
of solutions to the semilinear heat equation.
Also, in \cite{LiZh95}, the optimal upper estimate of the lifespan for blow-up solutions
was also given
(see also \cite{Ni03Ib} for the case $n=3$, \cite{IkeWa15} for $n\ge 4$ and $p<1+\frac{2}{n}$
and the first author and Ogawa \cite{IkeOg} for $n \ge 4$, $p=1+\frac{2}{n}$).
Later on, Todorova and Yordanov \cite{ToYo01} and Zhang \cite{Zh01}
determined the critical exponent as $p=1+\frac{2}{n}$
for all space dimensions.
Moreover, Ono \cite{On03, On06} derived $L^m$-decay of solutions
for $1\le m\le 2n/(n-2)_+$.
The results of \cite{LiZh95} and \cite{ToYo01} require
that the initial data belongs to $H^1\times L^2$ and has the compact support.
Ikehata, Miyaoka and Nakatake \cite{IkMiNa04},
Ikehata and Tanizawa \cite{IkTa05}
and Hayashi, Kaikina and Naumkin \cite{HaKaNa04}
removed the compactness assumption and proved the global existence
of solutions for the initial data belonging to $L^1$.
Moreover,
Nakao and Ono \cite{NakOn93},
Ikehata and Ohta \cite{IkOh02} and Narazaki and Nishihara \cite{NaNi08}
studied the global well-posedness for slowly decaying initial data
not belonging to $L^1$.
In particular, in \cite{IkOh02}, small data global existence is proved
when the nonlinearity is $\mathcal{N}(u) = |u|^{p-1}u$ with
$p> 1+\frac{2r}{n}$
for $n \le 6$ and
$(H^1\cap L^r)\times (L^2 \cap L^r)$-data,
where
$r$ satisfies
$r\in [1,2]$ if $n=1,2$ and $r\in [ \frac{\sqrt{n^2+16n}-n}{4}, \min\{2,\frac{n}{n-2}\} ]$ if
$3\le n \le 6$.
Finite time blow-up of local solutions was also obtained for any
$n \ge 1$ and $1<p<\frac{2r}{n}$.
However, the above global well-posedness results are restricted to $n\le 6$ and
there are no results for higher dimensional cases.
Also, Narazaki \cite{Na11} considered the slowly decaying data belonging to
modulation spaces and proved the global existence when
the nonlinearity has integer power.

Concerning the asymptotic profile of global solutions,
Gallay and Raugel \cite{GaRa98} determined the asymptotic expansion
up to the second order when $n=1$ and the initial data belongs to
the weighted Sobolev space
$H^{1,1} \times H^{0,1}$
(see Section 1.2 for the definition).
Using the expansion of solutions to the heat equation,
Kawakami and Ueda \cite{KaUe13} extended it to the case $n\le 3$.
Hayashi, Kaikina and Naumkin \cite{HaKaNa04} obtained the first order
asymptotics for all $n\ge 1$
and the initial data belonging to
$(H^{s,0}\cap H^{0,\alpha}) \times (H^{s-1,0}\cap H^{0,\alpha})$
with $\alpha > \frac{n}{2}$ (particularly, belonging to $L^1$).
Recently, Takeda \cite{Ta15, Ta16} determined the higher order
asymptotic expansion of global solutions.
Narazaki and Nishihara \cite{NaNi08} studied the case of slowly decaying data
and proved that if $n\le 3$ and the data behaves like
$(1+|x|)^{-kn}$ with $0<k \le 1$,
then, the asymptotic profile of the global solution is given by
$G(t,x) \ast (1+|x|)^{-kn}$, where $G$ is the Gaussian and $\ast$ denotes the convolution
with respect to spatial variables.

Related to the equation \eqref{nldw},
systems of nonlinear damped wave equation were studied and
the critical exponent and the asymptotic behavior of solutions
were investigated
(see \cite{SuWa07, Na09, Ta09, OgTa10, Na11, OgTa11,
Ni12, NiWa14, HaNaTo15,  NiWa15}).

In the present paper, we establish
the large data local well-posedness
and the small data global well-posedness
for the nonlinear damped wave equation \eqref{nldw}
with slowly decaying initial data.
Our global well-posedness results extend those of
\cite{IkOh02, NaNi08}
to all space dimensions, and generalize that of
\cite{HaKaNa04} to slowly decaying initial data.
Moreover, we study the asymptotic profile of the global solution.
This also extends those of \cite{NaNi08} to all space dimensions.
Considering the asymptotic behavior of solutions in weighted norms,
we further extended the result of \cite{HaKaNa04} to
the asymptotics in $L^m$-norm with $m\le 2$.
Finally, we give an almost optimal lifespan estimate from
both above and below.
This is also an extension of \cite{LiZh95, Ni03Ib, IkeWa15},
in which $L^1$-data were treated.

\subsection{Main results}
We say that
$u \in L^{\infty}(0,T;L^2(\mathbb{R}^n))$
is a mild solution of \eqref{nldw} if
$u$
satisfies the integral equation
\[
	u(t)=
	\left( \partial_t+1 \right)\mathcal{D}(t)\varepsilon u_0
		+ \mathcal{D}(t)\varepsilon u_1
		+ \int_0^t \mathcal{D}(t-\tau)\mathcal{N}(u(\tau))\,d\tau.
\]
in
$L^{\infty}(0,T;L^2(\mathbb{R}^n))$,
where
$\mathcal{D}(t)$ is the solution operator of the damped wave equation
defined in \eqref{d} below.

We assume that
there exists $p >1$ such that
the nonlinear term
$\mathcal{N}(u)$
satisfies
$\mathcal{N} \in C^{p_0} (\mathbb{R})$
with some integer
$p_0 \in [0, p]$
and
\begin{align}
\label{nonlin}
	\left\{ \begin{array}{ll}
	\displaystyle \mathcal{N}^{(l)} (0) = 0,\\[5pt]
	 \displaystyle \left| \mathcal{N}^{(l)}(u) - \mathcal{N}^{(l)}(v) \right|
		\lesssim | u-v| (|u| + |v| )^{p-l-1}
	\end{array}
	\quad (l=0,\ldots, p_0). \right.
\end{align}

\begin{theorem}[Local well-posedness for large data]\label{thm_lwp}
Let
$n\in \mathbb{N}$
and let
$s\ge 0$ be
$0\le [s] \le p_0$.
When $n=1$, we also assume that $0 \le s <1$.
Let
$r \in [1,2]$
and
$s, p$
satisfy
\begin{align*}
	\begin{array}{lll}
		&\displaystyle \min\left\{ 1+\frac{r}{2}, 1+\frac{r}{n} \right\} \le p < \infty
		&\displaystyle \mbox{if}\ \ 1\le n \le 2s,\\[7pt]
		&\displaystyle 1+\frac{r}{n}
				\le p \le \min\left\{ 1+ \frac{2}{n-2s}, \frac{2n}{r(n-2s)} \right\},
		&\displaystyle \mbox{if}\ \ 2s < n, n\ge 2,\\[7pt]
		&\displaystyle 1+\frac{r}{2} \le p \le \frac{1}{1-2s},
		&\displaystyle \mbox{if}\ \ 2s < n, n=1.
	\end{array}
\end{align*}
We take an initial data from
\[
	u_0 \in H^{s,0}(\mathbb{R}^n) \cap H^{0,\alpha} (\mathbb{R}^n),\quad
	u_1 \in H^{s-1,0}(\mathbb{R}^n) \cap H^{0,\alpha} (\mathbb{R}^n),
\]
where
$\alpha > n\left( \frac{1}{r}  - \frac{1}{2} \right)$.
Then, for any
$\varepsilon>0$,
there exists
$T=T(\varepsilon) \in (0, \infty]$
such that
the Cauchy problem \eqref{nldw} admits a unique local mild solution
$u\in C([0,T); H^{s,0}(\mathbb{R}^n) \cap H^{0,\alpha}(\mathbb{R}^n))$.
Moreover, if
$T < \infty$,
$u$
satisfies
\[
	\liminf_{t\to T} \| u(t) \|_{H^{s,0}\cap H^{0,\alpha}} = \infty.
\]
\end{theorem}

\begin{theorem}[Global well-posedness for small data]\label{thm_gwp}
In addition to the assumption in Theorem \ref{thm_lwp},
we assume that
$r\in [1,2]$
and
$s, p$
satisfy
\begin{align*}
	\begin{array}{lll}
		&\displaystyle 1+\frac{2r}{n} < p < \infty
		&\displaystyle \mbox{if}\ \ 1\le n \le 2s,\\[7pt]
		&\displaystyle 1+\frac{2r}{n} < p \le \min\left\{ 1+ \frac{2}{n-2s}, \frac{2n}{r(n-2s)} \right\}
		&\displaystyle \mbox{if}\ \ 2s < n, n\ge 2,\\
		&\displaystyle 1+ 2r < p \le \frac{1}{1-2s}
		&\displaystyle \mbox{if}\ \ 2s < n, n=1
	\end{array}
\end{align*}
(when $r \in (1,2]$, we may take $p= 1+\frac{2r}{n}$).
Then, there exists a constant
$\varepsilon_0 =
\varepsilon_0(n,p,r,s,\alpha$,
$\| u_0\|_{H^{s,0}\cap H^{0,\alpha}}$,$\|u_1\|_{H^{s-1,0}\cap H^{0,\alpha}}) > 0$
such that for any
$\varepsilon \in (0,\varepsilon_0]$,
the Cauchy problem \eqref{nldw} admits a unique global mild solution
$u\in C([0,\infty); H^{s,0}(\mathbb{R}^n) \cap H^{0,\alpha}(\mathbb{R}^n))$.
Moreover, the solution $u$ satisfies the decay estimates
\begin{align*}
	\| u(t) \|_{L^2}
		&\lesssim \varepsilon \langle t \rangle^{-\frac{n}{2}\left( \frac{1}{r}-\frac{1}{2}\right)},
		\quad
	\| |\nabla|^s u(t) \|_{L^2}
		\lesssim \varepsilon
		\langle t \rangle^{-\frac{n}{2}\left( \frac{1}{r}-\frac{1}{2}\right)-\frac{s}{2}}, \quad
	\| |\cdot|^{\alpha} u(t) \|_{L^2}
		\lesssim \varepsilon
		\langle t \rangle^{-\frac{n}{2}\left( \frac{1}{r}-\frac{1}{2}\right)+\frac{\alpha}{2}}.
\end{align*}
\end{theorem}


Next, we study the asymptotic behavior of the global solutions.
To state our result, we denote
$G(t,x)= (4\pi t)^{-n/2} \exp\left( -\frac{|x|^2}{4t} \right)$
and
$\mathcal{G}(t)\phi = G(t)\ast \phi$.

\begin{theorem}[Asymptotic behavior of global solutions]\label{thm_asy}
Under the assumption of Theorem \ref{thm_gwp},
we assume that $p>1+\frac{2r}{n}$ if $r\in (1,2]$.
Let $m$ be
\begin{align}
\label{m}
	r\le m \le \frac{2n}{n-2s}\ \left(s<\frac{n}{2} \right),\quad
	r\le m < \infty \ \left( s=\frac{n}{2} \right),\quad
	r\le m \le \infty\ \left( s > \frac{n}{2} \right)
\end{align}
and let
$\varsigma \in (0,1)$ is an arbitrarily small number.
Then, the the global solution $u$ of \eqref{nldw}
constructed in Theorem \ref{thm_gwp}
satisfies the following
asymptotic behavior:
When $r>1$, we have
\begin{align}
\label{or}
	\| u(t) - \varepsilon \mathcal{G}(t)(u_0+u_1) \|_{L^m}
	\lesssim \langle t \rangle^{%
			-\frac{n}{2}\left( \frac{1}{r}-\frac{1}{m} \right)%
		-\min\{ \frac{n}{2}\left(1-\frac{1}{r}\right), \frac12, \frac{n}{2r}(p-1)-1 \} %
		+ \varsigma}%
\end{align}
for $t \ge 1$.
When $r=1$, we have
\begin{align}
\label{o1}
	\| u(t) - \theta G(t) \|_{L^m}
	\lesssim \langle t \rangle^{%
			-\frac{n}{2}\left( 1-\frac{1}{m} \right)%
		-\min\{ \frac{\alpha}{2} - \frac{n}{4}, \frac12, \frac{n}{2}(p-1)-1 \} %
		+ \varsigma}
\end{align}
for $t\ge 1$, where
\begin{align}
\label{theta}
	\theta = \varepsilon \int_{\mathbb{R}^n} (u_0+u_1)(x) dx
	+ \int_0^{\infty} \int_{\mathbb{R}^n} \mathcal{N}(u) dx dt.
\end{align}
\end{theorem}
\begin{remark}
(i) The case $r=1, m\ge 2$
was studied by
Hayashi, Kaikina and Naumkin \cite{HaKaNa04},
though in this paper, we refine their argument
(see Section 3.3 and Section 3.4).

(ii)
When
$r>1$,
as we will see in the proof,
the nonlinear term
$\mathcal{N}(u)$
has better spatial integrability than the linear part.
Hence, the nonlinear term decays faster as time tends to infinity
and does not affect the asymptotic profile.

(iii)
When
$n\le 3$ and
the initial data behaves like
$\langle x \rangle^{-k}$
as $|x| \to \infty$,
a similar asymptotic behavior was obtained by
Narazaki and Nishihara \cite{NaNi08}.
The above theorem generalizes the result of \cite{NaNi08}
to all $n\ge 1$ and more general initial data,
while the class of the solution is slightly different.
\end{remark}

In the critical or subcritical case,
we also have the estimate of the lifespan from below.
We define the lifespan of the solution of \eqref{nldw} by
\begin{align*}
	T(\varepsilon)
	&:= \sup \{ T \in (0,\infty) ; \mbox{there exists a  mild solution}
	\ u \in C([0,T) ; H^{s,0}(\mathbb{R}^n)\cap H^{0,\alpha}(\mathbb{R}^n)) \}.
\end{align*}

\begin{theorem}[Lower bound of the lifespan]\label{thm_lflow}
In addition to the assumption in Theorem \ref{thm_lwp},
we assume that
\begin{align}
\tag{Case 1}
	r\in [1,2],\quad
	\min\left\{ 1+\frac{r}{2}, 1+\frac{r}{n} \right\} \le p < 1+\frac{2r}{n}
\end{align}
or
\begin{align}
\tag{Case 2}
	 r = 1,\quad p= 1+\frac{2r}{n}.
\end{align}
Then, there exists
$\varepsilon_1=
\varepsilon_1(n,p,r,s,\alpha,
\| u_0\|_{H^{s,0}\cap H^{0,\alpha}}, \|u_1\|_{H^{s-1,0}\cap H^{0,\alpha}}) >0$
such that for any
$\varepsilon \in (0,\varepsilon_1]$,
the lifespan
$T=T(\varepsilon)$
of the solution is estimated as
\begin{align*}
	T(\varepsilon) \ge
	\begin{cases}
	C\varepsilon^{-1/\omega} &\mbox{in {\rm Case 1}},\\
	\exp\left( C\varepsilon^{-(p-1)} \right) &\mbox{in {\rm Case 2}},
	\end{cases}
\end{align*}
where
$\omega = \frac{1}{p-1} - \frac{n}{2r}$
and $C>0$ is a positive constant independent of
$\varepsilon \in (0,\varepsilon_1]$.
\end{theorem}

\begin{remark}
The estimate in {\rm Case 2} was proved by
the first author and Ogawa \cite{IkeOg}.
\end{remark}

Finally, we prove a blow-up result
in the subcritical case
with the nonlinearity
$\mathcal{N}(u) = \pm |u|^p$.
\begin{theorem}[Upper bound of the lifespan]\label{thm_lfupp}
In addition to the assumptions in Theorem \ref{thm_lwp},
we assume that
\begin{align}
\label{bunl}
	\mathcal{N}(u) = \pm |u|^p
	\quad \mbox{and}\quad p < 1+\frac{2r}{n}
\end{align}
and
$\alpha$
satisfies
\[
	n\left( \frac{1}{r} - \frac{1}{2} \right) < \alpha < \frac{2}{p-1} - \frac{n}{2}.
\]
Moreover, we take the initial data
\[
	u_0 \in H^{s,0}(\mathbb{R}^n) \cap H^{0,\alpha} (\mathbb{R}^n),\quad
	u_1 \in H^{s-1,0}(\mathbb{R}^n) \cap H^{0,\alpha} (\mathbb{R}^n)
\]
fulfilling
\[
	\pm (u_0(x) + u_1(x)) \ge
	\begin{cases}
	|x|^{-\lambda}&\mbox{if}\quad |x| > 1,\\
	0&\mbox{if}\quad |x|\le 1
	\end{cases}
\]
(double-sign corresponds to \eqref{bunl})
with some
$\lambda$
satisfying
$\frac{n}{2} + \alpha < \lambda < \frac{2}{p-1}$.
Then, there exists
$\varepsilon_2=\varepsilon_2(n,p,r,s,\alpha,\lambda)>0$
such that for any
$\varepsilon \in (0,\varepsilon_2]$,
the lifespan
$T=T(\varepsilon)$
of the solution is estimated as
\[
	T(\varepsilon) \le C \varepsilon^{-1/\kappa},
\]
where
$\kappa = \frac{1}{p-1} - \frac{\lambda}{2}$
and $C$ is a positive constant independent of
$\varepsilon \in (0,\varepsilon_2]$.
\end{theorem}
\begin{remark}
(i) The case $r=1$ and $p\le 1+ \frac{2r}{n}$ was studied
by Todorova and Yordanov \cite{ToYo01}, Zhang \cite{Zh01},
Li and Zhou \cite{LiZh95}, Nishihara \cite{Ni03Ib},
the first author and Ogawa \cite{IkeOg},
and the first author and the third author \cite{IkeWa15}.
In particular, Theorem \ref{thm_lflow} shows that
when $1+\frac{1}{n}\le p<1+\frac{2}{n}$,
the lifespan $T(\varepsilon)$
is estimates as
$T(\varepsilon) \sim \varepsilon^{-1/\omega}$
with $\omega = \frac{1}{p-1}-\frac{n}{2}$.

(ii) When $r>1$, $p=1+\frac{2r}{n}$ and
the initial data belong to $L^r(\mathbb{R}^n)$
but not $H^{0,\alpha}(\mathbb{R}^n)$,
it is still an open problem whether the local solution blows up or not.

(iii)
When $\mathcal{N}(u) = |u|^{p-1}u$,
the blow-up of the solution was proved by
Ikehata and Ohta \cite{IkOh02},
while estimates of lifespan were not obtained.
Theorems \ref{thm_lflow} and \ref{thm_lfupp} give
an almost optimal estimate of lifespan
for $\mathcal{N}(u)=\pm |u|^p$.

\end{remark}

Our results are summarized in Table 1 below.
\begin{table}[!h]
\begingroup
\renewcommand{\arraystretch}{1.8}
\begin{tabular}{|c|c|c|c|} \hline
	$r \backslash p$
		&  $\displaystyle \min\left\{ 1+\frac{r}{2} , 1+\frac{r}{n} \right\}$  $\cdots $
		&  $\displaystyle 1+\frac{2r}{n}$
		&  $\cdots $  $\displaystyle 1+\frac{2}{n-2s}$ \\[4pt]
	\hline \hline
	$1$ & SDBU \cite{ToYo01}& SDBU \cite{Zh01} & SDGE \cite{HaKaNa04}\\
	& $ \varepsilon^{-1/\omega} \lesssim T(\varepsilon) \lesssim \varepsilon^{-1/\kappa}$
		\cite{IkeWa15}
	&  $ e^{C\varepsilon^{-(p-1)}} \lesssim T(\varepsilon) \lesssim e^{C\varepsilon^{-p}}$
		\cite{IkeOg}
	& $T(\varepsilon)=\infty$ \\[5pt]
	\hline
	$r >1$ & SDBU \cite{IkOh02}& SDGE & SDGE \\ 
	& $ \varepsilon^{-1/\omega} \lesssim T(\varepsilon) \lesssim \varepsilon^{-1/\kappa}$
	& $T(\varepsilon)=\infty$ & $T(\varepsilon)=\infty$ \\[5pt]
	\hline
\end{tabular}
\caption{SDBU: Small data blow-up,
SDGE: Small data global existence}
\endgroup
\end{table}

Our strategy for proving Theorems \ref{thm_lwp} and \ref{thm_gwp} are based on
that of Hayashi, Kaikina and Naumkin \cite{HaKaNa04}.
However,
we have to refine their estimates to fit the slowly decaying data and solutions.
The main ingredient is the estimate of the fundamental solution
$\mathcal{D}(t)$ (see \eqref{d} for the definition) of the linear problem,
which are given in Lemma \ref{lem_lin}.
To prove these estimates, we use a gain of one derivative
coming from the high frequency part of the kernel $L(t,\xi)$ of $\mathcal{D}(t)$
(see \eqref{l}).
Combining these estimates and nonlinear estimates with the contraction mapping principle,
we prove the existence of solutions.

To prove Theorem \ref{thm_asy},
we first show that the solution $u$ of the damped wave equation \eqref{nldw}
is approximated by the solution of the linear heat equation
with the homogeneous term $\mathcal{N}(u)$
(see Proposition \ref{prop_asy}).
After that, we investigate the precise asymptotic behavior
of solutions to the inhomogeneous linear heat equation
(see Proposition \ref{prop_asy2}).

For the upper estimates of the lifespan, we employ
a test function method developed by Zhang \cite{Zh01},
while this is based on a contradiction argument and
not directly applicable to obtain the lifespan estimate.
To avoid the contradiction argument,
we use the ideas by Kuiper \cite{Ku03}, Sun \cite{Su10} and \cite{IkeWa15}
to obtain an almost optimal estimate of the lifespan.

The rest of the paper is organized as follows.
In Section 2, we give a proof of Theorems \ref{thm_lwp} and \ref{thm_gwp}.
Section 3 is devoted to the proof of Theorem \ref{thm_asy}.
Theorems \ref{thm_lflow} and \ref{thm_lfupp} will be proved in Section 4.
Finally, we collect some useful lemmas in Appendix.


\subsection{Notations}\label{se_no}
For the reader's convenience,
we collect the notations used throughout this paper.
The letter $C$ indicates a generic constant, which may change from line to line.
Let $\langle x \rangle = (1+|x|^2)^{1/2}$.
We also use the symbol $f\lesssim g$, which means that $f\le C g$ holds with some constant $C>0$.
The relation $f \sim g$ stands for $f\lesssim g$ and $g\lesssim f$.
Sometimes we use $a\lor b := \max\{ a, b \}$.

For functions $f=f(x):\mathbb{R}^n \to \mathbb{R}$
and $\phi = \phi (\xi) : \mathbb{R}^n \to \mathbb{R}$,
the Fourier transform and the inverse Fourier transform
are defined by
\[
	\mathcal{F}[f](\xi) = \hat{f}(\xi) = (2\pi)^{-n/2} \int_{\mathbb{R}^n} f(x) e^{-ix\xi} dx,\ 
	\mathcal{F}^{-1}[\phi](x) =
	(2\pi)^{-n/2} \int_{\mathbb{R}^n} \phi(\xi) e^{ix\xi} d\xi,
\]
respectively.

Let
$L^p(\mathbb{R}^n)\ (1\le p \le \infty)$
and
$H^{s,\alpha}(\mathbb{R}^n) \ (s, \alpha \ge 0)$
be the usual Lebesgue and the weighted Sobolev spaces, respectively,
equipped with the norms defined by
\begin{align*}
	&\| f \|_{L^p} = \left( \int_{\mathbb{R}^n} |f(x)|^p dx \right)^{1/p}\ (1\le p <\infty),\quad
	\| f\|_{L^{\infty}} = {\rm ess\,sup\,}|f|.\\
	&\| f \|_{H^{s,\alpha}}
	= \| \langle x \rangle^{\alpha} \langle \nabla \rangle^{s} f \|_{L^2},
\end{align*}
where
$\langle \nabla \rangle^{s} f = \mathcal{F}^{-1}[ \langle \xi \rangle^s \hat{f} ]$.
For an interval $I$ and a Banach space $X$,
we define $C^r(I;X)$ as the space of $r$-times continuously differentiable mapping from
$I$ to $X$ with respect to the topology in $X$.

We denote by $\mathcal{D}(t)$ and $\mathcal{G}(t)$
the solution operator of the linear damped wave and linear heat equations, respectively,
that is,
\begin{align}
\label{d}
	\mathcal{D}(t) &:= e^{-\frac{t}{2}} \mathcal{F}^{-1} L(t,\xi) \mathcal{F},\\
\label{g}
	\mathcal{G}(t) &:= \mathcal{F}^{-1} e^{-t|\xi|^2} \mathcal{F},
\end{align}
where
\begin{align}
\label{l}
	L(t,\xi) := \begin{cases}
	\displaystyle \frac{\sinh (t \sqrt{1/4-|\xi|^2})}{\sqrt{1/4-|\xi|^2}}
		&\mbox{if}\ |\xi| < 1/2,\\
	\displaystyle \frac{\sin (t \sqrt{|\xi|^2-1/4})}{\sqrt{|\xi|^2-1/4}}
		&\mbox{if}\ |\xi| > 1/2.
	\end{cases}
\end{align}
Also, we use
\begin{align}
\label{dt}
	\tilde{\mathcal{D}}(t) := \left( \partial_t + 1 \right) \mathcal{D}(t).
\end{align}

Throughout this paper, we always use
$s, r, \alpha$ as real numbers satisfying
$s\ge 0$, $r \in [1,2]$, $\alpha > n\left( \frac{1}{r} - \frac{1}{2} \right)$,
respectively.
Also, for a real number $s$, we denote by $[s]$ the integer part of $s$.
For $T\in (0,\infty]$, we define
\begin{align}
\nonumber
	\| \phi \|_{X(T)}
	& :=
	\sup_{0<t<T} \Big[
		\langle t \rangle^{\frac{n}{2}\left( \frac{1}{r} -\frac{1}{2} \right)} \| \phi (t) \|_{L^2}
		+ \langle t \rangle^{\frac{n}{2}\left( \frac{1}{r} -\frac{1}{2} \right)+\frac{s}{2}}
			\| |\nabla |^s \phi (t) \|_{L^2} \\
\label{xnorm}
	&\qquad \qquad \quad
		+ \langle t \rangle^{\frac{n}{2}\left( \frac{1}{r} -\frac{1}{2} \right)-\frac{\alpha}{2}}
			\| |\cdot|^{\alpha} \phi (t) \|_{L^2}
		\Big],\\
\nonumber
	\| \psi \|_{Y(T)}
	&:=
	\sup_{0<t<T} \Big[
		\langle t \rangle^{\eta} \| |\nabla|^{[s]} \psi (t) \|_{L^{\rho}}
			+ \sup_{\gamma\in [\sigma_1, \sigma_2]}
			\langle t \rangle^{\frac{n}{2}\left( \frac{p}{r}-\frac{1}{\gamma}\right) }
			\| \psi(t) \|_{L^{\gamma}} \\
\label{ynorm}
	&\qquad \qquad \quad
		+ \langle t \rangle^{\zeta} \| \langle \cdot \rangle^{\alpha} \psi(t) \|_{L^q}
		\Big],
\end{align}
where the parameters are defined in Table 2.
Here $(n-2s)_+$ denotes $0 \lor (n-2s)$.
\begin{table}[!h]
\renewcommand{\arraystretch}{2}
\begin{center}
\begin{tabular}{|c|c|c|}
	\hline
	\ &$n\ge 2$&$n=1$\\
	\hline
	$\eta:=$
	&$\displaystyle \frac{\mu}{2}(p-1)+\frac{s}{2}+\frac{np}{2}\left(\frac{1}{r}-\frac{1}{2}\right)$
	&$\displaystyle \frac{1}{2}\left(\frac{p}{r}-\frac{1}{2}\right)$\\[5pt]
	\hline
	$\mu:=$
	&$\displaystyle \frac{n}{2}-\frac{1}{p-1}$
	&-- \\[5pt]
	\hline
	$\zeta:=$
	&$\displaystyle \frac{n}{2r}( p-1) - \frac{1}{2}
		+ \frac{n}{2}\left( \frac{1}{r} - \frac{1}{2} \right)-\frac{\alpha}{2}$
	&$\displaystyle \frac{1}{2r}( p-1) - \frac{1}{4}
		+ \frac{1}{2}\left( \frac{1}{r} - \frac{1}{2} \right)-\frac{\alpha}{2}$\\[5pt]
	\hline
	$q:=$
	&$\displaystyle \frac{2n}{n+2}$
	&$1$\\[5pt]
	\hline
	$\rho :=$
	&$\displaystyle \frac{2n}{n+2-2(s-[s])}$
	&$2$\\[5pt]
	\hline
	$\sigma_1:=$
	&$\displaystyle \max\left\{ 1, \frac{nr}{n+r} \right\}$
	&$1$\\[5pt]
	\hline
	$\sigma_2 :=$
	&$\displaystyle \min \left\{ 2, \frac{2n}{p(n-2s)_+} \right\}$
	&$2$\\[5pt]
	\hline
\end{tabular}
\caption{Definition of parameters}
\end{center}
\end{table}

\vspace{1em}



\section{Local and global existence of solutions}
\subsection{Preliminary estimates}
\begin{lemma}\label{lem_lin}
Let $\gamma, \nu \in [1,2]$, $\beta \ge 0$, $s_1 \ge s_2 \ge 0$.
Then, we have
\begin{align}
\label{est_der}
	\| |\nabla|^{s_1} \mathcal{D}(t) \psi \|_{L^2}
	&\lesssim
	\langle t \rangle^{-\frac{n}{2}\left( \frac{1}{\gamma}-\frac{1}{2}\right)-\frac{s_1-s_2}{2}}
		\| |\nabla|^{s_2} \psi \|_{L^{\gamma}}
		+ e^{-\frac{t}{4}} \| |\nabla|^{s_1} \langle \nabla \rangle^{-1} \psi \|_{L^2},\\
\label{est_wei}
	\| | \cdot |^{\beta} \mathcal{D}(t) \psi \|_{L^2}
	&\lesssim
	\langle t \rangle^{-\frac{n}{2}\left( \frac{1}{\gamma}-\frac{1}{2} \right)+\frac{\beta}{2}}
		\| \psi \|_{L^{\gamma}}
	+\langle t \rangle^{-\frac{n}{2}\left( \frac{1}{\nu}-\frac{1}{2}\right)}
		\| |\cdot|^{\beta} \psi \|_{L^{\nu}}
		+ e^{-\frac{t}{4}} \| \langle \cdot \rangle^{\beta} \langle \nabla \rangle^{-1} \psi \|_{L^2}.
\end{align}
We also have the continuity of $\mathcal{D}(t)$ with respect to
$H^{s_1,0} \cap H^{0,\beta}$-norm:
\begin{align}
\label{con1}
	\lim_{t_1\to t_2} \| \mathcal{D}(t_1) \psi - \mathcal{D}(t_2)\psi \|_{H^{s_1,0}\cap H^{0,\beta}}
	=0.
\end{align}
\end{lemma}

When $\gamma=1$, the estimates \eqref{est_der} and \eqref{est_wei} were proved
by \cite{HaKaNa04}.
Here we give a generalization of it to $\gamma \in [1,2]$
with a slightly simpler proof.
\begin{proof}
By the definition of $L(t,\xi)$, it is easy to see that
\[
	e^{-\frac{t}{2}} \| |\xi|^{s} L(t,\xi) \|_{L^{\gamma}(|\xi| \le 1)}
	\lesssim \langle t \rangle^{-\frac{s}{2}-\frac{n}{2\gamma}}
\]
and
\[
	\| \langle \xi \rangle L(t,\xi) \|_{L^{\infty}(|\xi|\ge 1)} \lesssim 1.
\]
Therefore, applying the Plancherel theorem and
the H\"{o}lder inequality, we have
\begin{align*}
	\| |\nabla|^{s_1} \mathcal{D}(t) \psi \|_{L^2}
	&= \| |\xi|^{s_1} e^{-\frac{t}{2}} L(t,\xi) \hat{\psi} \|_{L^2} \\
	&\lesssim e^{-\frac{t}{2}}
		\| |\xi|^{s_1-s_2} L(t,\xi) \|_{L^{\frac{2\gamma}{2-\gamma}}(|\xi|\le 1)}
		\| |\xi|^{s_2} \hat{\psi} \|_{L^{\frac{\gamma}{\gamma-1}}(|\xi| \le 1)} \\
	&\quad + e^{-\frac{t}{2}}
		\| \langle \xi \rangle L(t,\xi) \|_{L^{\infty}(|\xi|\ge 1)}
		\| |\xi|^{s_1} \langle \xi \rangle^{-1} \hat{\psi} \|_{L^2(|\xi|\ge 1)} \\
	&\lesssim
		\langle t \rangle^{-\frac{s_1-s_2}{2} - \frac{n}{2} \left( \frac{1}{\gamma}-\frac{1}{2} \right)}
		\| |\nabla|^{s_2} \psi \|_{L^{\gamma}}
		+ e^{-\frac{t}{4}}
		\| |\nabla|^{s_1} \langle \nabla \rangle^{-1} \psi \|_{L^2},
\end{align*}
which gives \eqref{est_der}.

Next, we prove \eqref{est_wei}.
Let
$\chi(\xi) \in C_0^{\infty}(\mathbb{R}^n)$
be a cut-off function satisfying
$\chi(\xi) = 1$ for $|\xi|\le 1$
and $\chi(\xi) = 0$ for $|\xi|\ge 2$.
We put
\[
	K(t,x) := e^{-\frac{t}{2}} \mathcal{F}^{-1} \left( L(t,\xi) \chi(\xi) \right).
\]
The Hausdorff-Young inequality implies
\begin{align}
\nonumber
	&\| |\cdot|^{\beta} \mathcal{F}^{-1}
		\left( e^{-\frac{t}{2}} L(t,\xi) \chi(\xi) \hat{\psi} \right) \|_{L^2} \\
\nonumber
	&\quad = \left\| |x|^{\beta} \int_{\mathbb{R}^n} K(t,x-y) \psi(y) dy \right\|_{L^2} \\
\nonumber
	&\quad \lesssim
		\left\|
		\int_{\mathbb{R}^n} \left( |x-y|^{\beta} | K(t,x-y) | + |K(t,x-y)| |y|^{\beta} \right)
		\psi (y) dy \right\|_{L^2} \\
\label{ine_K}
	&\quad \lesssim
		\| |\cdot|^{\beta} K(t) \|_{L^{\frac{2\gamma}{3\gamma-2}}}
		\| \psi \|_{L^{\gamma}}
		+ \| K(t) \|_{L^{\frac{2\nu}{3\nu-2}}} \| |\cdot|^{\beta} \psi \|_{L^{\nu}}.
\end{align}
Now we prove
\begin{align}
\label{est_K}
	\| |\cdot|^{\beta} K(t) \|_{L^k}
	\lesssim \langle t \rangle^{\frac{\beta}{2}-\frac{n}{2}\left( 1-\frac{1}{k} \right)}
\end{align}
for $t>0$ and $k \in [1,\infty)$.
First, we divide the proof into the cases $0<t<1$ and $t\ge 1$.
For the case $0<t<1$, we easily prove
$\| |\cdot|^{\beta} K(t) \|_{L^k} \lesssim 1$.
Indeed, noting
$|x|^{\beta k} \le \langle x \rangle^{2Nk-(n+1)}$
for sufficiently large integer $N\in \mathbb{N}$,
we have
\begin{align*}
	\| |\cdot|^{\beta} K(t) \|_{L^k}^k
	&\lesssim
		\int_{\mathbb{R}^n} \langle x \rangle^{2Nk-(n+1)} | K(t,x) |^k dx \\
	&= \int_{\mathbb{R}^n} \langle x \rangle^{-(n+1)}
		\left| \int_{|\xi|\le 1/4}
			e^{ix\xi} e^{-\frac{t}{2}} \langle \nabla \rangle^{2N} ( L(t,\xi) \chi(\xi) ) d\xi
		\right|^k dx \\
	&\lesssim 1.
\end{align*}
For the case $t \ge 1$, we change the variables as
$\sqrt{t} \xi = \eta$ and $x = \sqrt{t} y$ to obtain
\begin{align*}
	\| |\cdot|^{\beta} K(t) \|_{L^k}^k
	&= \int_{\mathbb{R}^n} |x|^{\beta k}
		\left|
			\int_{\mathbb{R}^n} e^{ix\xi} e^{-\frac{t}{2}} L(t,\xi) \chi(\xi) d\xi
		\right|^k dx \\
	&= t^{-\frac{n}{2}k} \int_{\mathbb{R}^n} |x|^{\beta k}
		\left|
			\int_{\mathbb{R}^n} e^{ix\eta/\sqrt{t}}
			e^{-\frac{t}{2}} L(t,\eta/\sqrt{t}) \chi(\eta/\sqrt{t}) d\eta
		\right|^k dx \\
	&= t^{\frac{\beta}{2}k-\frac{n}{2}(k-1)} \int_{\mathbb{R}^n} |y|^{\beta k}
		\left|
			\int_{\mathbb{R}^n} e^{iy\eta}
			e^{-\frac{t}{2}} L(t,\eta/\sqrt{t}) \chi(\eta/\sqrt{t}) d\eta
		\right|^k dy.
\end{align*}
As before, using
$|y|^{\beta k} \le \langle y \rangle^{2Nk-(n+1)}$
with sufficiently large integer
$N \in \mathbb{N}$, we have
\begin{align*}
	&\| |\cdot|^{\beta} K(t) \|_{L^k}^k \\
	&\quad \lesssim
		t^{\frac{\beta}{2}k-\frac{n}{2}(k-1)}
			\int_{\mathbb{R}^n} \langle y \rangle^{-(n+1)}
		\left|
			\int_{\mathbb{R}^n} e^{iy\eta}
			\langle \nabla_{\eta} \rangle^{2N} \left(
			e^{-\frac{t}{2}} L(t,\eta/\sqrt{t}) \chi(\eta/\sqrt{t})
			\right) d\eta
		\right|^k dy. \\
\end{align*}
By the definition of $L(t,\xi)$ (see \eqref{l}),
we can easily see that for $t\ge 1$
\[
	\left| \langle \nabla_{\eta} \rangle^{2N} \left(
			e^{-\frac{t}{2}} L(t,\eta/\sqrt{t}) \chi(\eta/\sqrt{t})
			\right) \right|
	\lesssim e^{-\frac{|\eta|^2}{2}}.
\]
Thus, we obtain the desired estimate \eqref{est_K}.

Applying \eqref{est_K} to \eqref{ine_K}, we have
\[
	\left\| |\cdot|^{\beta} \mathcal{F}^{-1}
		\left( e^{-\frac{t}{2}} L(t,\xi) \chi(\xi) \hat{\psi} \right) \right\|_{L^2}
	\lesssim
	\langle t \rangle^{\frac{\beta}{2}-\frac{n}{2}\left( \frac{1}{\gamma}-\frac{1}{2} \right)}
		\| \psi \|_{L^{\gamma}}
	+\langle t \rangle^{-\frac{n}{2}\left( \frac{1}{\nu}-\frac{1}{2}\right)}
		\| |\cdot|^{\beta} \psi \|_{L^{\nu}}.
\]
Hence, it suffices to show that
\begin{align}
\label{est_lc2}
	\left\| |\cdot|^{\beta} \mathcal{F}^{-1}
		\left( e^{-\frac{t}{2}} L(t,\xi) (1-\chi(\xi)) \hat{\psi} \right) \right\|_{L^2}
	\lesssim
	e^{-\frac{t}{4}}
		\| \langle \cdot \rangle^{\beta} \langle \nabla \rangle^{-1} \psi \|_{L^2}.
\end{align}
To prove this, we calculate the fractional derivative.
Let
$\omega := \beta - [\beta]$.
Making use of Lemma \ref{lem_fd}, we see that
\begin{align*}
	| \partial_j|^{\omega} ( \phi \psi )(x)
	&= C \int_{\mathbb{R}} ( \phi(x) \psi(x) - \phi(x+y)\psi(x+y) ) \frac{dy_j}{|y_j|^{1+\omega}} \\
	&= \phi(x) |\partial_j|^{\omega} \psi(x)
		+ C\int_{\mathbb{R}} ( \phi(x) - \phi(x+y)) \psi(x+y) \frac{dy_j}{|y_j|^{1+\omega}}
\end{align*}
holds, where
$y := (0, \ldots, y_j, \ldots, 0)$.
Then, using the Plancherel theorem and the Leibniz rule, we have
\begin{align*}
	\| |\nabla|^{\beta} ( \phi \psi ) \|_{L^2}
	&\lesssim \sum_{j=1}^n \| |\partial_j|^{\omega} |\partial_j|^{[\beta]} ( \phi \psi )\|_{L^2}
	\lesssim \sum_{j=1}^n \sum_{k=0}^{[\beta]}
		\left\| |\partial_j|^{\omega} \left( (\partial_j^{[\beta]-k} \phi ) \partial_j^k \psi \right) \right\|_{L^2},
\end{align*}
and hence,
\begin{align*}
	\| |\nabla|^{\beta} (\phi \psi ) \|_{L^2}
	&\lesssim
	\sum_{j=1}^n \sum_{k=0}^{[\beta]}
		\left\| \left( \partial_j^{m-k} \phi \right) |\partial_j|^{k+\omega} \psi \right\|_{L^2} \\
	&\quad + \sum_{j=1}^n \sum_{k=0}^{[\beta]}
		\left\| \int_{\mathbb{R}} \left(
			\partial_j^{[\beta]-k}\phi (\cdot) - \partial_j^{[\beta]-k} \phi(\cdot + y) \right)
			\partial_j^k \psi (\cdot + y) \frac{dy_j}{|y_j|^{1+\omega}}
		\right\|_{L^2} \\
	&\lesssim 
		\left(
		\sum_{j=1}^n \sum_{k=0}^{[\beta]+1}
		\| \partial_j^k \phi \|_{L^{\infty}} \right)
		\| \langle \nabla \rangle^{\beta} \psi \|_{L^2}.
\end{align*}
We apply the above estimate to the left-hand side of \eqref{est_lc2} with
$\psi = \langle \xi \rangle^{-1} \hat{\psi}$
and
$\phi = \langle \xi \rangle e^{-\frac{t}{2}} L(t,\xi) (1-\chi(\xi))$.
Since
\[
	\| \partial_j^{k} e^{-\frac{t}{2}} \langle \xi \rangle L(t,\xi) (1-\chi(\xi)) \|_{L^{\infty}}
	\lesssim e^{-\frac{t}{4}}
\]
for any
$k \in \mathbb{N}$,
we obtain
\begin{align*}
	\left\| |\cdot|^{\beta} \mathcal{F}^{-1}
		\left( e^{-\frac{t}{2}} L(t,\xi) (1-\chi(\xi)) \hat{\psi} \right) \right\|_{L^2}
	\lesssim e^{-\frac{t}{4}}
		\| \langle \nabla \rangle^{\beta} \langle \xi \rangle^{-1} \hat{\psi} \|_{L^2}.
\end{align*}

Finally, the bounds of $L(t,\xi)$ proved above and
the continuity of $L(t,\xi)$ with respect to $t$ show \eqref{con1}.
This completes the proof.
\end{proof}

Similarly, we can prove the following lemma.
\begin{lemma}\label{lem_lin2}
Let $\gamma, \nu \in [1,2]$, $\beta \ge 0$, $s_1 \ge s_2 \ge 0$.
Then, we have
\begin{align}
\label{est_der2}
	\| |\nabla|^{s_1} \tilde{\mathcal{D}}(t) \psi \|_{L^2}
	&\le
	\langle t \rangle^{-\frac{n}{2}\left( \frac{1}{\gamma}-\frac{1}{2}\right)-\frac{s_1-s_2}{2}}
		\| |\nabla|^{s_2} \psi \|_{L^{\gamma}}
		+ e^{-\frac{t}{4}} \| |\nabla|^{s_1} \psi \|_{L^2},\\
\label{est_wei2}
	\| | \cdot |^{\beta} \tilde{\mathcal{D}}(t) \psi \|_{L^2}
	&\le
	\langle t \rangle^{-\frac{n}{2}\left( \frac{1}{\gamma}-\frac{1}{2} \right)+\frac{\beta}{2}}
		\| \psi \|_{L^{\gamma}}
	+\langle t \rangle^{-\frac{n}{2}\left( \frac{1}{\nu}-\frac{1}{2}\right)}
		\| |\cdot|^{\beta} \psi \|_{L^{\nu}}
	+ e^{-\frac{t}{4}} \| \langle \cdot \rangle^{\beta} \psi \|_{L^2}.
\end{align}
We also have the continuity of $\tilde{\mathcal{D}}(t)$ with respect to
$H^{s_1,0} \cap H^{0,\beta}$-norm:
\begin{align}
\label{con2}
	\lim_{t_1\to t_2}
	\| \tilde{\mathcal{D}}(t_1) \psi - \tilde{\mathcal{D}}(t_2)\psi \|_{H^{s_1,0}\cap H^{0,\beta}}
	=0.
\end{align}
\end{lemma}

Let $\| \cdot \|_{X(T)}$ be defined by \eqref{xnorm}.
In order to prove Theorems \ref{thm_lwp} and \ref{thm_gwp},
we frequently use the following interpolation inequalities.

\begin{lemma}\label{lem_ip}
Let
$s\ge 0$, $r\in [1,2]$, $\alpha > n ( \frac{1}{r}-\frac12 )$
and let
$\| \cdot \|_{X(T)}$
be defined by \eqref{xnorm}.
Then, the following interpolation inequalities hold:

\noindent
(i) We have
\begin{align*}
	&\sup_{0<t<T}\left(
	\sup_{0\le s^{\prime} \le s}
		\langle t \rangle^{\frac{n}{2}\left(\frac{1}{r}-\frac{1}{2}\right)+\frac{s^{\prime}}{2}}
		\| |\nabla|^{s^{\prime}} \phi(t) \|_{L^2}
		+
	\sup_{0\le \beta\le \alpha}
		\langle t \rangle^{\frac{n}{2}\left(\frac{1}{r}-\frac{1}{2}\right)-\frac{\beta}{2}}
		\| | \cdot |^{\beta} \phi(t) \|_{L^2}
	\right)
	\lesssim
	\| \phi \|_{X(T)}.
\end{align*}

\noindent
(ii)
We have
\begin{align}
\label{eq_ip}
	\sup_{0<t<T} 
	\langle t \rangle^{\frac{n}{2}\left( \frac{1}{r}-\frac{1}{\gamma} \right)}
	\| \phi(t) \|_{L^{\gamma}}
	\lesssim
	\| \phi \|_{X(T)}
	\quad
	\mbox{for}\quad
	\begin{cases}
		r\le \gamma \le \infty&\mbox{if}\quad 1\le n < 2s,\\
		r\le \gamma < \infty&\mbox{if}\quad n =2s,\\
		r\le \gamma \le \frac{2n}{n-2s}&\mbox{if}\quad 2s<n.
	\end{cases}
\end{align}
\end{lemma}
\begin{proof}
(i)
The Plancherel theorem and the H\"{o}lder inequality imply
\begin{align*}
	\langle t \rangle^{\frac{n}{2}\left(\frac{1}{r}-\frac{1}{2}\right)+\frac{s^{\prime}}{2}}
		\| |\nabla|^{s^{\prime}} \phi(t) \|_{L^2}
	&= \langle t \rangle^{\frac{n}{2}\left(\frac{1}{r}-\frac{1}{2}\right)+\frac{s^{\prime}}{2}}
		\| |\xi|^{s^{\prime}} \hat{\phi}(t) \|_{L^2} \\
	&\le
		\left( \langle t \rangle^{\frac{n}{2}\left(\frac{1}{r}-\frac{1}{2}\right)}
		\|  \hat{\phi}(t) \|_{L^2} \right)^{1-\frac{s^{\prime}}{s}}
		\left( \langle t \rangle^{\frac{n}{2}\left(\frac{1}{r}-\frac{1}{2}\right)+\frac{s}{2}}
		\| |\xi|^{s} \hat{\phi}(t) \|_{L^2} \right)^{\frac{s^{\prime}}{s}} \\
	&\le \| \phi \|_{X(T)}.
\end{align*}
In the same way, we have
\begin{align*}
	\langle t \rangle^{\frac{n}{2}\left(\frac{1}{r}-\frac{1}{2}\right)-\frac{\beta}{2}}
		\| | \cdot |^{\beta} \phi(t) \|_{L^2}
	&\le
	\left( \langle t \rangle^{\frac{n}{2}\left(\frac{1}{r}-\frac{1}{2}\right)}
		\| \phi(t) \|_{L^2} \right)^{1-\frac{\beta}{\alpha}}
	\left( \langle t \rangle^{\frac{n}{2}\left(\frac{1}{r}-\frac{1}{2}\right)-\frac{\alpha}{2}}
		\| | \cdot |^{\alpha} \phi(t) \|_{L^2} \right)^{\frac{\beta}{\alpha}} \\
	&\le \| \phi \|_{X(T)}.
\end{align*}

\noindent
(ii) If $\gamma$ satisfies
$2\le \gamma < \infty$ and the condition in \eqref{eq_ip},
we apply the Sobolev inequality and obtain
\begin{align}
\label{sob1}
	\langle t \rangle^{\frac{n}{2}\left( \frac{1}{r}-\frac{1}{\gamma} \right)}
	\| \phi(t) \|_{L^{\gamma}}
	\lesssim
	\langle t \rangle^{\frac{n}{2}\left( \frac{1}{r}-\frac{1}{\gamma} \right)}
	\| |\nabla|^{s^{\prime}} \phi (t) \|_{L^2}
\end{align}
with
$s^{\prime} = n ( \frac{1}{2} - \frac{1}{\gamma} )$.
Clearly, $s^{\prime} \in [0,s]$ holds under the condition in \eqref{eq_ip} and hence,
(i) implies
\[
	\langle t \rangle^{\frac{n}{2}\left( \frac{1}{r}-\frac{1}{\gamma} \right)}
	\| |\nabla|^{s^{\prime}} \phi (t) \|_{L^2}
	= \langle t \rangle^{\frac{n}{2}\left( \frac{1}{r}-\frac{1}{2} \right)+\frac{s^{\prime}}{2}}
	\| |\nabla|^{s^{\prime}} \phi (t) \|_{L^2}
	\lesssim \| \phi \|_{X(T)}.
\]
When $2s > n$ and $\gamma = \infty$, instead of \eqref{sob1},
using \eqref{eq_ip3} below with
$\varphi = \hat{\phi}$, $\gamma =1$ and $\beta=s$,
we have
\begin{align*}
	\langle t \rangle^{\frac{n}{2r}}
	\| \phi(t) \|_{L^{\infty}}
	&\lesssim
	\langle t \rangle^{\frac{n}{2r}}
	\| \hat{\phi}(t) \|_{L^{1}} \\
	& \lesssim
	\left( \langle t \rangle^{\frac{n}{2}\left( \frac{1}{r}-\frac{1}{2} \right) + \frac{s}{2}}
		\| |\nabla|^s \phi (t) \|_{L^2} \right)^{\frac{n}{2s}}
	\left(
		\langle t \rangle^{\frac{n}{2}\left( \frac{1}{r}-\frac{1}{2} \right)}
			\| \phi (t) \|_{L^2}
	\right)^{1-\frac{n}{2s}} \\
	& \lesssim
	\| \phi \|_{X(T)}.
\end{align*}

On the other hand,
to prove \eqref{eq_ip} in the case
$r \le \gamma \le 2$,
we first claim that
\begin{align}
\label{eq_ip3}
	\| \varphi \|_{L^{\gamma}}
	\lesssim 	\| \varphi \|_{L^{2}}^{1-\frac{n}{\beta}\frac{2-\gamma}{2\gamma}}
		\| |\cdot |^{\beta} \varphi \|_{L^{2}}^{\frac{n}{\beta}\frac{2-\gamma}{2\gamma}}
\end{align}
with $\beta$ satisfying $\beta > n ( \frac{1}{\gamma} - \frac{1}{2} )$
and for $\varphi \in H^{0,\beta}$.
Indeed, let $\theta$ be determined later and we calculate
\begin{align*}
	\| \varphi \|_{L^{\gamma}}^{\gamma}
	&= \int_{\mathbb{R}^n}
		( \theta^2 + |x|^2 )^{-\frac{\beta \gamma}{2}}
			( \theta^2 + |x|^2 )^{\frac{\beta \gamma}{2}}
		| \varphi(x) |^{\gamma} dx \\
	&\le \left(
			\int_{\mathbb{R}^n} ( \theta^2 + |x|^2 )^{\beta} |\varphi(x)|^2 dx
		\right)^{\frac{\gamma}{2}}
		\left( \int_{\mathbb{R}^n}
			( \theta^2 + |x|^2 )^{-\frac{\beta \gamma}{2-\gamma}} dx
			\right)^{\frac{2-\gamma}{2}}.
\end{align*}
Since
$\frac{\beta \gamma}{2-\gamma} > \frac{n}{2}$, we have
\[
	\int_{\mathbb{R}^n}
			( \theta^2 + |x|^2 )^{-\frac{\beta \gamma}{2-\gamma}} dx
	\lesssim \theta^{-\frac{2 \beta \gamma}{2-\gamma} + n}
\]
and hence,
\begin{align*}
	\| \varphi \|_{L^{\gamma}}
	&\lesssim
	\theta^{-\beta+\frac{n(2-\gamma)}{2\gamma}}
	\left(
		\int_{\mathbb{R}^n} ( \theta^2 + |x|^2 )^{\beta} |\phi(x)|^2\,dx
	\right)^{\frac{1}{2}} \\
	&\lesssim
	\theta^{\frac{n(2-\gamma)}{2\gamma}} \| \varphi \|_{L^2}
	+ \theta^{-\beta + \frac{n(2-\gamma)}{2\gamma}} \| |\cdot|^{\beta} \varphi \|_{L^2}.
\end{align*}
Taking
$\theta = \| \varphi \|_{L^2}^{-\frac{1}{\beta}}\| |\cdot|^{\beta}\varphi \|_{L^2}^{\frac{1}{\beta}}$,
we have \eqref{eq_ip3}.

From \eqref{eq_ip3}, for $r\le \gamma \le 2$,
letting $\beta$ satisfy
$n(\frac{1}{\gamma}-\frac{1}{2}) < \beta \le \alpha$,
we obtain
\begin{align*}
	\langle t \rangle^{\frac{n}{2}\left( \frac{1}{r}-\frac{1}{\gamma} \right)}
	\| \phi(t) \|_{L^{\gamma}}
	\le \left( \langle t \rangle^{\frac{n}{2}\left( \frac{1}{r}-\frac{1}{2} \right)}
		\| \phi(t) \|_{L^{2}} \right)^{1-\frac{n}{\beta}\frac{2-r}{2r}}
		\left( \langle t \rangle^{\frac{n}{2}\left( \frac{1}{r}-\frac{1}{2} \right)- \frac{\beta}{2} }
		\| |\cdot |^{\beta} \phi(t) \|_{L^{2}} \right)^{\frac{n}{\beta}\frac{2-r}{2r}}
\end{align*}
and hence, (i) gives
\[
	\langle t \rangle^{\frac{n}{2}\left( \frac{1}{r}-\frac{1}{\gamma} \right)}
	\| \phi(t) \|_{L^{\gamma}}
	\lesssim \| \phi \|_{X(T)}.
\]
Thus, we finish the proof.
\end{proof}

\subsection{Proof of Theorems \ref{thm_lwp}\ and \ref{thm_gwp}\ 
in higher dimensional cases}

We start with the estimate of the Duhamel term.
Let $\| \cdot \|_{Y(T)}$ be defined by \eqref{ynorm}.
\begin{lemma}\label{lem_duh}
Under the assumption in Theorem \ref{thm_lwp}, we have
\begin{align}
\label{est_duh1}
	\left\| \int_0^t \mathcal{D}(t-\tau) \psi (\tau) d\tau \right\|_{X(T)}
	\lesssim \int_0^T \| \psi \|_{Y(\tau)} d\tau
\end{align}
for $0<T\le 1$.
Moreover, we have
\begin{align}
\label{est_duh2}
	\left\| \int_0^t \mathcal{D}(t-\tau) \psi (\tau) d\tau \right\|_{X(T)}
	\lesssim
	\| \psi \|_{Y(T)}
	\begin{cases}
		1
			&\mbox{if}\quad p>1+\frac{2r}{n},\\
		\log \left( 2+T \right)
			&\mbox{if}\quad p = 1+\frac{2r}{n},\\
		\langle T \rangle^{1-\frac{n}{2r}(p-1)}
			&\mbox{if}\quad p< 1+ \frac{2r}{n}
	\end{cases}
\end{align}
for $r=1$, $0<T<\infty$, and
\begin{align}
\label{est_duh3}
	\left\| \int_0^t \mathcal{D}(t-\tau) \psi (\tau) d\tau \right\|_{X(T)}
	\lesssim
	\| \psi \|_{Y(T)}
	\begin{cases}
		1
			&\mbox{if}\quad p \ge 1+\frac{2r}{n},\\
		\langle T \rangle^{1-\frac{n}{2r}(p-1)}
			&\mbox{if}\quad p< 1+ \frac{2r}{n}
	\end{cases}
\end{align}
for $r \in (1,2]$, $0<T<\infty$.
Here we may take $T=\infty$ if
$p>1+\frac{2}{n}, r =1$ or $p\ge 1+\frac{2r}{n}, r \in (1,2]$.
\end{lemma}
\begin{proof}
The estimate \eqref{est_duh1} is easily proved by
looking at the proof of \eqref{est_duh2} and \eqref{est_duh3} carefully,
and we may omit it.

\noindent
{\bf Step1:}
Estimate of $\| |\nabla|^s \int_0^t \mathcal{D}(t-\tau) \psi (\tau)\,d\tau \|_{L^2}$.
We have
\begin{align*}
	\left\| |\nabla|^s \int_0^t \mathcal{D}(t-\tau) \psi (\tau)\,d\tau \right\|_{L^2} &\le
		\int_0^{\frac{t}{2}} \| |\nabla|^s \mathcal{D}(t-\tau)\psi(\tau)\|_{L^2}\,d\tau
		+\int_{\frac{t}{2}}^t \| |\nabla|^s \mathcal{D}(t-\tau)\psi(\tau)\|_{L^2}\,d\tau \\
	& =: I + I\!I.
\end{align*}
For $I\!I$, we apply Lemma \ref{lem_lin} with
$s_1=s, s_2=[s], \gamma = \rho$
(see Table 2 for notations)
and have
\begin{align*}
	I\!I &\lesssim
		\int_{\frac{t}{2}}^t \langle t-\tau \rangle^{
			-\frac{s-[s]}{2} -\frac{n}{2}\left( \frac{1}{\rho}-\frac{1}{2} \right) }
		\| |\nabla|^{[s]} \psi (\tau) \|_{L^{\rho}}\, d\tau
		+ \int_{\frac{t}{2}}^t e^{-\frac{t-\tau}{4}}
			\| |\nabla|^s \langle \nabla \rangle^{-1} \psi(\tau) \|_{L^2}\,d\tau.
\end{align*}
The Sobolev inequality implies
\begin{align}
\label{sobs}
	\| |\nabla|^s \langle \nabla \rangle^{-1} \psi(\tau) \|_{L^2}
	\lesssim
	\| |\nabla|^{[s]} \psi (\tau) \|_{L^{\rho}}
\end{align}
and hence,
\begin{align*}
	I\!I &\lesssim \int_{\frac{t}{2}}^t \langle t-\tau \rangle^{%
			-\frac{n}{2}\left( \frac{1}{\rho}-\frac{1}{2} \right)-\frac{s-[s]}{2}}
		\| |\nabla|^{[s]} \psi (\tau) \|_{L^{\rho}}\, d\tau \\
	&\lesssim \| \psi \|_{Y(T)} \int_{\frac{t}{2}}^t
	\langle t-\tau \rangle^{-\frac{n}{2}\left( \frac{1}{\rho}-\frac{1}{2} \right)-\frac{s-[s]}{2} }
	\langle \tau\rangle^{-\eta}\,d\tau \\
	&\lesssim \langle t \rangle^{%
		-\frac{s}{2}-\frac{n}{2}\left( \frac{1}{r}-\frac{1}{2} \right)
		+1 -\frac{n}{2r}(p-1)}%
		\| \psi \|_{Y(T)}.
\end{align*}

Next, we estimate $I$.
Applying Lemma \ref{lem_lin} with
$s_1 = s, s_2 =0$ and $\gamma \in [\sigma_1, \sigma_2]$
(see Table 2 for notations)
determined later,
we have
\begin{align*}
	I &\lesssim
		\int_0^{\frac{t}{2}} \langle t-\tau \rangle^{%
			-\frac{n}{2}\left( \frac{1}{\gamma}-\frac{1}{2} \right) -\frac{s}{2}}%
		\| \psi (\tau) \|_{L^{\gamma}}\, d\tau
		+\int_0^{\frac{t}{2}}
		 e^{-\frac{t-\tau}{4}} \| |\nabla|^s \langle \nabla \rangle^{-1} \psi (\tau) \|_{L^2}\,
		d\tau \\
	&=: I_1 + I_2.
\end{align*}
We calculate
$I_2$.
The Sobolev embedding and the definition of $Y(T)$-norm imply
\begin{align*}
	I_2 &\lesssim \int_0^{\frac{t}{2}}
		 e^{-\frac{t-\tau}{4}} \| |\nabla|^{[s]}\psi (\tau) \|_{L^{\rho}}\,
		d\tau
		\lesssim \| \psi \|_{Y(T)} \int_0^{\frac{t}{2}}
			e^{-\frac{t-\tau}{4}} \langle \tau \rangle^{-\eta}\, d\tau
		\lesssim e^{-\frac{t}{8}} \| \psi \|_{Y(T)}.
\end{align*}
We divide the estimate of $I_1$ into three cases.

\noindent
{\bf Case 1:} When $p <1+\frac{2r}{n}$, taking
$\gamma = r$,
and noting
$-\frac{n}{2r}(p-1) > -1$
and
$r \in [\sigma_1, \sigma_2]$,
we see that
\begin{align*}
	I_1 &\lesssim \| \psi \|_{Y(T)}
		\int_0^{\frac{t}{2}} \langle t-\tau \rangle^{%
			-\frac{n}{2}\left( \frac{1}{r} - \frac{1}{2} \right) - \frac{s}{2}}%
		\langle \tau \rangle^{-\frac{n}{2r}(p-1)}\,d\tau \\
	&\lesssim \| \psi \|_{Y(T)}
		\langle t \rangle^{%
			-\frac{n}{2}\left( \frac{1}{r} - \frac{1}{2} \right) - \frac{s}{2}
			+ 1-\frac{n}{2r}(p-1)}.%
\end{align*}

\noindent
{\bf Case 2: } When $p>1+\frac{2r}{n}$, taking $\gamma = r$ and noting
$-\frac{n}{2r}(p-1) < -1$,
we infer that
\begin{align*}
	I_1 &\lesssim
		\| \psi \|_{Y(T)}
		\int_0^{\frac{t}{2}} \langle t-\tau \rangle^{%
			-\frac{n}{2}\left( \frac{1}{r} - \frac{1}{2} \right) - \frac{s}{2}}%
		\langle \tau \rangle^{-\frac{n}{2r}(p-1)}\,d\tau \\
	&\lesssim \| \psi \|_{Y(T)}
		\langle t \rangle^{%
			-\frac{n}{2}\left( \frac{1}{r} - \frac{1}{2} \right) - \frac{s}{2} }.%
\end{align*}

\noindent
{\bf Case 3-1:} When $p=1+\frac{2r}{n}$ and $r>1$, 
we take $\gamma = \sigma_1$
and have
$-\frac{n}{2} ( \frac{p}{r}-\frac{1}{\sigma_1} ) = -\frac{n}{2}(\frac{1}{r}-\frac{1}{\sigma_1} ) -1 > -1$,
since $r > \sigma_1$.
Thus, we obtain
\begin{align*}
	I_1 &\lesssim
		\| \psi \|_{Y(T)}
		\int_0^{\frac{t}{2}} \langle t-\tau \rangle^{%
			-\frac{n}{2}\left( \frac{1}{\sigma_1} - \frac{1}{2} \right) - \frac{s}{2}}%
		\langle \tau \rangle^{-\frac{n}{2}\left( \frac{p}{r}-\frac{1}{\sigma_1} \right) }\,d\tau \\
	&\lesssim \| \psi \|_{Y(T)}
		\langle t \rangle^{%
			-\frac{n}{2}\left( \frac{1}{\sigma_1} - \frac{1}{2} \right) - \frac{s}{2}
			-\frac{n}{2}\left( \frac{1}{r}-\frac{1}{\sigma_1} \right) } \\%
	&= \| \psi \|_{Y(T)} \langle t \rangle^{%
			-\frac{n}{2}\left( \frac{1}{r} - \frac{1}{2} \right) - \frac{s}{2} }.%
\end{align*}

\noindent
{\bf Case 3-2:} When $p=1+\frac{2r}{n}$ and $r=1$,
taking $\gamma = r$, we see that
\begin{align*}
	I_1&\lesssim \| \psi \|_{Y(T)}
		\int_0^{\frac{t}{2}} \langle t-\tau \rangle^{%
			-\frac{n}{2}\left( \frac{1}{r} - \frac{1}{2} \right) - \frac{s}{2}}%
		\langle \tau \rangle^{-\frac{n}{2r}(p-1) }\, d\tau \\
	& \lesssim \| \psi \|_{Y(T)}
		\langle t \rangle^{%
			-\frac{n}{2}\left( \frac{1}{r} - \frac{1}{2} \right) - \frac{s}{2}}%
			\log (1+t).
\end{align*}

\noindent
{\bf Step2:}
Estimate of
$\| \int_0^t \mathcal{D}(t-\tau) \psi (\tau)\,d\tau \|_{L^2}$.
We have
\begin{align*}
	\left\| \int_0^t \mathcal{D}(t-\tau) \psi (\tau)\,d\tau \right\|_{L^2}
	&\lesssim \int_0^{\frac{t}{2}} \| \mathcal{D}(t-\tau) \psi (\tau) \|_{L^2}\, d\tau 
		+ \int_{\frac{t}{2}}^t \| \mathcal{D}(t-\tau) \psi (\tau) \|_{L^2}\, d\tau \\
	&=: I\!I\!I + I\!V.
\end{align*}
For $I\!V$, we apply Lemma \ref{lem_lin} with
$s_1 = s_2 = 0$ and $\gamma = q$,
where $q$ is defined in Section \ref{se_no},
and obtain
\begin{align*}
	I\!V & \lesssim \int_{\frac{t}{2}}^t \left[
		\langle t-\tau \rangle^{-\frac{n}{2}\left( \frac{1}{q} - \frac{1}{2} \right)}
		\| \psi (\tau) \|_{L^q}
		+ e^{-\frac{t-\tau}{4}} \| \langle \nabla \rangle^{-1} \psi(\tau) \|_{L^2}\right] \,d\tau \\
	&\lesssim
		\int_{\frac{t}{2}}^t 
		\langle t-\tau \rangle^{-\frac{n}{2}\left( \frac{1}{q} - \frac{1}{2} \right)}
		\| \psi (\tau) \|_{L^q}\, d\tau.		
\end{align*}
Here we have used the Sobolev embedding
$\| \langle \nabla \rangle^{-1} \psi(\tau) \|_{L^2} \lesssim \| \psi (\tau) \|_{L^q}$.
Since
$q = \frac{2n}{n+2} \in [\sigma_1, \sigma_2]$
and
$-\frac{n}{2}( \frac{1}{q}-\frac{1}{2} ) = -\frac{1}{2}$,
we have
\begin{align*}
	I\!V & \lesssim
		\| \psi \|_{Y(T)}
		\int_{\frac{t}{2}}^t 
		\langle t-\tau \rangle^{-\frac{n}{2}\left( \frac{1}{q} - \frac{1}{2} \right)}
		\langle \tau \rangle^{-\frac{n}{2}\left( \frac{p}{r}-\frac{1}{q} \right)}\, d\tau \\
	&\lesssim \| \psi \|_{Y(T)}
		\langle t \rangle^{-\frac{n}{2}\left( \frac{1}{r} - \frac{1}{2} \right) %
						+1 - \frac{n}{2r}(p-1) }.%
\end{align*}

Let us estimate the term $I\!I\!I$.
Applying Lemma \ref{lem_lin} with
$s_1=s_2=0$ and $\gamma \in [1,2]$ determined later, we demonstrate
\begin{align*}
	I\!I\!I & \lesssim
		\int_0^{\frac{t}{2}} \langle t-\tau \rangle^{%
			-\frac{n}{2}\left( \frac{1}{\gamma} - \frac{1}{2} \right) }%
			\| \psi (\tau) \|_{L^{\gamma}}\, d\tau
		+ \int_0^{\frac{t}{2}} e^{-\frac{t-\tau}{4}}
			\| \langle \nabla \rangle^{-1} \psi(\tau) \|_{L^2}\, d\tau \\
	&=: I\!I\!I_1 + I\!I\!I_2.
\end{align*}
From $q \in [\sigma_1, \sigma_2]$, 
the Sobolev embedding and the definition of $Y(T)$-norm imply
\begin{align*}
	I\!I\!I_2 &\lesssim \int_0^{\frac{t}{2}} e^{-\frac{t-\tau}{4}} \| \psi (\tau) \|_{L^q}\, d\tau
	\lesssim \| \psi \|_{Y(T)} \int_0^{\frac{t}{2}} e^{-\frac{t-\tau}{4}}
		\langle \tau \rangle^{-\frac{n}{2}\left( \frac{p}{r}-\frac{1}{q}\right) }\, d\tau
	\lesssim e^{-\frac{t}{8}} \| \psi \|_{Y(T)}.
\end{align*}
Similarly to the estimate of $I_1$, we divide the estimate of $I\!I\!I_1$ into three cases.

\noindent
{\bf Case 1:} When $p<1+\frac{2r}{n}$, we take $\gamma = r$ to obtain
\begin{align*}
	I\!I\!I_1 & \lesssim
		\| \psi \|_{Y(T)} \int_0^{\frac{t}{2}} 
		\langle t-\tau \rangle^{-\frac{n}{2}\left( \frac{1}{r} - \frac{1}{2} \right)}
		\langle \tau \rangle^{-\frac{n}{2r}(p-1)}\, d\tau\\
		&\lesssim \| \psi \|_{Y(T)}
			\langle t \rangle^{%
			-\frac{n}{2}\left( \frac{1}{r} - \frac{1}{2} \right) + 1-\frac{n}{2r}(p-1)},%
\end{align*}
where we have used $-\frac{n}{2r}(p-1) > -1$.

\noindent
{\bf Case 2:} When $p>1+\frac{2r}{n}$, taking $\gamma =r$, we see that
\begin{align*}
	I\!I\!I_1&\lesssim \| \psi \|_{Y(T)}
		\int_0^{\frac{t}{2}}
		\langle t-\tau \rangle^{ -\frac{n}{2}\left( \frac{1}{r}-\frac{1}{2} \right)}
		\langle \tau \rangle^{-\frac{n}{2r}(p-1)} \,d\tau \\
		&\lesssim
		\| \psi \|_{Y(T)} \langle t \rangle^{-\frac{n}{2}\left( \frac{1}{r}-\frac{1}{2} \right)},
\end{align*}
since $-\frac{n}{2r}(p-1) < -1$.

\noindent
{\bf Case 3-1:} When $p=1+\frac{2r}{n}$ and $r>1$,
we let $\gamma = \sigma_1$ to obtain
\begin{align*}
	I\!I\!I_1 &\lesssim \| \psi \|_{Y(T)}
		\int_0^{\frac{t}{2}}
		\langle t-\tau \rangle^{ -\frac{n}{2}\left( \frac{1}{\sigma_1}-\frac{1}{2} \right)}
		\langle \tau \rangle^{-\frac{n}{2}\left( \frac{p}{r}-\frac{1}{\sigma_1}\right) } \,d\tau \\
	&\lesssim \| \psi \|_{Y(T)}
		\langle t \rangle^{-\frac{n}{2}\left( \frac{1}{r}-\frac{1}{2} \right)}.
\end{align*}
Here we have used
$-\frac{n}{2}( \frac{p}{r}-\frac{1}{\sigma_1}) = -\frac{n}{2}( \frac{1}{r} - \frac{1}{\sigma_1} ) -1 > -1$.

\noindent
{\bf Case 3-2:} When $p=1+\frac{2r}{n}$ and $r=1$,
taking $\gamma = r$ gives
\begin{align*}
	I\!I\!I_1&\lesssim \| \psi \|_{Y(T)}
		\int_0^{\frac{t}{2}} 
		\langle t-\tau \rangle^{ -\frac{n}{2}\left( \frac{1}{r}-\frac{1}{2} \right)}
		\langle \tau \rangle^{-\frac{n}{2r}(p-1)} \,d\tau \\
	&\lesssim \| \psi \|_{Y(T)}
		\langle t \rangle^{ -\frac{n}{2}\left( \frac{1}{r}-\frac{1}{2} \right)} \log (1+t),
\end{align*}
since $-\frac{n}{2r}(p-1) = -1$.

\noindent
{\bf Step3:}
Estimate of
$\| |\cdot|^{\alpha} \int_0^t \mathcal{D}(t-\tau) \psi (\tau)\,d\tau \|_{L^2}$.
We apply Lemma \ref{lem_lin} with
$\beta = \alpha, \nu = q$ and $\gamma \in [\sigma_1, \sigma_2]$ determined later,
and obtain
\begin{align*}
	& \left\| |\cdot|^{\alpha} \int_0^t \mathcal{D}(t-\tau) \psi (\tau)\,d\tau \right\|_{L^2} \\
	&\quad \lesssim
	\int_0^t \langle t-\tau \rangle^{%
		-\frac{n}{2}\left( \frac{1}{\gamma}-\frac{1}{2} \right) + \frac{\alpha}{2} }%
		\| \psi (\tau) \|_{L^{\gamma}}\, d\tau
		+ \int_0^t \langle t-\tau \rangle^{%
		-\frac{n}{2}\left( \frac{1}{q} - \frac{1}{2} \right)}%
		\| |\cdot |^{\alpha} \psi(\tau) \|_{L^q}\,d\tau \\
	&\qquad + \int_0^t e^{-\frac{t-\tau}{4}}
		\| \langle \cdot \rangle^{\alpha} \langle \nabla \rangle^{-1} \psi (\tau) \|_{L^2}\,d\tau \\
	&\quad =: V_1 + V_2 + V_3.
\end{align*}
In order to estimate $V_3$, we employ Lemma \ref{lem_sob} and deduce that
\[	
	V_3 \lesssim \int_0^t e^{-\frac{t-\tau}{4}}
		\| \langle \cdot \rangle^{\alpha} \psi (\tau) \|_{L^q}\,d\tau.
\]
Therefore, the estimates of $V_2, V_3$ reduce to that of
$ \int_0^t \langle t-\tau \rangle^{%
-\frac{n}{2}\left( \frac{1}{q} - \frac{1}{2} \right)}%
\| \langle \cdot \rangle^{\alpha} \psi(\tau) \|_{L^q}\,d\tau$.
Noting $-\frac{n}{2}( \frac{1}{q} - \frac{1}{2} )=-\frac{1}{2}$,
we have
\begin{align*}
	\int_0^t \langle t-\tau \rangle^{%
		-\frac{n}{2}\left( \frac{1}{q} - \frac{1}{2} \right)}%
		\| \langle \cdot \rangle^{\alpha} \psi(\tau) \|_{L^q}\,d\tau
	 \lesssim \| \psi \|_{Y(T)}
		\int_0^t \langle t-\tau \rangle^{%
		-\frac{n}{2}\left( \frac{1}{q} - \frac{1}{2} \right)}%
		\langle \tau \rangle^{-\zeta}\, d\tau,
\end{align*}
where $\zeta$ is defined in Section \ref{se_no}.
We compute
\[
	\int_0^t \langle t-\tau \rangle^{%
		-\frac{n}{2}\left( \frac{1}{q} - \frac{1}{2} \right)}%
		\langle \tau \rangle^{-\zeta}\, d\tau
	\lesssim
	\begin{cases}
	\langle t \rangle^{%
		-\frac{n}{2}\left( \frac{1}{r}-\frac{1}{2} \right)+\frac{\alpha}{2}%
		+1 -\frac{n}{2r}(p-1)}%
	&(\zeta < 1),\\
	\langle t \rangle^{-\frac{1}{2}} \log (1+t)
	&(\zeta =1),\\
	\langle t \rangle^{-\frac{1}{2}}
	&(\zeta >1)
	\end{cases}
\]	
and note that
$\zeta <1$ holds if $p < 1+\frac{2r}{n}$.
Therefore, we may summarize them as
\[
	\int_0^t \langle t-\tau \rangle^{%
		-\frac{n}{2}\left( \frac{1}{q} - \frac{1}{2} \right)}%
		\langle \tau \rangle^{-\zeta}\, d\tau
	\lesssim
	\begin{cases}
	\langle t \rangle^{%
		-\frac{n}{2}\left( \frac{1}{r}-\frac{1}{2} \right)+\frac{\alpha}{2}%
		+1 -\frac{n}{2r}(p-1)}%
	&(p < 1+\frac{2r}{n}),\\
	\langle t \rangle^{%
		-\frac{n}{2}\left( \frac{1}{r}-\frac{1}{2} \right)+\frac{\alpha}{2}}%
	&(p\ge 1+\frac{2r}{n})
	\end{cases}
\]
and hence, we have
\[
	V_2 + V_3 \lesssim
	\| \psi \|_{Y(T)}
	\begin{cases}
	\langle t \rangle^{%
		-\frac{n}{2}\left( \frac{1}{r}-\frac{1}{2} \right)+\frac{\alpha}{2}%
		+1 -\frac{n}{2r}(p-1)}%
	&(p < 1+\frac{2r}{n}),\\
	\langle t \rangle^{%
		-\frac{n}{2}\left( \frac{1}{r}-\frac{1}{2} \right)+\frac{\alpha}{2}}%
	&(p\ge 1+\frac{2r}{n}).
	\end{cases}
\]

For $V_1$,
as before, we divide the estimate into three cases.

\noindent
{\bf Case 1:} When $p<1+\frac{2r}{n}$,
by taking $\gamma =r$,
the definition of $Y(T)$-norm leads to
\begin{align*}
	V_1 &\lesssim
	\| \psi \|_{Y(T)} \int_0^t \langle t-\tau \rangle^{%
		-\frac{n}{2}\left( \frac{1}{r}-\frac{1}{2} \right) + \frac{\alpha}{2} }%
		\langle \tau \rangle^{-\frac{n}{2r}(p-1)} \, d\tau \\
	&\lesssim \| \psi \|_{Y(T)}
		\langle t \rangle^{%
			-\frac{n}{2}\left( \frac{1}{r}-\frac{1}{2} \right) + \frac{\alpha}{2}
			+1 -\frac{n}{2r}(p-1)},%
\end{align*}
since
$-\frac{n}{2r}(p-1)>-1$
and
$-\frac{n}{2}\left( \frac{1}{r}-\frac{1}{2} \right) + \frac{\alpha}{2}>-1$.

\noindent
{\bf Case 2:} When $p>1+\frac{2r}{n}$,
we choose $\gamma = r$ and have
\begin{align*}
	V_1 &\lesssim
	\| \psi \|_{Y(T)} \int_0^t \langle t-\tau \rangle^{%
		-\frac{n}{2}\left( \frac{1}{r}-\frac{1}{2} \right) + \frac{\alpha}{2} }%
		\langle \tau \rangle^{-\frac{n}{2r}(p-1)} \, d\tau \\
	&\lesssim \| \psi \|_{Y(T)}
		\langle t \rangle^{%
			-\frac{n}{2}\left( \frac{1}{r}-\frac{1}{2} \right) + \frac{\alpha}{2}},%
\end{align*}
since
$-\frac{n}{2r}(p-1)<-1$
and
$-\frac{n}{2}\left( \frac{1}{r}-\frac{1}{2} \right) + \frac{\alpha}{2}>-1$.

\noindent
{\bf Case 3-1:} When $p=1+\frac{2r}{n}$ and $r>1$,
letting $\gamma = \sigma_1$, we see that
\begin{align*}
	V_1&\lesssim
	\| \psi \|_{Y(T)} \int_0^t
	\langle t-\tau \rangle^{%
		-\frac{n}{2}\left( \frac{1}{\sigma_1}-\frac{1}{2} \right) + \frac{\alpha}{2}}%
		\langle \tau \rangle^{-\frac{n}{2}\left(\frac{p}{r}-\frac{1}{\sigma_1} \right)}
		\,d\tau \\
	&\lesssim
	\| \psi \|_{Y(T)}
	\langle t \rangle^{%
		-\frac{n}{2}\left( \frac{1}{\sigma_1}-\frac{1}{2} \right) + \frac{\alpha}{2}
		-\frac{n}{2}\left(\frac{p}{r}-\frac{1}{\sigma_1} \right) + 1} \\%
	&= \| \psi \|_{Y(T)}
	\langle t \rangle^{%
		-\frac{n}{2}\left( \frac{1}{r}-\frac{1}{2} \right) + \frac{\alpha}{2}}.%
\end{align*}
Here we have used that
$-\frac{n}{2}(\frac{p}{r}-\frac{1}{\sigma_1})>-1$
and
$-\frac{n}{2}\left( \frac{1}{\sigma_1}-\frac{1}{2} \right) + \frac{\alpha}{2}>-1$.

\noindent
{\bf Case 3-2:} When $p=1+\frac{2r}{n}$ and $r=1$,
taking $\gamma = r$ implies
\begin{align*}
	V_1 &\lesssim
	\| \psi \|_{Y(T)} \int_0^t
		\langle t-\tau \rangle^{%
			-\frac{n}{2}\left( \frac{1}{r}-\frac{1}{2} \right) + \frac{\alpha}{2}}%
		\langle \tau \rangle^{-\frac{n}{2r}(p-1)}\, d\tau \\
	&\lesssim \| \psi \|_{Y(T)}
			\langle t \rangle^{-\frac{n}{2}\left( \frac{1}{r}-\frac{1}{2} \right) + \frac{\alpha}{2}}
			\log (1+t),
\end{align*}
since $-\frac{n}{2r}(p-1)=-1$.
Summing up all the estimates above, we reach
\eqref{est_duh2} and \eqref{est_duh3}.
\end{proof}

\begin{lemma}\label{lem_nl}
Under the assumptions in Theorem \ref{thm_lwp}, we have
\begin{align*}
	\| \mathcal{N}(u) \|_{Y(T)} & \lesssim \| u \|_{X(T)}^p,\\
	\| \mathcal{N}(u) - \mathcal{N}(v) \|_{Y(T)}
		&\lesssim \| u-v \|_{X(T)} \left(  \| u \|_{X(T)} +  \| v \|_{X(T)} \right)^{p-1}.
\end{align*}
\end{lemma}
\begin{proof}
First, we consider the term
$\| \langle \cdot \rangle^{\alpha} \mathcal{N}(u) \|_{L^q}$.
By the assumption \eqref{nonlin} and the H\"{o}lder inequality
with $\frac{1}{q} = \frac{1}{2}+\frac{1}{n}$, we have
\[
	\| \langle \cdot \rangle^{\alpha} \mathcal{N}(u) \|_{L^q}
	\le \| \langle \cdot \rangle^{\alpha} u \|_{L^2}
		\| u \|_{L^{n(p-1)}}^{p-1}
\]
and hence,
\begin{align*}
	\langle t \rangle^{\zeta} \| \langle \cdot \rangle^{\alpha} \mathcal{N}(u) \|_{L^q}
	\le
		\langle t \rangle^{\frac{n}{2}\left(\frac{1}{r}-\frac{1}{2}\right)-\frac{\alpha}{2}}
		\| \langle \cdot \rangle^{\alpha} u \|_{L^2}
	\left( \langle t \rangle^{%
			\frac{n}{2}\left( \frac{1}{r} - \frac{1}{n(p-1)} \right)}%
		\| u \|_{L^{n(p-1)}} \right)^{p-1}.
\end{align*}
From the assumption of Theorem \ref{thm_lwp},
$r \le n(p-1)$ is valid and
$n(p-1) \le \frac{2n}{n-2s}$ also holds when $2s < n$.
Therefore, we apply Lemma \ref{lem_ip} to obtain
\[
	\langle t \rangle^{\zeta} \| \langle \cdot \rangle^{\alpha} \mathcal{N}(u) \|_{L^q}
	\lesssim \| u \|_{X(T)}^p.
\]

Next, we estimate
$ \| \mathcal{N}(u) \|_{L^{\gamma}}$
for $\gamma \in [\sigma_1, \sigma_2]$.
By the assumption \eqref{nonlin}, we see that
\begin{align*}
	\langle t \rangle^{\frac{n}{2}\left( \frac{p}{r}-\frac{1}{\gamma} \right)}
	\| \mathcal{N}(u) \|_{L^{\gamma}}
	\lesssim \left(
		\langle t \rangle^{\frac{n}{2}\left( \frac{1}{r}-\frac{1}{p\gamma} \right)}
		\| u \|_{L^{p\gamma}} \right)^p.
\end{align*}
Since
$p \ge 1+\frac{r}{n}$ and
$\gamma \ge \sigma_1 \ge \frac{nr}{n+r}$,
we have
\[
	r =  \left( 1+ \frac{r}{n} \right) \frac{nr}{n+r} \le p \gamma.
\]
Also, when $2s < n$,
the assumption $\sigma_2 \le \frac{2n}{p(n-2s)}$ leads to
\[
	p \gamma \le p \sigma_2 \le \frac{2n}{n-2s}.
\]
Therefore, $p \gamma \in [r, \frac{2n}{n-2s}]$ is valid
and we apply Lemma \ref{lem_ip} to obtain
\begin{align}
\label{est_nl1}
	\langle t \rangle^{\frac{n}{2}\left( \frac{p}{r}-\frac{1}{\gamma} \right)}
	\| \mathcal{N}(u) \|_{L^{\gamma}}
	\lesssim \| u \|_{X(T)}^p.
\end{align}

Finally, we estimate
$\| |\nabla|^{[s]} \mathcal{N}(u) \|_{L^{\rho}}$.
The case $[s] = 0$ reduces to \eqref{est_nl1}, because $\rho \in [\sigma_1, \sigma_2]$.
When $[s] \ge 1$,
by the Fa\`{a} di Bruno formula,
for $\nu \in \mathbb{Z}_{\ge 0}$ with $|\nu|\ge 1$
we have
\[
	\partial^{\nu}\mathcal{N}(u)
	= \sum_{l=1}^{|\nu|} \mathcal{N}^{(l)}(u)
	\sum_{\substack{|\nu_1|\ge 1,\ldots, |\nu_l| \ge 1\\ \nu_1+\cdots+\nu_l = \nu}}
		C_{\nu_1,\ldots,\nu_l}^l \partial^{\nu_1}u \cdots \partial^{\nu_l} u.
\]
Using this with
$|\mathcal{N}^{(l)}(u) | \lesssim |u|^{p-l} = |u|^{p-[s]} |u|^{[s]-l}$
and the H\"{o}lder inequality, we see that
\begin{align}
\label{est_non1}
	\| |\nabla|^{[s]} \mathcal{N}(u) \|_{L^{\rho}}
	\lesssim \| |u|^{p-[s]} \|_{L^{q_0}}
		\sum_{k}  \| |\nabla|^{k_1} u \|_{L^{q_1(k)}}
			\cdots  \| |\nabla|^{k_{[s]}} u \|_{L^{q_{[s]}(k)}},
\end{align}
where the sum is taken over
$k = (k_1, \ldots, k_{[s]}) \in \mathbb{Z}_{\ge 0}^{[s]}$
satisfying $|k| = k_1 + \cdots k_{[s]} = [s]$,
and
$q_1(k), \ldots, q_{[s]}(k)$ satisfy
\begin{align}
\label{hol_exp}
	\frac{1}{\rho} = \frac{1}{q_0} + \frac{1}{q_1(k)} + \cdots + \frac{1}{q_{[s]}(k)}
\end{align}
and are defined in the following way.

For each fixed $k = (k_1, \ldots, k_{[s]})$,
let us choose $s_1, \ldots, s_{[s]}$ so that
\begin{align}
\label{sj}
	\max\{ 0, k_j - \mu \} \le s_j < k_j + \frac{1}{p-1}, \quad
	\sum_{j=1}^{[s]} s_j = s-\mu,
\end{align}
where $\mu = \frac{n}{2}-\frac{1}{p-1}$.
This is always possible.
Indeed, first, it is obvious that
$\max\{ 0, k_j - \mu \} < k_j + \frac{1}{p-1}$
and the interval
$[\max\{ 0, k_j - \mu \}, k_j + \frac{1}{p-1})$ is not empty.
Next, we demonstrate that
\begin{align}
\label{s-mu}
	\sum_{j=1}^{[s]} \max\{ 0, k_j - \mu \}
	\le s - \mu
	< \sum_{j=1}^{[s]} \left( k_j + \frac{1}{p-1} \right).
\end{align}
To prove this, with a direct calculation we have
\[
	\sum_{j=1}^{[s]} \left( k_j + \frac{1}{p-1} \right) = [s] + \frac{[s]}{p-1}
		> s-\mu,
\]
since $n \ge 2$.
Also, when $[s] = 1$, it is trivial that
$\max\{ 0, [s] - \mu \} \le s -\mu$,
since the assumption that
$p \le 1+\frac{2}{n-2s}$ if $2s < n$ implies
$s \ge \mu$.
When $[s] \ge 2$,
noting that
$\mu < 0$ leads to $[s]< p < 1+\frac{2}{n} \le 2$,
we may assume $\mu \ge 0$.
Therefore, we have
\[
	\sum_{j=1}^{[s]} \max \{ 0, k_j - \mu \}
	=0 \le s - \mu
\]
if $k_j \le \mu$ for all $j=1,\ldots,[s]$
and
\[
	\sum_{j=1}^{[s]} \max \{ 0, k_j - \mu \}
	\le (k_{i} - \mu) + \sum_{j\neq i} k_j = [s] - \mu \le s-\mu
\]
if $k_i > \mu$ for some $i \in \{ 1, \ldots, [s] \}$.
Thus, we prove \eqref{s-mu} and we can actually find
$s_j \ (j = 1,\ldots, [s])$ satisfying \eqref{sj}.

From these $s_j$, we define
\[
	\frac{1}{q_j(k)} = \frac{1}{2} - \frac{\mu+s_j-k_j}{n},\quad
	\frac{1}{q_0} = \left( \frac{1}{2} - \frac{\mu}{n} \right) ( p - [s] ).
\]
Then, a straightforward calculation shows \eqref{hol_exp}.
Moreover, by the property \eqref{sj} and the assumption $p \ge 1+\frac{r}{n}$,
we have
$2 \le q_j(k) < \infty$ for $j=1,\ldots,[s]$
and $r \le n(p-1) = q_0 (p-[s]) < \infty$, respectively.
Also, we remark that
$\mu + s_j \le s$ holds for all $j=1,\ldots,[s]$ due to
$s_j \ge 0$ and $\sum_{j=1}^{[s]} s_j = s-\mu$.
Hence, we apply the Sobolev embedding and Lemma \ref{lem_ip}
to \eqref{est_non1} and obtain
\begin{align*}
	\langle t \rangle^{\eta}\| |\nabla|^{[s]} \mathcal{N}(u) \|_{L^{\rho}}
	&\lesssim \langle t \rangle^{\eta} \| u \|_{L^{q_0(p-[s])}}^{p-[s]}
			\sum_{k} \| |\nabla|^{\mu + s_1} u \|_{L^2}
			\cdots \| |\nabla |^{\mu + s_{[s]}} u \|_{L^2} \\
	&\lesssim \left( \langle t \rangle^{ \frac{n}{2}\left( \frac{1}{r} - \frac{1}{n(p-1)} \right)}
				\| u \|_{L^{n(p-1)}} \right)^{p-[s]} \\
	&\quad \times \sum_k \prod_{j=1}^{[s]}
				\langle t \rangle^{\frac{n}{2}\left( \frac{1}{r}-\frac{1}{2} \right)+\frac{\mu+s_j}{2}}
				\| |\nabla|^{\mu + s_j} u \|_{L^2} \\
	&\lesssim \| u \|_{X(T)}^p.
\end{align*}

In the same manner, with the assumption \eqref{nonlin},
we can prove the estimate for
$\| \mathcal{N}(u) - \mathcal{N}(v)\|_{Y(T)}$.
Indeed, for example,
we demonstrate
\begin{align*}
	\| \mathcal{N}(u) - \mathcal{N}(v) \|_{L^{\gamma}}
	&\lesssim \| u- v\|_{L^{p \gamma}} \| (|u|+|v|)^{p-1} \|_{L^{\frac{p \gamma}{p-1}}} \\
	&\lesssim \| u-v \|_{L^{p \gamma}}
			\left( \| u \|_{L^{p\gamma}} + \| v \|_{L^{p\gamma}} \right)^{p-1}.
\end{align*}
Hence, we find that
\begin{align*}
	\langle t \rangle^{\frac{n}{2}\left( \frac{p}{r} - \frac{1}{\gamma} \right)}
	\| \mathcal{N}(u) - \mathcal{N}(v) \|_{L^{\gamma}}
	&\lesssim
	\langle t \rangle^{\frac{n}{2}\left(\frac{1}{r} - \frac{1}{p \gamma} \right)}
	\| u - v \|_{L^{p\gamma}} 
	\left( \langle t \rangle^{\frac{n}{2}\left(\frac{1}{r} - \frac{1}{p \gamma} \right)}
			\left( \| u \|_{L^{p\gamma}} + \| v \|_{L^{p\gamma}} \right) \right)^{p-1}\\
	&\lesssim
	\| u - v \|_{X(T)} \left( \| u \|_{X(T)} + \| v \|_{X(T)} \right)^{p-1}.
\end{align*}
The other terms can be estimated in a similar way.
\end{proof}

Now we are in the position to prove Theorem \ref{thm_lwp} when $n\ge 2$. 
\begin{proof}[Proof of Theorem \ref{thm_lwp} when $n\ge 2$]
We apply the contraction mapping principle in
\begin{align}
\label{x_ep}
	X_{\varepsilon}(T)
	= \left\{
		v \in 
		L^{\infty}(0,T; H^{s,0}\cap H^{0,\alpha} (\mathbb{R}^n) )
		; \| v \|_{X(T)} \le C_0 \varepsilon
		\right\},
\end{align}
where
$C_0>0$
is determined later.
Also, we define a metric in $X_{\varepsilon}(T)$ by
\begin{align}
\label{met}
	d(u,v) := \| u - v \|_{L^{\infty}(0,T; H^{s,0}\cap H^{0,\alpha})}.
\end{align}
Then, clearly,
$X_{\varepsilon}(T)$ becomes a complete metric space.

We define the mapping
$\mathcal{M}$
by
\begin{align}
\label{map}
	\mathcal{M}v(t) =
	\tilde{\mathcal{D}}(t) \varepsilon u_0 + \mathcal{D}(t)\varepsilon u_1
	+ \int_0^t \mathcal{D}(t-\tau) \mathcal{N}(v(\tau)) d\tau.
\end{align}
Let $0 < T < 1$.
Then, by Lemmas \ref{lem_lin}, \ref{lem_lin2}, \ref{lem_duh} \eqref{est_duh1}, and \ref{lem_nl}, we have
\begin{align*}
	\| \mathcal{M}v \|_{X(T)}
	\le C \varepsilon
		( \| u_0 \|_{H^{s,0}\cap H^{0,\alpha}} + \| u_1 \|_{H^{s-1,0}\cap H^{0,\alpha}} )
	+ C T \| v \|_{X(T)}^p.
\end{align*}
Taking a constant
$C_0$
so that
$C ( \| u_0 \|_{H^{s,0}\cap H^{0,\alpha}} + \| u_1 \|_{H^{s-1,0}\cap H^{0,\alpha}} ) \le \frac{C_0}{2}$,
we have
\[
	\| \mathcal{M}v \|_{X(T)}
	\le \frac{C_0}{2}\varepsilon + C T C_0^p \varepsilon^p.
\]
Letting $T$ be sufficiently small so that
$C T C_0^p \varepsilon^{p-1} \le \frac{C_0}{2}$,
we conclude
$\| \mathcal{M}v \|_{X(T)} \le C_0 \varepsilon$.
Thus, $\mathcal{M}$ maps $X_{\varepsilon}(T)$ to itself.

In a similar way, we have
\[
	\| \mathcal{M}(v) - \mathcal{M}(w) \|_{X(T)}
	\le C \varepsilon^{p-1} T \| v - w \|_{X(T)}
\]
for $v, w \in X_{\varepsilon}(T)$
and hence, taking $T$ further small, we have
\[
	\| \mathcal{M}(v) - \mathcal{M}(w) \|_{X(T)}
	\le \frac{1}{2} \| v - w \|_{X(T)}.
\]
Therefore, $\mathcal{M}$ is a contraction mapping
on $X_{\varepsilon}(T)$ and
there is a unique fixed point $u$ in $X_{\varepsilon}(T)$.
By the definition of $\mathcal{M}$,
$u$ is a solution of the Cauchy problem \eqref{nldw}.

Finally, we prove the continuity with respect to $t$ of the solution $u$.
For $t_1, t_2 \ge 0$, we have
\begin{align*}
	u(t_1) - u(t_2) &= (\mathcal{D}(t_1) - \mathcal{D}(t_2) )u_1
		+ (\tilde{\mathcal{D}}(t_1) -\tilde{\mathcal{D}}(t_2)) u_0 \\
		&\quad + \int_{t_2}^{t_1} \mathcal{D}(t_1-\tau) \mathcal{N}(u(\tau))\, d\tau
				+ \int_0^{t_2} (\mathcal{D}(t_1-\tau) - \mathcal{D}(t_2-\tau) )
				\mathcal{N}(u(\tau))\, d\tau.
\end{align*}
Applying \eqref{con1}, \eqref{con2} and the Lebesgue convergence theorem
with the bound \eqref{est_duh1}, we can easily prove
$\lim_{t_1\to t_2} \| u(t_1) - u(t_2) \|_{H^{s,0}\cap H^{0,\alpha}}=0$,
which finishes the proof.
\end{proof}

\begin{lemma}\label{lem_uni}
Under the assumptions in Theorem \ref{thm_lwp},
the mild solution of \eqref{nldw} is unique in the class
$C([0,T); H^{s,0}(\mathbb{R}^n) \cap H^{0,\alpha}(\mathbb{R}^n))$.
\end{lemma}
\begin{proof}
Let $T_0>0$ and
let
$u, v \in C([0,T_0); H^{s,0}(\mathbb{R}^n) \cap H^{0,\alpha}(\mathbb{R}^n))$
are mild solutions of \eqref{nldw} with the same initial data
$(u_0, u_1)$.
We take arbitrary $0< T_1 < \min\{ 1, T_0 \}$ and fix it.
Then, we have
$\| u \|_{X(T_1)} + \| v \|_{X(T_1)} \le M$
with some $M > 0$.
By the first assertion of Lemma \ref{lem_duh}
and applying Lemma \ref{lem_nl},
we see that for $T\in [0,T_1]$
\begin{align*}
	\| u - v \|_{X(T)}
	&\lesssim M^{p-1} \int_0^{T} \| u - v \|_{X(\tau)} d\tau.
\end{align*}
Hence, we apply the Gronwall inequality and obtain
$\| u - v \|_{X(T)} \equiv 0$ for $T\in [0,T_1]$,
namely
$u \equiv v$ in $t\in [0,T_1]$.
Applying the same argument starting at
$(u(T_1), u_t(T_1))$
instead of
$(u_0, u_1)$,
we have
$u\equiv v$ on $[0, 2T_1]$.
Continuing this until reaching $T_0$,
we have the uniqueness in 
$C([0,T_0); H^{s,0}(\mathbb{R}^n) \cap H^{0,\alpha}(\mathbb{R}^n))$.
\end{proof}

Finally, we mention about the continuity of the solution with respect to
the initial data.
\begin{lemma}[Lipschitz continuity of the solution map]\label{lem_lip}
The solution map
\[
	\left( H^{s,0}\cap H^{0,\alpha} \right)
	\times \left( H^{s-1,0}\cap H^{0,\alpha} \right)
	\to C([0,T); H^{s,0}\cap H^{0,\alpha});\quad
	(u_0,u_1) \mapsto u
\]
is locally Lipschitz continuous, that is, for any $T_1 < T$, we have
\[
	\| u(t) - v(t) \|_{H^{s,0}\cap H^{0,\alpha}}
	\lesssim \| u_0 - v_0 \|_{H^{s,0}\cap H^{0,\alpha}} + \| u_1 - v_1 \|_{H^{s-1,0}\cap H^{0,\alpha}}
\]
on $t \in [0,T_1]$, where
$u$ and  $v$ are solutions of \eqref{nldw}
in $C([0,T); H^{s,0}\cap H^{0,\alpha})$
with the initial data
$(u_0, u_1)$ and $(v_0, v_1)$, respectively.
\end{lemma}
\begin{proof}
Let $u$ and  $v$ are solutions of \eqref{nldw}
in $C([0,T); H^{s,0}\cap H^{0,\alpha})$ and fix $T_1 < T$.
Then, we have
$\| u \|_{X(T_1)} + \| v \|_{X(T_1)} \le M$
with some constant $M>0$.
Therefore, by Lemmas \ref{lem_lin}, \ref{lem_duh} and \ref{lem_nl}, we have
\begin{align*}
	\| u(t) - v(t) \|_{X(T_1)}
	&\lesssim \| u_0 - v_0 \|_{H^{s,0}\cap H^{0,\alpha}}
		+ \| u_1 - v_1 \|_{H^{s-1,0}\cap H^{0,\alpha}} \\
	&\quad + M^{p-1} \int_0^{T_1} \| u(\tau) - v(\tau) \|_{X(\tau)}\, d\tau.
\end{align*}
The Gronwall inequality implies
\[
	\| u(t) - v(t) \|_{X(T_1)} \lesssim
	\| u_0 - v_0 \|_{H^{s,0}\cap H^{0,\alpha}}
		+ \| u_1 - v_1 \|_{H^{s-1,0}\cap H^{0,\alpha}},
\]
which completes the proof.
\end{proof}

\begin{proof}[Proof of Theorem \ref{thm_gwp}]
We assume that
$p>1+\frac{2r}{n}$, $r\in [1,2]$
or
$p=1+\frac{2r}{n}$, $r\in (1,2]$
and consider the mapping
$\mathcal{M}$ defined on \eqref{map} in the complete metric space
\[
	X_{\varepsilon}(\infty)
	= \left\{
		v \in 
		L^{\infty}(0,\infty; H^{s,0}\cap H^{0,\alpha} (\mathbb{R}^n) )
		; \| v \|_{X(\infty)} \le C_0 \varepsilon
		\right\}
\]
with the metric \eqref{met}.
Then, by Lemma \ref{lem_duh}, we have
\[
	\left\| \int_0^t \mathcal{D}(t-\tau) \mathcal{N}(u(\tau)) d\tau \right\|_{X(\infty)}
	\lesssim \| \mathcal{N}(u) \|_{Y(\infty)}.
\]
Hence, Lemma  \ref{lem_nl} implies
\begin{align*}
	\| \mathcal{M}v \|_{X(\infty)}
	\le C \varepsilon
		( \| u_0 \|_{H^{s,0}\cap H^{0,\alpha}} + \| u_1 \|_{H^{s-1,0}\cap H^{0,\alpha}} )
	+ C \| v \|_{X(\infty)}^p.
\end{align*}
As before, taking a constant
$C_0$
so that
$C ( \| u_0 \|_{H^{s,0}\cap H^{0,\alpha}} + \| u_1 \|_{H^{s-1,0}\cap H^{0,\alpha}} ) \le \frac{C_0}{2}$,
we see that
\[
	\| \mathcal{M}v \|_{X(\infty)}
	\le \frac{C_0}{2}\varepsilon + C\varepsilon^p.
\]
Finally, taking $\varepsilon$ sufficiently small, we conclude
$\| \mathcal{M}v \|_{X(\infty)} \le C_0 \varepsilon$
and hence, $\mathcal{M}$ maps $X_{\varepsilon}(\infty)$ to itself.

In a similar way, we have
\[
	\| \mathcal{M}(v) - \mathcal{M}(w) \|_{X(\infty)}
	\le C \varepsilon^{p-1} \| v - w \|_{X(\infty)}
\]
for $v, w \in X_{\varepsilon}(\infty)$
and hence, taking $\varepsilon$ further small,
we conclude that $\mathcal{M}$ is a contraction mapping.
Therefore, $\mathcal{M}$ has a unique fixed point
$u$ in $X_{\varepsilon}(\infty)$
and by the definition of $\mathcal{M}$,
$u$ is a mild solution of \eqref{nldw}.
In the same way as in the proof of Theorem \ref{thm_lwp},
we deduce that
$u$ belongs to
$C([0,\infty); H^{s,0}(\mathbb{R}^n) \cap H^{0,\alpha}(\mathbb{R}^n))$.
The uniqueness and continuity of the solution have already proved in 
Lemmas \ref{lem_lip} and  \ref{lem_uni}.
\end{proof}

\subsection{Proof of Theorems \ref{thm_lwp}\ and \ref{thm_gwp}\ 
in the one-dimensional case}
\begin{lemma}\label{lem_duh1}
Under the assumption in Theorem \ref{thm_lwp}, we have
\begin{align*}
	\left\| \int_0^t \mathcal{D}(t-\tau) \psi (\tau) d\tau \right\|_{X(T)}
	\lesssim \int_0^T \| \psi \|_{Y(\tau)} d\tau
\end{align*}
for $0<T\le 1$.
Moreover, we have
\begin{align*}
	\left\| \int_0^t \mathcal{D}(t-\tau) \psi (\tau) d\tau \right\|_{X(T)}
	\lesssim
	\| \psi \|_{Y(T)}
	\begin{cases}
		1
			&\mbox{if}\quad p>1+2r,\\
		\log \left( 2+T \right)
			&\mbox{if}\quad p = 1+2r,\\
		\langle T \rangle^{1-\frac{1}{2r}(p-1)}
			&\mbox{if}\quad p< 1+ 2r
	\end{cases}
\end{align*}
for $r=1$, $0<T<\infty$, and
\begin{align*}
	\left\| \int_0^t \mathcal{D}(t-\tau) \psi (\tau) d\tau \right\|_{X(T)}
	\lesssim
	\| \psi \|_{Y(T)}
	\begin{cases}
		1
			&\mbox{if}\quad p \ge 1+ 2r,\\
		\langle T \rangle^{1-\frac{1}{2r}(p-1)}
			&\mbox{if}\quad p< 1+ 2r
	\end{cases}
\end{align*}
for $r \in (1,2]$, $0<T<\infty$.
\end{lemma}
\begin{proof}
The proof is almost the same as that of Lemma \ref{lem_duh} and hence
we present only the outline.
The main difference arises in the Sobolev inequality, that is, we shall use
\[
	\| |\nabla|^s \langle \nabla \rangle^{-1} \psi \|_{L^2}
	\lesssim \| \psi \|_{L^2}
\]
instead of \eqref{sobs}, and we use Lemma \ref{lem_sob} with $q=1$, that is,
\[
	\| \langle \cdot \rangle^{\alpha} \langle \nabla \rangle^{-1} \psi \|_{L^2}
	\lesssim \| \langle \cdot \rangle^{\alpha} \psi \|_{L^1}.
\]
We estimate
\begin{align*}
	&\left\| |\nabla|^s \int_0^t \mathcal{D}(t-\tau) \psi(\tau) \,d\tau \right\|_{L^2} \\
	&\quad \lesssim
		\int_0^{\frac{t}{2}} \| |\nabla|^s \mathcal{D}(t-\tau)\psi(\tau)\|_{L^2}\,d\tau
		+ \int_{\frac{t}{2}}^t \| |\nabla|^s \mathcal{D}(t-\tau)\psi(\tau)\|_{L^2}\,d\tau \\
	&\quad =: I + I\!I
\end{align*}
and Lemma \ref{lem_lin} with $s_1=s, s_2=0$, $\gamma =1$ implies
\begin{align*}
	I\!I &\lesssim
		\int_{\frac{t}{2}}^t \langle t- \tau \rangle^{-\frac{1}{4}-\frac{s}{2}}
			\|\psi(\tau) \|_{L^1}\,d\tau
		+ \int_{\frac{t}{2}}^t e^{-\frac{t-\tau}{4}} \| \psi (\tau) \|_{L^2}\, d\tau \\
	&\lesssim
		\| \psi \|_{Y(T)} \langle t \rangle^{%
			-\frac{1}{2}\left( \frac{1}{r}-\frac{1}{2}\right)-\frac{s}{2}%
			-\frac{1}{2r}(p-1) + 1}.%
\end{align*}
Similarly, we have
\begin{align*}
	I&\lesssim
	\int_0^{\frac{t}{2}}
		\langle t-\tau \rangle^{-\frac{1}{2}\left(\frac{1}{\gamma}-\frac{1}{2}\right)-\frac{s}{2}}
		\| \psi (\tau) \|_{L^{\gamma}}\,d\tau
		+ \int_0^{\frac{t}{2}} e^{-\frac{t-\tau}{4}}
			\| \psi (\tau) \|_{L^2}\,d\tau \\
	&\lesssim
	\| \psi \|_{Y(T)}
	\langle t \rangle^{-\frac{1}{2}\left( \frac{1}{r}-\frac{1}{2}\right) - \frac{s}{2}}
	\begin{cases}
		1
			&\mbox{if}\quad p>1+2r\ \mbox{or}\ p=1+2r, r>1,\\
		\log \left( 2+t \right)
			&\mbox{if}\quad p = 1+2r, r=1,\\
		\langle t \rangle^{1-\frac{1}{2r}(p-1)}
			&\mbox{if}\quad p< 1+ 2r,
	\end{cases}
\end{align*}
where
$\gamma = 1$ if $p=1+2r, r>1$ and $\gamma= r$ otherwise.
In the same way, we have
\begin{align*}
	&\left\| \int_0^t \mathcal{D}(t-\tau) \psi(\tau) \,d\tau \right\|_{L^2} \\
	&\quad \lesssim
	\| \psi \|_{Y(T)}
	\langle t \rangle^{-\frac{1}{2}\left( \frac{1}{r}-\frac{1}{2}\right) }
	\begin{cases}
		1
			&\mbox{if}\quad p>1+2r\ \mbox{or}\ p=1+2r, r>1,\\
		\log \left( 2+t \right)
			&\mbox{if}\quad p = 1+2r, r=1,\\
		\langle t \rangle^{1-\frac{1}{2r}(p-1)}
			&\mbox{if}\quad p< 1+ 2r.
	\end{cases}
\end{align*}
Finally, we estimate
\begin{align*}
	& \left\| |\cdot|^{\alpha} \int_0^t \mathcal{D}(t-\tau) \psi (\tau)\,d\tau \right\|_{L^2} \\
	&\quad \lesssim
	\int_0^t \langle t-\tau \rangle^{%
		-\frac{n}{2}\left( \frac{1}{\gamma}-\frac{1}{2} \right) + \frac{\alpha}{2} }%
		\| \psi (\tau) \|_{L^{\gamma}}\, d\tau
		+ \int_0^t \langle t-\tau \rangle^{-\frac{1}{4} }%
		\| |\cdot |^{\alpha} \psi(\tau) \|_{L^1}\,d\tau \\
	&\qquad + \int_0^t e^{-\frac{t-\tau}{4}}
		\| \langle \cdot \rangle^{\alpha} \langle \nabla \rangle^{-1} \psi (\tau) \|_{L^2}\,d\tau \\
	&\quad =: V_1 + V_2 + V_3.
\end{align*}
Lemma \ref{lem_sob} leads to
\[	
	V_3 \lesssim \int_0^t e^{-\frac{t-\tau}{4}}
		\| \langle \cdot \rangle^{\alpha} \psi (\tau) \|_{L^q}\,d\tau.
\]
Therefore, the estimates of $V_2, V_3$ reduce to that of
$ \int_0^t \langle t-\tau \rangle^{-\frac{1}{4}}
\| \langle \cdot \rangle^{\alpha} \psi(\tau) \|_{L^1}\,d\tau$.
We have
\begin{align*}
	\int_0^t \langle t-\tau \rangle^{-\frac{1}{4}}
			\| \langle \cdot \rangle^{\alpha} \psi(\tau) \|_{L^q}\,d\tau
	 \lesssim \| \psi \|_{Y(T)}
		\int_0^t \langle t-\tau \rangle^{-\frac{1}{4}}
		\langle \tau \rangle^{-\zeta}\, d\tau,
\end{align*}
where $\zeta$ is defined in Table 2 (see Section \ref{se_no}).
We compute
\[
	\int_0^t \langle t-\tau \rangle^{%
		-\frac{n}{2}\left( \frac{1}{q} - \frac{1}{2} \right)}%
		\langle \tau \rangle^{-\zeta}\, d\tau
	\lesssim
	\begin{cases}
	\langle t \rangle^{%
		-\frac{1}{2}\left( \frac{1}{r}-\frac{1}{2} \right)+\frac{\alpha}{2}%
		+1 -\frac{n}{2r}(p-1)}%
	&(\zeta < 1),\\
	\langle t \rangle^{-\frac{1}{4}} \log (1+t)
	&(\zeta =1),\\
	\langle t \rangle^{-\frac{1}{4}}
	&(\zeta >1)
	\end{cases}
\]	
and note that
$\zeta <1$ holds if $p < 1+2r$.
Therefore, we may summarize them as
\[
	\int_0^t \langle t-\tau \rangle^{%
		-\frac{n}{2}\left( \frac{1}{q} - \frac{1}{2} \right)}%
		\langle \tau \rangle^{-\zeta}\, d\tau
	\lesssim
	\begin{cases}
	\langle t \rangle^{%
		-\frac{n}{2}\left( \frac{1}{r}-\frac{1}{2} \right)+\frac{\alpha}{2}%
		+1 -\frac{n}{2r}(p-1)}%
	&(p < 1+\frac{2r}{n}),\\
	\langle t \rangle^{%
		-\frac{n}{2}\left( \frac{1}{r}-\frac{1}{2} \right)+\frac{\alpha}{2}}%
	&(p\ge 1+\frac{2r}{n})
	\end{cases}
\]
and hence, we have
\[
	V_2 + V_3 \lesssim
	\| \psi \|_{Y(T)}
	\begin{cases}
	\langle t \rangle^{%
		-\frac{n}{2}\left( \frac{1}{r}-\frac{1}{2} \right)+\frac{\alpha}{2}%
		+1 -\frac{n}{2r}(p-1)}%
	&(p < 1+\frac{2r}{n}),\\
	\langle t \rangle^{%
		-\frac{n}{2}\left( \frac{1}{r}-\frac{1}{2} \right)+\frac{\alpha}{2}}%
	&(p\ge 1+\frac{2r}{n}).
	\end{cases}
\]
For $V_1$, taking
$\gamma = 1$ if $p=1+2r, r>1$ and  $\gamma= r$ otherwise,
we see that
\begin{align*}
	V_1 \lesssim
	\| \psi \|_{Y(T)}
	\langle t \rangle^{-\frac{1}{2}\left( \frac{1}{r}-\frac{1}{2}\right) +\frac{\alpha}{2} }
	\begin{cases}
		1
			&\mbox{if}\quad p>1+2r\ \mbox{or}\ p=1+2r, r>1,\\
		\log \left( 2+t \right)
			&\mbox{if}\quad p = 1+2r, r=1,\\
		\langle t \rangle^{1-\frac{1}{2r}(p-1)}
			&\mbox{if}\quad p< 1+ 2r.
	\end{cases}
\end{align*}
This completes the proof.
\end{proof}

\begin{lemma}\label{lem_nl1}
Under the assumptions in Theorem \ref{thm_lwp}, we have
\begin{align*}
	\| \mathcal{N}(u) \|_{Y(T)} & \lesssim \| u \|_{X(T)}^p,\\
	\| \mathcal{N}(u) - \mathcal{N}(v) \|_{Y(T)}
		&\lesssim \| u-v \|_{X(T)} \left(  \| u \|_{X(T)} + \| v \|_{X(T)} \right)^{p-1}.
\end{align*}
\end{lemma}
\begin{proof}
At first, we consider
$\| \langle \cdot \rangle^{\alpha} \mathcal{N}(u) \|_{L^1}$.
The H\"{o}lder inequality yields
\begin{align*}
	\langle t \rangle^{\zeta} \| \langle \cdot \rangle^{\alpha}  \mathcal{N}(u) \|_{L^1}
	\lesssim \langle t \rangle^{%
				\frac{1}{2}\left(\frac{1}{r}-\frac{1}{2}\right)-\frac{\alpha}{2}}%
	\| \langle \cdot \rangle^{\alpha} u \|_{L^2}
			\left( \langle t \rangle^{%
				\frac{1}{2}\left( \frac{1}{r}-\frac{1}{2(p-1)}\right)}%
			\| u \|_{L^{2(p-1)}} \right)^{p-1}.
\end{align*}
Since $1+\frac{r}{2} \le p$, we see that $2(p-1) \ge r$.
Moreover, when $n>2s$, the assumption
$p \le \frac{1}{1-2s}$ implies $2(p-1) \le \frac{2n}{n-2s}$.
Therefore, we apply Lemma \ref{lem_ip} to obtain
\[
	\langle t \rangle^{\zeta} \| \langle \cdot \rangle^{\alpha}  \mathcal{N}(u) \|_{L^1}
		\lesssim \| u \|_{X(T)}^p.
\]
Next, we estimate
$\| \mathcal{N}(u) \|_{L^{\gamma}}$ with $\gamma \in [1,2]$.
We first obtain
\[
	\| \mathcal{N}(u) \|_{L^{\gamma}}
	\lesssim \| u \|_{L^{p \gamma}}^p.
\]
It follows from $p \ge 1+\frac{r}{2}$ that $p \gamma \ge r$ for any $\gamma \in [1,2]$.
Also, when $n > 2s$, the assumption $p \le \frac{1}{1-2s}$ ensures
$p \gamma \le \frac{2n}{n-2s}$.
Hence, we apply Lemma \ref{lem_ip} to derive
\[
	\langle t \rangle^{\frac{1}{2}\left( \frac{p}{r}-\frac{1}{\gamma}\right)}
	\| \mathcal{N}(u) \|_{L^{\gamma}}
	\lesssim
	\left( \langle t \rangle^{\frac{1}{2}\left( \frac{1}{r}-\frac{1}{p\gamma}\right)}
	\| u \|_{L^{p \gamma}} \right)^p
	\lesssim \| u \|_{X(T)}^p.
\]
The above estimates shows the conclusion.
\end{proof}

\begin{proof}[Proof of Theorems \ref{thm_lwp} and \ref{thm_gwp} when $n=1$]
The proof of Theorems \ref{thm_lwp} and \ref{thm_gwp}
is completely the same as the case $n\ge 2$ and
we omit the detail.
\end{proof}

\section{Asymptotic behavior of the global solution}
\subsection{Approximation by an inhomogeneous heat equation}
In this section, we study the asymptotic behavior of the global solution.
Let $u$ be the global-in-time solution proved in the previous section, that is,
\[
	u(t)= \tilde{\mathcal{D}}(t)\varepsilon u_0
		+ \mathcal{D}(t)\varepsilon u_1
		+ \int_0^t \mathcal{D}(t-\tau)\mathcal{N}(u(\tau))\,d\tau.
\]
First, we consider the solution of the inhomogeneous linear heat equation
\begin{align}
\label{inhe}
	( \partial_t - \Delta )v = \mathcal{N}(u)
\end{align}
with the initial data
$v(0,x)=\varepsilon (u_0+u_1)$, that is,
\[
	v(t)=
	\mathcal{G}(t) \left( \varepsilon u_0 + \varepsilon u_1 \right)
	+ \int_0^t \mathcal{G}(t-\tau) \mathcal{N}(u(\tau))\,d\tau,
\]
where $\mathcal{G}(t)$ is defined in \eqref{g}.
We first prove that the asymptotic profile of $u$ is given by $v$
in $H^{s,0}\cap H^{0,\alpha}$-sense.
\begin{proposition}\label{prop_asy}
Under the assumptions of Theorem \ref{thm_gwp},
we have
\begin{align}
\label{esdifs}
	\| |\nabla|^s ( u(t) - v(t) ) \|_{L^2}
	& \lesssim \langle t \rangle^{%
		-\frac{n}{2} \left( \frac{1}{r}-\frac{1}{2}\right)-\frac{s}{2}%
		- \min\{ 1 , \frac{n}{2r}(p-1) - \frac{1}{2} \}},\\
\label{esdif}
	\| u(t) - v(t) \|_{L^2}
	& \lesssim \langle t \rangle^{%
		-\frac{n}{2} \left( \frac{1}{r}-\frac{1}{2}\right)
		- \min\{ 1 , \frac{n}{2r}(p-1) - \frac{1}{2} \} },\\
\label{esdifa}
	\| |\cdot|^{\beta} ( u(t) - v(t) ) \|_{L^2}
	& \lesssim \langle t \rangle^{%
		 - \frac{n}{2} \left( \frac{1}{r}-\frac{1}{2}\right) + \frac{\beta}{2}
		-\min\{ 1 , \frac{n}{2r}(p-1) - \min\{ \frac{n}{4}, \frac{1}{2} \} \} },
\end{align}
where
$\beta$ is arbitrary number satisfying
$n(\frac1r - \frac12) < \beta \le \alpha$.
\end{proposition}

\subsection{Preliminary estimates}
In order to prove Proposition \ref{prop_asy}, we prepare the following lemma.
\begin{lemma}\label{lem_dif}
Let
$\gamma, \nu \in[1,2]$, $\beta \ge 0$ and $s_1 \ge s_2 \ge 0$.
Then, we have
\begin{align*}
	\| |\nabla |^{s_1} ( \mathcal{D}(t)\psi - \mathcal{G}(t) \psi ) \|_{L^2}
	&\lesssim
	\langle t \rangle^{-\frac{s_1-s_2}{2}-\frac{n}{2}\left(\frac{1}{\gamma}-\frac{1}{2}\right)-1}
	\| |\nabla|^{s_2} \psi\|_{L^{\gamma}}
	+e^{-\frac{t}{4}}
	\| |\nabla|^{s_1} \langle \nabla \rangle^{-1} \psi \|_{L^2},\\
	\| |\cdot |^{\beta} ( \mathcal{D}(t)\psi - \mathcal{G}(t) \psi ) \|_{L^2}
	&\lesssim
	\langle t \rangle^{\frac{\beta}{2}-\frac{n}{2}\left(\frac{1}{\gamma}-\frac{1}{2}\right)-1}
	\|\psi\|_{L^{\gamma}}
	+\langle t \rangle^{-\frac{n}{2}\left(\frac{1}{\nu}-\frac{1}{2}\right)-1}
	\| |\cdot|^{\beta} \psi \|_{L^{\nu}} \\
	&\quad +e^{-\frac{t}{4}}
	\| \langle \cdot \rangle^{\beta} \langle \nabla \rangle^{-1} \psi \|_{L^2}
\end{align*}
for any $t \ge 1$.
\end{lemma}
\begin{proof}
Since
\[
	\| |\xi|^{s} ( e^{-\frac{t}{2}} L(t,\xi)-e^{- t |\xi|^2} ) \|_{L^{\gamma}(|\xi|\le 1)}
	\lesssim \langle t \rangle^{-\frac{s}{2}-\frac{n}{2\gamma}-1}
\]
and
\[
	\| \langle \xi \rangle L(t,\xi) \|_{L^{\infty}(|\xi| > 1 )} \lesssim 1,
	\quad
	\| \langle \xi \rangle e^{-t |\xi|^2} \|_{L^{\infty}(|\xi| > 1 )} \lesssim e^{-t/4},
\]
we have
\begin{align*}
	\| |\nabla|^{s_1} ( \mathcal{D}(t)\psi - \mathcal{G}(t) \psi ) \|_{L^2}
	&=
	\| |\xi|^{s_1} ( e^{-\frac{t}{2}}L(t,\xi)-e^{-t|\xi|^2}) \hat{\psi} \|_{L^2}\\
	&\lesssim
	\| |\xi|^{s_1-s_2}( e^{-\frac{t}{2}}L(t,\xi)-e^{-t|\xi|^2})
		\|_{L^{\frac{2\gamma}{2-\gamma}}(|\xi|\le 1)}
		\| |\xi|^{s_2} \hat{\psi} \|_{L^{\frac{\gamma}{\gamma-1}}(|\xi|\le 1)}\\
	&\quad
	+e^{-\frac{t}{2}}\| \langle \xi \rangle L(t,\xi) \|_{L^{\infty}(|\xi| >1)}
		\| |\xi|^{s_1}\langle \xi \rangle^{-1} \hat{\psi} \|_{L^2(|\xi|>1)}\\
	&\quad
	+e^{-\frac{t}{4}} \|\langle \xi \rangle e^{-\frac{t}{4}|\xi|^2} \|_{L^{\infty}(|\xi|>1)}
		\| |\xi|^{s_1}\langle \xi \rangle^{-1} \hat{\psi} \|_{L^2(|\xi|>1)}\\
	&\lesssim
	\langle t \rangle^{-\frac{s_1-s_2}{2}-\frac{n}{2}\left(\frac{1}{\gamma}-\frac{1}{2}\right)-1}
	\| |\nabla|^{s_2} \psi \|_{L^{\gamma}}
	+ e^{-\frac{t}{4}}
	\| |\nabla|^{s_1} \langle \nabla \rangle^{-1} \psi \|_{L^2},
\end{align*}
which implies the first assertion.

To prove the second estimate, we take a cut-off function
$\chi(\xi) \in C_0^{\infty}(\mathbb{R}^n)$
satisfying $\chi(\xi) = 1$ for $|\xi|\le 1$ and $\chi(\xi) = 0$ for $|\xi| \ge 2$.
We put
\[
	\bar{K}(t,x) := \mathcal{F}^{-1} \chi(\xi) \left( e^{-\frac{t}{2}} L(t,\xi) - e^{-t|\xi|^2} \right).
\]
Then, in the same way to \eqref{ine_K}, we see that
\begin{align*}
	&\left\| |\cdot|^{\beta} \mathcal{F}^{-1} \chi(\xi)
		\left( e^{-\frac{t}{2}} L(t,\xi) - e^{-t|\xi|^2} \right) \hat{\psi} \right\|_{L^2} \\
	&\quad \lesssim
		\| |\cdot|^{\beta} \bar{K}(t) \|_{L^{\frac{2\gamma}{3\gamma-2}}}
		\| \psi \|_{L^{\gamma}}
	+ \| \bar{K}(t) \|_{L^{\frac{2\nu}{3\nu-2}}} \| |\cdot|^{\beta} \psi \|_{L^{\nu}}.
\end{align*}
Moreover, in a similar way to \eqref{est_K}
with
\[
	e^{-\frac{t}{2}} L(t,\xi) - e^{-t |\xi|^2}
	= e^{-t|\xi|^2} O(|\xi|^2)
\]
as $|\xi| \to 0$, we can prove
\begin{align*}
\label{est_bK}
	\| |\cdot|^{\beta} \bar{K}(t) \|_{L^k}
	\lesssim \langle t \rangle^{\frac{\beta}{2}-\frac{n}{2}\left( 1-\frac{1}{k} \right)-1},
\end{align*}
which implies
\begin{align*}
	&\left\| |\cdot|^{\beta} \mathcal{F}^{-1} \chi(\xi)
		\left( e^{-\frac{t}{2}} L(t,\xi) - e^{-t|\xi|^2} \right) \hat{\psi} \right\|_{L^2} \\
	&\quad \lesssim
	\langle t \rangle^{-\frac{n}{2}\left( \frac{1}{\gamma}-\frac{1}{2} \right)+\frac{\beta}{2}-1}
	\| \psi \|_{L^{\gamma}}
	+ \langle t \rangle^{ - \frac{n}{2}\left( \frac{1}{\nu}-\frac{1}{2} \right)-1}
	\| |\cdot|^{\beta} \psi \|_{L^{\nu}}.
\end{align*}
We also easily obtain
\[
	\| |\cdot|^{\beta} \mathcal{F}^{-1}\left(
		e^{-t |\xi|^2} (1-\chi(\xi)) \hat{\psi} \right) \|_{L^2}
	\lesssim e^{-\frac{t}{4}}
		\| \langle \cdot \rangle^{\beta} \langle \nabla \rangle^{-1} \psi \|_{L^2}.
\]
This and \eqref{est_lc2} lead to
\[
	\left\| |\cdot|^{\beta} \mathcal{F}^{-1}
		\left( e^{-\frac{t}{2}}L(t,\xi) - e^{-t |\xi|^2} \right) (1-\chi(\xi)) \hat{\psi} \right\|_{L^2}
	\lesssim e^{-\frac{t}{4}}
	\left\| \langle \cdot \rangle^{\beta} \langle \nabla \rangle^{-1} \psi  \right\|_{L^2}.
\]
From this, we reach the conclusion.
\end{proof}

In the same way, we have the following.
\begin{lemma}\label{lem_dif2}
Let $\gamma, \nu \in [1,2]$, $\beta \ge 0$, $s_1 \ge s_2 \ge 0$.
Then, we have
\begin{align}
\label{est_der4}
	\| |\nabla|^{s_1} (\tilde{\mathcal{D}}(t) - \mathcal{G}(t)) \psi \|_{L^2}
	&\lesssim
	\langle t \rangle^{-\frac{n}{2}\left( \frac{1}{\gamma}-\frac{1}{2}\right)-\frac{s_1-s_2}{2}-1}
		\| |\nabla|^{s_2} \psi \|_{L^{\gamma}}
		+ e^{-t/4} \| |\nabla|^{s_1} \psi \|_{L^2},\\
\nonumber
	\| | \cdot |^{\beta} ( \tilde{\mathcal{D}}(t) - \mathcal{G}(t) ) \psi \|_{L^2}
	&\lesssim
	\langle t \rangle^{-\frac{n}{2}\left( \frac{1}{\gamma}-\frac{1}{2} \right)+\frac{\beta}{2}-1}
		\| \psi \|_{L^{\gamma}}
	+\langle t \rangle^{-\frac{n}{2}\left( \frac{1}{\nu}-\frac{1}{2}\right)-1}
		\| |\cdot|^{\beta} \psi \|_{L^{\nu}}\\
\label{est_wei4}
	&\quad + e^{-\frac{t}{4}} \| \langle \cdot \rangle^{\beta} \psi \|_{L^2}.
\end{align}
\end{lemma}

\subsection{Proof of Proposition \ref{prop_asy}}

For simplicity, we treat only the case $n\ge 2$.
The proof of the case $n=1$ is similar.
Let us estimate the nonlinear part of $u-v$.
We divide
\[
	\left\| |\nabla|^s \int_0^t \left( \mathcal{D}(t-\tau)\mathcal{N}(u(\tau))
		-\mathcal{G}(t-\tau) \mathcal{N}(u(\tau))\right)\,d\tau \right\|_{L^2}
	\le I + I\!I,
\]
where
\[
	I= \int_0^{t/2}\left\| |\nabla|^s \left( \mathcal{D}(t-\tau)-\mathcal{G} (t-\tau)\right)
		\mathcal{N}(u(\tau)) \right\|_{L^2} \,d\tau 
\]
and
\[
	I\!I= \int_{t/2}^t \left\| |\nabla|^s  \left(
		\mathcal{D}(t-\tau)-\mathcal{G}(t-\tau) \right)
		\mathcal{N}(u(\tau)) \right\|_{L^2} \,d\tau .
\]
We claim that each term has the desired decay rate.

First, we estimate $I\!I$.
Applying Lemma \ref{lem_dif} with $s_1 = s$, $s_2 = [s]$ and $\gamma = \rho$, we have
\begin{align*}
	I\!I
	&\lesssim \int_{\frac{t}{2}}^t
		\| |\nabla|^{s} (\mathcal{D}(t-\tau) - \mathcal{G}(t-\tau))
			\mathcal{N}(u(\tau))\|_{L^2}\,d\tau \\
	&\lesssim \int_{\frac{t}{2}}^t
		\langle t- \tau \rangle^{-\frac{s-[s]}{2}-\frac{n}{2}\left(\frac{1}{\rho}-\frac{1}{2}\right)-1}
		\| |\nabla|^{[s]} \mathcal{N}(u(\tau)) \|_{L^{\rho}}\, d\tau \\
	&\quad + \int_{\frac{t}{2}}^t e^{-\frac{t-\tau}{4}}
				\| |\nabla|^s \langle \nabla \rangle^{-1} \mathcal{N}(u(\tau)) \|_{L^2}\, d\tau.
\end{align*}
Using the Sobolev inequality \eqref{sobs}, we see that
\begin{align}
\label{esdifII1}
	I\!I &\lesssim \int_{\frac{t}{2}}^t
		\langle t- \tau \rangle^{-\frac{n}{2}\left(\frac{1}{\rho}-\frac{1}{2}\right)-\frac{s-[s]}{2}-1}
		\| |\nabla|^{[s]} \mathcal{N}(u(\tau)) \|_{L^{\rho}}\, d\tau \\
\nonumber
	&\lesssim \| \mathcal{N}(u) \|_{Y(\infty)}
		\int_{\frac{t}{2}}^t
		\langle t- \tau \rangle^{-\frac{n}{2}\left(\frac{1}{\rho}-\frac{1}{2}\right)-\frac{s-[s]}{2}-1}
		\langle \tau \rangle^{-\eta}\, d\tau \\
\nonumber
	&\lesssim \langle t \rangle^{%
			-\frac{n}{2}\left(\frac{1}{r}-\frac{1}{2}\right)-\frac{s}{2} 
				+ \frac{1}{2} - \frac{n}{2r}(p-1)}.%
\end{align}
Next, we give an estimate of $I$.
Applying Lemma \ref{lem_dif} with $s_1=s$, $s_2=0$ and $\gamma \in [1,2]$
determined later, we have
\begin{align*}
	I &\lesssim
		\int_0^{\frac{t}{2}} \langle t-\tau \rangle^{%
			-\frac{n}{2}\left( \frac{1}{\gamma}-\frac{1}{2} \right) - \frac{s}{2} - 1}%
		\| \mathcal{N}(u(\tau)) \|_{L^{\gamma}}\, d\tau
	+\int_0^{\frac{t}{2}} e^{-\frac{t-\tau}{4}}
		\| |\nabla|^s \langle \nabla \rangle^{-1} \mathcal{N}(u(\tau)) \|_{L^2}\,d\tau \\
	&=: I_1 + I_2.
\end{align*}
In the same way as the proof of Lemma \ref{lem_duh}, it is easy to see that
$I_2 \lesssim e^{-\frac{t}{8}} \| \mathcal{N}(u) \|_{Y(\infty)}$.
Also, for $I_1$, taking
$\gamma = r$ if $p>1+\frac{2r}{n}$ and $\gamma = \sigma_1$ if $p=1+\frac{2r}{n}$, $r>1$,
respectively, we can prove
\begin{align}
\label{esdifI1}
	I_1 \lesssim \| \mathcal{N}(u) \|_{Y(\infty)}
		\langle t \rangle^{-\frac{n}{2}\left( \frac{1}{r}-\frac{1}{2} \right) - \frac{s}{2} -1 }.
\end{align}
Combining the estimates \eqref{esdifII1} and \eqref{esdifI1}, we conclude
\begin{align*}
	&\left\| |\nabla|^s \int_0^t \left( \mathcal{D}(t-\tau)\mathcal{N}(u(\tau))
		-\mathcal{G}(t-\tau) \mathcal{N}(u(\tau))\right)\,d\tau \right\|_{L^2}
		\lesssim \langle t \rangle^{%
		-\frac{n}{2}\left( \frac{1}{r}-\frac{1}{2} \right) - \frac{s}{2}
		- \min \{ 1, \frac{n}{2r}(p-1)-\frac{1}{2}\} }.%
\end{align*}
The estimate of the linear part of $u-v$ is obvious from
Lemmas \ref{lem_dif} and \ref{lem_dif2}.
Thus, we obtain the first assertion \eqref{esdifs}.

Secondly, we prove \eqref{esdif}.
As before, we divide the nonlinear term of $u-v$ as
\begin{align*}
	&\left\| \int_0^t (\mathcal{D}(t-\tau) - \mathcal{G}(t-\tau)) \mathcal{N}(u(\tau)) \,d\tau \right\|_{L^2} \\
	&\quad \lesssim
		\int_0^{\frac{t}{2}}
		\| (\mathcal{D}(t-\tau)-\mathcal{G}(t-\tau))\mathcal{N}(u(\tau))\|_{L^2}\, d\tau 
		+ \int_{\frac{t}{2}}^t
		\| (\mathcal{D}(t-\tau) -\mathcal{G}(t-\tau))\mathcal{N}(u(\tau)) \|_{L^2}\, d\tau \\
	&=: I\!I\!I + I\!V.
\end{align*}
Making use of Lemma \ref{lem_dif} with $s_1=s_2=0$ and $\gamma =q$,
and then using the Sobolev embedding
$\| \langle \nabla \rangle^{-1} \psi \|_{L^2} \lesssim \| \psi \|_{L^q}$,
we see that
\begin{align*}
	I\!V & \lesssim
		\int_{\frac{t}{2}}^t \left[ \langle t-\tau \rangle^{%
			-\frac{n}{2}\left( \frac{1}{q} - \frac{1}{2} \right)-1 }%
		\| \mathcal{N}(u(\tau)) \|_{L^q}
		+ e^{-\frac{t-\tau}{4}} \| \langle \nabla \rangle^{-1} \mathcal{N}(u(\tau)) \|_{L^2} 
		\right] \,d\tau \\
	&\lesssim \int_{\frac{t}{2}}^t \langle t-\tau \rangle^{%
			-\frac{n}{2}\left( \frac{1}{q} - \frac{1}{2} \right)-1 }%
		\| \mathcal{N}(u(\tau)) \|_{L^q}\, d\tau.
\end{align*}
Noting $-\frac{n}{2}\left( \frac{1}{q} - \frac{1}{2} \right)-1 = -\frac{3}{2} < -1$,
we further estimate
\begin{align}
\label{esdifIV}
	I\!V &\lesssim
		\| \mathcal{N}(u) \|_{Y(\infty)}
		\int_{\frac{t}{2}}^t \langle t-\tau \rangle^{%
			-\frac{n}{2}\left( \frac{1}{q} - \frac{1}{2} \right)-1 }%
			\langle \tau \rangle^{-\frac{n}{2}\left( \frac{p}{r} -\frac{1}{q} \right)}\,d\tau \\
\nonumber
	&\lesssim
		\langle t \rangle^{%
		-\frac{n}{2} \left( \frac{1}{r} - \frac{1}{2} \right) + \frac{1}{2} - \frac{n}{2r}(p-1)}.%
\end{align}
On the other hand, applying Lemma \ref{lem_dif} with $s_1=s_2=0$ and
$\gamma \in [1,2]$ determined later, we deduce
\begin{align*}
	I\!I\!I &\lesssim \int_0^{\frac{t}{2}}
		\langle t-\tau \rangle^{-\frac{n}{2}\left( \frac{1}{\gamma} - \frac{1}{2} \right)-1 }
		\| \mathcal{N}(u(\tau)) \|_{L^{\gamma}}\,d\tau
		+ \int_0^{\frac{t}{2}} e^{-\frac{t-\tau}{4}}
			\| \langle \nabla \rangle^{-1} \mathcal{N}(u(\tau)) \|_{L^2}\,d\tau \\
	&=: I\!I\!I_1 + I\!I\!I_2.
\end{align*}
In the same manner as the proof of Lemma \ref{lem_duh},
we see that
the term $I\!I\!I_2$ is bounded by $e^{-\frac{t}{8}} \| \mathcal{N}(u) \|_{Y(\infty)}$.
Furthermore, taking $\gamma = r$ if $p>1+\frac{2r}{n}$ and
$\gamma = \sigma_1$ if $p=1+\frac{2r}{n}$, $r>1$, respectively,
we have
\begin{align}
\label{esdifIII}
	I\!I\!I_1 &\lesssim
	\| \mathcal{N}(u) \|_{Y(\infty)}
	\langle t \rangle^{-\frac{n}{2} \left( \frac{1}{r} - \frac{1}{2} \right)-1}.
\end{align}
Summing up the estimates \eqref{esdifIV} and \eqref{esdifIII},
we conclude
\begin{align*}
	\left\| \int_0^t (\mathcal{D}(t-\tau) - \mathcal{G}(t-\tau)) \mathcal{N}(u(\tau)) \,d\tau \right\|_{L^2}
	\lesssim
		\langle t \rangle^{%
		-\frac{n}{2} \left( \frac{1}{r} - \frac{1}{2} \right)
		-\min\{ 1, \frac{n}{2r}(p-1) - \frac{1}{2}\} }.%
\end{align*}
The linear part of $u-v$ is easily estimated
from Lemmas \ref{lem_dif} and \ref{lem_dif2} and
we reach the second assertion \eqref{esdif}.

Finally, we prove \eqref{esdifa}.
We take $\beta$ satisfying
$n(\frac1r-\frac12) < \beta \le \alpha$.
Then, by the interpolation
\begin{align*}
	\| \langle \cdot \rangle^{\beta} \psi \|_{L^q}
	\le \| \langle \cdot \rangle^{\alpha} \psi \|_{L^q}^{\frac{\beta}{\alpha}}
		\| \psi \|_{L^q}^{1-\frac{\beta}{\alpha}},
\end{align*}
we deduce that the solution $u$ satisfies
\begin{align*}
	\langle t \rangle^{\zeta_{\beta}} \| \langle \cdot \rangle^{\beta} \mathcal{N} (u(t)) \|_{L^q}
	\lesssim
		\left( \langle t \rangle^{\zeta} 
		\| \langle \cdot \rangle^{\alpha} \mathcal{N} (u(t)) \|_{L^q} \right)^{\frac{\beta}{\alpha}}
		\left( \langle t \rangle^{\frac{n}{2}\left(\frac{p}{r}-\frac{1}{q}\right)}
		\| \mathcal{N} (u(t)) \|_{L^q} \right)^{1-\frac{\beta}{\alpha}}
	\lesssim 1
\end{align*} 
with
$\zeta_{\beta} = \frac{n}{2r}(p-1)-\frac12 + \frac{n}{2}(\frac1r - \frac12) - \frac{\beta}{2}$.
From Lemma \ref{lem_dif} with
$\nu = q$ and $\gamma \in [1,2]$ determined later,
we obtain
\begin{align*}
	&\left\| |\cdot|^{\beta} \int_0^t
		( \mathcal{D}(t-\tau)-\mathcal{G}(t-\tau)) \mathcal{N}(u(\tau)) \,d\tau \right\|_{L^2}\\
	&\quad \lesssim
	\int_0^t \langle t-\tau \rangle^{%
		-\frac{n}{2}\left( \frac{1}{\gamma}-\frac{1}{2} \right) + \frac{\beta}{2}-1 }%
		\| \mathcal{N}(u(\tau)) \|_{L^{\gamma}}\, d\tau +  \int_0^t \langle t-\tau \rangle^{%
		-\frac{n}{2}\left( \frac{1}{q} - \frac{1}{2} \right) -1}%
		\| |\cdot |^{\beta} \mathcal{N}(u(\tau)) \|_{L^q}\,d\tau \\
	&\qquad + \int_0^t e^{-\frac{t-\tau}{4}}
		\| \langle \cdot \rangle^{\beta} \langle \nabla \rangle^{-1}
				\mathcal{N}(u(\tau)) \|_{L^2}\,d\tau \\
	&\quad =: V_1 + V_2 + V_3.
\end{align*}
Lemma \ref{lem_sob} implies
\[
	V_3 \lesssim \int_0^t e^{-\frac{t-\tau}{4}}
		\| \langle \cdot \rangle^{\beta} \mathcal{N}(u(\tau)) \|_{L^q}\, d\tau
\]
and hence, we have
\[
	V_2 + V_3 \lesssim \int_0^t \langle t-\tau \rangle^{%
		-\frac{n}{2}\left( \frac{1}{q} - \frac{1}{2} \right) -1}%
		\| \langle \cdot \rangle^{\beta} \mathcal{N}(u(\tau)) \|_{L^q}\,d\tau.
\]
Noting
$-\frac{n}{2}(\frac{1}{q}-\frac{1}{2})-1 = -\frac{3}{2}$,
we proceed the estimate as
\begin{align*}
	V_2 + V_3
	& \lesssim
		\| \mathcal{N}(u)\|_{Y(\infty)} \int_0^t
			\langle t-\tau \rangle^{-\frac{3}{2}}
			\langle \tau \rangle^{-\zeta_{\beta}}\,d\tau \\
	& \lesssim
		\langle t \rangle^{%
			-\frac{n}{2}\left( \frac{1}{r}-\frac{1}{2} \right) + \frac{\beta}{2}%
			-\frac{n}{2r}(p-1) + \frac{1}{2}}%
	+
	\begin{cases}
		\langle t \rangle^{%
			-\frac{n}{2}\left( \frac{1}{r}-\frac{1}{2} \right) + \frac{\beta}{2}%
			-\frac{n}{2r}(p-1)}%
		&(\zeta_{\beta} <1),\\
		\langle t \rangle^{-\frac{3}{2}} \log (1+t)
		&(\zeta_{\beta} =1),\\
		\langle t \rangle^{-\frac{3}{2}}
		&(\zeta_{\beta} > 1).
	\end{cases}
\end{align*}
For $V_1$, as in the proof of Lemma \ref{lem_duh},
taking $\gamma = r$ if $p>1+\frac{2r}{n}$ and
$\gamma = \sigma_1$ if $p=1+\frac{2r}{n}$, $r>1$, respectively,
we can see that
\begin{align*}
	V_1 \lesssim
		\langle t \rangle^{%
			-\frac{n}{2}\left( \frac{1}{r}-\frac{1}{2} \right) + \frac{\beta}{2} -1 }.%
\end{align*}
Consequently, we obtain
\[
	 \left\| |\cdot|^{\beta} \int_0^t
		( \mathcal{D}(t-\tau)-\mathcal{G}(t-\tau)) \mathcal{N}(u(\tau)) \,d\tau \right\|_{L^2}
	\lesssim \langle t \rangle^{%
		-\frac{n}{2}\left( \frac{1}{r}-\frac{1}{2} \right) + \frac{\beta}{2}
		- \min\{ 1, \frac{n}{2r}(p-1) - \frac{1}{2} \}},%
\]
which shows the third assertion \eqref{esdifa} and
finishes the proof of Proposition \ref{prop_asy}.

\subsection{Proof of Theorem \ref{thm_asy}}
By virtue of Proposition \ref{prop_asy}, the proof of Theorem \ref{thm_asy}
reduces to the analysis of the asymptotic behavior of solutions to
the inhomogeneous heat equation \eqref{inhe} with the initial data
$\varepsilon (u_0 + u_1)$.
In this subsection, we prove the following:
\begin{proposition}\label{prop_asy2}
Under the assumption of Theorem \ref{thm_asy},
Let $\beta$ satisfy
$n(\frac1r-\frac12) < \beta \le \alpha$.
Then, the solution $v$ to the equation \eqref{inhe} with the initial data
$\varepsilon (u_0 + u_1)$ satisfies
the following asymptotic behavior:
When $r>1$, we have
\begin{align}
\label{esdifhes}
	\| |\nabla|^s (v(t) - \varepsilon \mathcal{G}(t) (u_0+u_1) )\|_{L^2}
	&\lesssim
	\langle t \rangle^{-\frac{n}{2}\left( \frac{1}{r}-\frac{1}{2} \right) - \frac{s}{2}%
		-\min\{ \frac{n}{2}\left(1-\frac{1}{r}\right), \frac12, \frac{n}{2r}(p-1)-1 \} %
		+ \varsigma}, \\
\label{esdifhe}
	\| v(t) - \varepsilon \mathcal{G}(t) (u_0+u_1) \|_{L^2}
	&\lesssim \langle t \rangle^{-\frac{n}{2}\left( \frac{1}{r}-\frac{1}{2} \right)
	-\min\{ \frac{n}{2}\left(1-\frac{1}{r}\right), \frac12, \frac{n}{2r}(p-1)-1 \} %
		+ \varsigma}, \\
\label{esdifhea}
	\| |\cdot|^{\beta} (v(t) - \varepsilon \mathcal{G}(t) (u_0+u_1) ) \|_{L^2}
	&\lesssim \langle t \rangle^{-\frac{n}{2}\left( \frac{1}{r}-\frac{1}{2} \right) + \frac{\beta}{2}
	-\min\{ \frac{n}{2}\left(1-\frac{1}{r}\right), \frac12, \frac{n}{2r}(p-1)-1 \} %
		+ \varsigma}
\end{align}
for $t \ge 1$.
When $r=1$, we have
\begin{align}
\label{ot1}
	\| v(t) - \theta G(t) \|_{L^m} \lesssim \langle t \rangle^{%
		-\frac{n}{2}\left(1-\frac{1}{m}\right) %
		-\min\{ \frac{\alpha}{2} - \frac{n}{4}, \frac12, \frac{n}{2}(p-1)-1 \} %
		+ \varsigma}
\end{align}
for $t \ge 1$.
\end{proposition}

To prove this proposition, we first prepare the following lemma.
\begin{lemma}\label{lem_he}
Let
$s_1 \ge s_2 \ge 0$, $\beta \ge 0$ and $\gamma, \nu \in [1,2]$.
Then, we have
\begin{align*}
	\| |\nabla|^{s_1} \mathcal{G}(t) \psi \|_{L^2}
	&\lesssim \langle t \rangle^{%
		-\frac{n}{2}\left( \frac{1}{\gamma}-\frac{1}{2}\right) - \frac{s_1-s_2}{2}}%
		\| |\nabla|^{s_2} \psi \|_{L^{\gamma}},\\
	\| |\cdot|^{\beta} \mathcal{G}(t)\psi \|_{L^2}
	&\lesssim \langle t \rangle^{-\frac{n}{2}\left(\frac{1}{\gamma}-\frac{1}{2}\right)+\frac{\beta}{2}}
		\| \psi \|_{L^{\gamma}}
		+ \langle t \rangle^{-\frac{n}{2}\left(\frac{1}{\nu}-\frac{1}{2}\right)}
		\| |\cdot|^{\nu} \psi \|_{L^{\nu}}
\end{align*}
for $t\ge 1$.
\end{lemma}
\begin{proof}
The first assertion is trivial.
Noting
$\| |\cdot|^{\beta} G(t) \|_{L^{k}} \lesssim t^{-\frac{n}{2}\left(1-\frac{1}{k} \right) + \frac{\beta}{2}}$,
we can easily prove the second one in the same way as Lemma \ref{lem_lin}.
\end{proof}

\begin{proof}[Proof of Proposition \ref{prop_asy2}]
For simplicity, we treat only higher dimensional cases $n\ge 2$.
The one-dimensional case can be proved in a similar way.
At first we shall consider the case $r>1$.
In this case, it suffices to estimate
\[
	v(t) - \varepsilon \mathcal{G} (u_0 + u_1) =
	\int_0^t \mathcal{G}(t-\tau) \mathcal{N}(u(\tau)) \,d\tau.
\]
Let us start with
\begin{align*}
	\left\| |\nabla|^s \int_0^t \mathcal{G}(t-\tau) \mathcal{N}(u(\tau)) \, d\tau \right\|_{L^2}
	&\lesssim
	\int_0^{\frac{t}{2}} \| |\nabla|^s \mathcal{G}(t-\tau) \mathcal{N}(u(\tau)) \|_{L^2}\,d\tau \\
	&\quad + \int_{\frac{t}{2}}^t
	\| |\nabla|^s \mathcal{G}(t-\tau) \mathcal{N}(u(\tau)) \|_{L^2}\,d\tau \\
	&=: I + I\!I.
\end{align*}
Applying Lemma \ref{lem_he} with $s_1=s$, $s_2=0$ and $\gamma=\sigma_1$,
we have
\begin{align*}
	I &\lesssim
		\int_0^{\frac{t}{2}} \langle t-\tau \rangle^{%
			-\frac{n}{2}\left( \frac{1}{\sigma_1} - \frac{1}{2} \right) - \frac{s}{2}}%
			\| \mathcal{N}(u(\tau)) \|_{L^{\sigma_1}}\,d\tau \\
	&\lesssim
		\| \mathcal{N}(u) \|_{Y(\infty)}
		\langle t \rangle^{-\frac{n}{2}\left( \frac{1}{\sigma_1} - \frac{1}{2} \right) - \frac{s}{2}}
		\int_0^{\frac{t}{2}}
			\langle \tau \rangle^{-\frac{n}{2}\left( \frac{p}{r}-\frac{1}{\sigma_1}\right) }\,d\tau\\
	&\lesssim
		\langle t \rangle^{-\frac{n}{2}\left(\frac{1}{\sigma_1}-\frac{1}{2}\right)-\frac{s}{2}%
			-\min\{ 0, \frac{n}{2}\left(\frac{p}{r}-\frac{1}{\sigma_1}\right)-1\}
			+\varsigma} \\
	&\lesssim
		\langle t \rangle^{-\frac{n}{2}\left( \frac{1}{r}-\frac{1}{2} \right) - \frac{s}{2}
	-\min\{ \frac{n}{2}\left(1-\frac{1}{r}\right), \frac12, \frac{n}{2r}(p-1)-1 \} %
		+ \varsigma}.
\end{align*}
For $I\!I$, we apply Lemma \ref{lem_he} with
$s_1 = s-[s]$, $s_2 = [s]$ and $\gamma = \rho$ to obtain
\begin{align*}
	I\!I &\lesssim
		\int_{\frac{t}{2}}^t
		\langle t-\tau \rangle^{%
			-\frac{n}{2}\left( \frac{1}{\rho}-\frac{1}{2}\right)-\frac{s-[s]}{2}}%
			\| |\nabla|^{[s]} \mathcal{N}(u(\tau)) \|_{L^{\rho}}\,d\tau \\
		&\lesssim \| \mathcal{N}(u) \|_{Y(\infty)}
			\langle t \rangle^{-\eta}
			\int_{\frac{t}{2}}^t
		\langle t-\tau \rangle^{%
			-\frac{n}{2}\left( \frac{1}{\rho}-\frac{1}{2}\right)-\frac{s-[s]}{2}}\,d\tau \\
		&\lesssim
		\langle t \rangle^{%
			-\frac{n}{2}\left(\frac{1}{r}-\frac{1}{2}\right) - \frac{s}{2}
			-\frac{n}{2r}(p-1) + 1}.%
\end{align*}
Summing up the above estimates, we have \eqref{esdifhes}.

Next, we prove \eqref{esdifhe}.
First we have
\begin{align*}
	\left\| \int_0^t \mathcal{G}(t-\tau) \mathcal{N}(u(\tau)) \, d\tau \right\|_{L^2}
	&\lesssim \int_0^{\frac{t}{2}} \| \mathcal{G}(t-\tau)\mathcal{N}(u(\tau))\|_{L^2}\,d\tau \\
	&\quad + \int_{\frac{t}{2}}^t \| \mathcal{G}(t-\tau)\mathcal{N}(u(\tau))\|_{L^2}\,d\tau \\
	&=: I\!I\!I + I\!V.
\end{align*}
Lemma \ref{lem_he} with $s_1=s_2=0$ and $\gamma = \sigma_1$ leads to
\begin{align*}
	I\!I\!I & \lesssim
		\int_0^{\frac{t}{2}}
		\langle t -\tau \rangle^{-\frac{n}{2}\left( \frac{1}{\sigma_1}-\frac{1}{2} \right)}
		\| \mathcal{N}(u(\tau)) \|_{L^{\sigma_1}}\, d\tau \\
		&\lesssim
		\| \mathcal{N}(u) \|_{Y(\infty)}
		\langle t \rangle^{-\frac{n}{2}\left( \frac{1}{\sigma_1}-\frac{1}{2} \right)}
		\int_0^{\frac{t}{2}}
		\langle \tau \rangle^{-\frac{n}{2}\left( \frac{p}{r}-\frac{1}{\sigma_1} \right)}\,d\tau\\
		&\lesssim
		\langle t \rangle^{-\frac{n}{2}\left( \frac{1}{r}-\frac{1}{2} \right)
	-\min\{ \frac{n}{2}\left(1-\frac{1}{r}\right), \frac12, \frac{n}{2r}(p-1)-1 \} %
		+ \varsigma}.
\end{align*}
Also, taking $s_1=s_2=0$ and $\gamma = \sigma_2$ in Lemma \ref{lem_he},
we see that
\begin{align*}
	I\!V &\lesssim
		\int_{\frac{t}{2}}^t
		\langle t -\tau \rangle^{-\frac{n}{2}\left( \frac{1}{\sigma_2} - \frac{1}{2} \right)}
		\| \mathcal{N}(u(\tau)) \|_{L^{\sigma_2}} \,d\tau \\
	&\lesssim
		\| \mathcal{N}(u) \|_{Y(\infty)}
		\langle t \rangle^{-\frac{n}{2}\left( \frac{p}{r}-\frac{1}{\sigma_2}\right) }
		\int_{\frac{t}{2}}^t
		\langle t -\tau \rangle^{-\frac{n}{2}\left( \frac{1}{\sigma_2} - \frac{1}{2} \right)}\,d\tau\\
	&\lesssim
		\langle t \rangle^{%
			-\frac{n}{2}\left( \frac{1}{r}-\frac{1}{2}\right)
			-\frac{n}{2r}(p-1) + 1},%
\end{align*}
since $-\frac{n}{2}( \frac{1}{\sigma_2} - \frac{1}{2} )>-1$.
Combining the estimates of $I\!I\!I$ and $I\!V$, we reach the estimate \eqref{esdifhe}.

Finally, we give a proof of \eqref{esdifhea}.
\begin{align*}
	\int_0^t \| |\cdot|^{\beta} \mathcal{G}(t-\tau) \mathcal{N}(u(\tau)) \|_{L^2}\,d\tau
	&\lesssim
		\int_0^{\frac{t}{2}}
		\| |\cdot|^{\beta} \mathcal{G}(t-\tau) \mathcal{N}(u(\tau)) \|_{L^2}\,d\tau\\
	&\quad
		+\int_{\frac{t}{2}}^t
		\| |\cdot|^{\beta} \mathcal{G}(t-\tau) \mathcal{N}(u(\tau)) \|_{L^2}\,d\tau \\
	&=: V + V\!I.
\end{align*}
For $V$, applying Lemma \ref{lem_he} with
$\gamma=\sigma_1$ and $\nu=q$,
we have
\begin{align*}
	V&\lesssim
	\int_0^{\frac{t}{2}} \langle t-\tau \rangle^{%
			-\frac{n}{2}\left(\frac{1}{\sigma_1}-\frac{1}{2}\right)+\frac{\beta}{2}}%
			\| \mathcal{N}(u) \|_{L^{\sigma_1}}\,d\tau
		+ \int_0^{\frac{t}{2}}
			\langle t - \tau \rangle^{-\frac{n}{2}\left( \frac{1}{q}-\frac{1}{2} \right)}
			\| |\cdot|^{\beta} \mathcal{N}(u) \|_{L^q} \,d\tau \\
	&\lesssim
		\| \mathcal{N}(u) \|_{Y(\infty)}
		\langle t \rangle^{-\frac{n}{2}\left(\frac{1}{\sigma_1}-\frac{1}{2}\right)+\frac{\beta}{2}}%
		\int_0^{\frac{t}{2}}
			\langle \tau \rangle^{-\frac{n}{2}\left( \frac{p}{r}-\frac{1}{\sigma_1}\right)}\,d\tau \\
	&\quad + \| \mathcal{N}(u) \|_{Y(\infty)}
		\langle t \rangle^{-\frac{n}{2}\left( \frac{1}{q}-\frac{1}{2} \right)}
		\int_0^{\frac{t}{2}}
		\langle \tau \rangle^{-\zeta_{\beta}} \,d\tau \\
	&\lesssim
		\langle t \rangle^{-\frac{n}{2}\left( \frac{1}{r}-\frac{1}{2} \right) + \frac{\beta}{2}
	-\min\{ \frac{n}{2}\left(1-\frac{1}{r}\right), \frac12, \frac{n}{2r}(p-1)-1 \} %
		+ \varsigma}.
\end{align*}
For $V\!I$, letting
$\gamma=\sigma_2$ and $\nu = q$
in Lemma \ref{lem_he}, we see that
\begin{align*}
	V\!I&\lesssim
	\int_{\frac{t}{2}}^t
		\langle t - \tau \rangle^{%
			-\frac{n}{2}\left(\frac{1}{\sigma_2}-\frac{1}{2}\right)+\frac{\beta}{2}}%
		\| \mathcal{N}(u(\tau)) \|_{L^{\sigma_2}}\,d\tau \\
	&\quad
		+ \int_{\frac{t}{2}}^t
		\langle t- \tau \rangle^{%
			-\frac{n}{2}\left( \frac{1}{q}-\frac{1}{2}\right)}%
		\| |\cdot|^{\beta} \mathcal{N}(u(\tau)) \|_{L^{q}}\,d\tau \\
	&\lesssim
		\| \mathcal{N}(u) \|_{Y(\infty)}
			\langle t \rangle^{-\frac{n}{2}\left(\frac{p}{r}-\frac{1}{\sigma_2}\right)}
		\int_{\frac{t}{2}}^t \langle t-\tau \rangle^{%
			-\frac{n}{2}\left(\frac{1}{\sigma_2}-\frac{1}{2}\right)+\frac{\beta}{2}}\,d\tau \\
	&\quad +
		\| \mathcal{N}(u) \|_{Y(\infty)}
			\langle t \rangle^{-\zeta_{\beta}}
			\int_{\frac{t}{2}}^t
			\langle t-\tau \rangle^{%
				-\frac{n}{2}\left(\frac{1}{q}-\frac{1}{2}\right)}\,d\tau \\
	&\lesssim
		\langle t \rangle^{%
			-\frac{n}{2}\left(\frac{1}{r}-\frac{1}{2}\right)+\frac{\beta}{2}%
			-\frac{n}{2r}(p-1) + 1}.%
\end{align*}
Summing up the above estimates, we have \eqref{esdifhea} and complete the proof
when $r>1$.

Next, we give a proof of the case $r=1$.
We put
\[
	\theta = \theta_1 + \theta_2,\quad
	\theta_1 = \varepsilon \int_{\mathbb{R}^n} (u_0+u_1)\,dx,\quad
	\theta_2 = \int_0^{\infty}\int_{\mathbb{R}^n} \mathcal{N}(u(\tau))\,dxd\tau.
\]
We claim that for $m \in [1,\infty]$, it follows that
\begin{align}
\label{r1asy0}
	\| \varepsilon \mathcal{G}(t) (u_0+u_1) - \theta_1 G(t) \|_{L^m}
	= o (t^{-\frac{n}{2}\left(1-\frac{1}{m}\right)})
\end{align}
as $t\to \infty$ for $u_0+u_1\in L^1$ and
\begin{align}
\label{r1asy1}
	\| \varepsilon \mathcal{G}(t) (u_0+u_1) - \theta_1 G(t) \|_{L^m}
	\lesssim \langle t \rangle^{-\frac{n}{2}\left(1-\frac{1}{m}\right)%
		-\min\{ \frac{1}{2}, \frac{\delta}{2} \}}
	\| \langle \cdot \rangle^{\delta} (u_0+u_1) \|_{L^1}
\end{align}
for $t\ge 1$ and $u_0+u_1 \in H^{0,\alpha}$,
where
$\delta$ is an arbitrary number satisfying
$0 < \delta < \alpha - \frac{n}{2}$.
Here we also note that
$\| \langle \cdot \rangle^{\delta} \phi \|_{L^1}\lesssim \| \phi \|_{H^{0,\alpha}}$.
To prove \eqref{r1asy0}, we first consider the case $m=1$.
We write $\phi = \varepsilon (u_0+u_1)$ and have
\begin{align*}
	&\int_{\mathbb{R}^n} \left|
		\int_{\mathbb{R}^n} ( G(t, x-y)-G(t,x) ) \phi (y)\,dy \right| \,dx\\
	&\quad \lesssim
	\int_{\mathbb{R}^n} \left|
		\int_{|y|< \varrho t^{1/2}} ( G(t, x-y)-G(t,x) ) \phi (y)\,dy \right| \,dx\\
	&\qquad +\int_{\mathbb{R}^n} \left|
		\int_{|y|> \varrho t^{1/2}} ( G(t, x-y)-G(t,x) ) \phi (y)\,dy \right| \,dx
	=: I+ I\!I,
\end{align*}
where
$\varrho > 0$ is an arbitrary small number.
For $I$,
we use the mean value theorem
\[
	|G(t,x-y)-G(t,x)| \le 2(4\pi t)^{-\frac{n}{2}} t^{-\frac{1}{2}} |y|
		\int_0^1 \frac{|x-ay|}{\sqrt{t}} e^{-\frac{|x-ay|^2}{4t}}\,da.
\]
To prove \eqref{r1asy0}, noting that
$t^{-\frac{1}{2}} |y| \le \varrho$.
Thus, we have
$I \le C\varrho$.
For $I\!I$, we easily estimate
\begin{align*}
	I\!I &\lesssim
		\int_{|y|> \varrho t^{1/2}} \left( 
			\int_{\mathbb{R}^n} | G(t,x-y) - G(t,x) |\,dx \right)
			|\phi (y) |\, dy
	\lesssim \int_{|y| > \varrho t^{\frac{1}{2}}} |\phi(y)|\,dy.
\end{align*}
Taking $t$ sufficiently large, we have $I\!I \le \varrho$.
Consequently, we obtain
$I + I\!I \le C\varsigma$
for sufficiently large $t$.
Since $\varsigma$ is arbitrary, this proves \eqref{r1asy0}.
Similarly, we can easily prove \eqref{r1asy0} in the case
$m=\infty$, and then, by the interpolation,
we can obtain \eqref{r1asy0} for all $m\in [1,\infty]$.

To prove \eqref{r1asy1},
noting
$|y| \le |y|^{\min\{ \delta, 1 \}} t^{\frac{1}{2}(1-\min\{ \delta, 1\})}$
on the integral region of $I$ with $\varrho=1$, and then, by the Fubini theorem,
we see that
\[
	I \lesssim t^{-\frac{1}{2} + \frac{1}{2}(1-\min\{\delta,1\})}
		\int_{\mathbb{R}^n} |y|^{\min\{\delta, 1\}} |\phi (y)|\,dy
	\lesssim t^{-\min\{\frac{1}{2}, \frac{\delta}{2}\}}
		\| \langle \cdot \rangle^{\delta} \phi \|_{L^1}.
\]
On the other hand, the term $I\!I$ is easily estimated as
\begin{align*}
	I\!I &\lesssim
		\int_{|y|>t^{1/2}} \left( 
			\int_{\mathbb{R}^n} | G(t,x-y) - G(t,x) |\,dx \right)
			|\phi (y) |\, dy\\
	&\lesssim t^{-\frac{\delta}{2}} \int_{\mathbb{R}^n} |y|^{\delta} |\phi(y)|\,dy.
\end{align*}
Thus, we obtain \eqref{r1asy1} when $m=1$.
In a similar way, we can prove \eqref{r1asy1} when $m=\infty$
and hence, the interpolation gives \eqref{r1asy1} for all $m\in [1,\infty]$.

Next, we claim that
\begin{align}
\label{r1asy3}
	\left\| \int_0^t \mathcal{G}(t-\tau)\mathcal{N}(u(\tau))\,d\tau
		- \theta_2 G(t) \right\|_{L^m}
	= o(t^{-\frac{n}{2}\left( 1 - \frac{1}{m} \right)})
\end{align}
as $t\to \infty$ and
\begin{align}
\label{r1asy2}
	\left\| \int_0^t \mathcal{G}(t-\tau)\mathcal{N}(u(\tau))\,d\tau
		- \theta_2 G(t)  \right\|_{L^m}
	\lesssim \langle t \rangle^{-\frac{n}{2}\left( 1 - \frac{1}{m} \right)%
	- \min\{ \frac12, \frac{n}{2}(p-1)-1 \}%
	+\varsigma}%
\end{align}
for all $t\ge 1$ and $m$ satisfying \eqref{m}.
Indeed, we first divide the left-hand side as
\begin{align*}
	&\left\| \int_0^t \mathcal{G}(t-\tau)\mathcal{N}(u(\tau))\,d\tau
		- \theta_2 G(t) \right\|_{L^m} \\
	&\quad \lesssim
		\int_0^{\frac{t}{2}} \left\|  \mathcal{G}(t-\tau)\mathcal{N}(u(\tau))
			- G(t-\tau) \left( \int_{\mathbb{R}^n} \mathcal{N}(u(\tau))\,dx\right) \right\|_{L^m}
			\,d\tau\\
	&\qquad +
		\int_0^{\frac{t}{2}} \left\| \left(
			G(t-\tau) - G(t) \right) \left( \int_{\mathbb{R}^n} \mathcal{N}(u(\tau))\,dx\right)
			\right\|_{L^m}\,d\tau \\
	&\qquad +
		\int_{\frac{t}{2}}^t
			\left\| \mathcal{G}(t-\tau)\mathcal{N}(u(\tau)) \right\|_{L^m}\,d\tau +
		\int_{\frac{t}{2}}^{\infty}
			\left\| G(t) \left( \int_{\mathbb{R}^n} \mathcal{N}(u(\tau))\,dx\right)
					\right\|_{L^m}\,d\tau \\
	&\quad =: I\!I\!I_m + I\!V_m + V_m + V\!I_m.
\end{align*}
Let us start with $V\!I_m$.
Since
$\| G(t) \|_{L^m} = C t^{-\frac{n}{2}(1-\frac{1}{m})}$
and
$\| \mathcal{N}(u(\tau)) \|_{L^1}\lesssim
\langle \tau \rangle^{-\frac{n}{2}(p-1)}$,
it is easy to see that
\[
	V\!I_m \lesssim
	t^{-\frac{n}{2}(1-\frac{1}{m}) -\frac{n}{2}(p-1) +1}.
\]
Secondly, we estimate $V_m$.
For simplicity, we only consider the case $n>2s$.
When $m=\frac{2n}{n-2s}$, the Sobolev inequality gives
\begin{align*}
	V_{\frac{2n}{n-2s}} &\lesssim
	 \int_{\frac{t}{2}}^t
			\left\| |\nabla|^s\mathcal{G}(t-\tau)\mathcal{N}(u(\tau)) \right\|_{L^2}\,d\tau \\
	&\lesssim \int_{\frac{t}{2}}^t
		\langle t-\tau \rangle^{-\frac{n}{2}\left(\frac{1}{\rho}-\frac{1}{2}\right)-\frac{s-[s]}{2} }
			\| |\nabla|^{[s]} \mathcal{N}(u(\tau)) \|_{L^{\rho}}\,d\tau \\
	&\lesssim \int_{\frac{t}{2}}^t
		\langle t-\tau \rangle^{-\frac{n}{2}\left(\frac{1}{\rho}-\frac{1}{2}\right)-\frac{s-[s]}{2} }
		\langle \tau \rangle^{-\eta}\, d\tau \\
	&\lesssim \langle t \rangle^{-\frac{n}{2}(1-\frac{1}{m}) -\frac{n}{2}(p-1) +1}.
\end{align*}
Also, by
$\| \mathcal{N}(u(\tau)) \|_{L^1} \lesssim 
\langle \tau \rangle^{-\frac{n}{2}(p-1)}$,
we immediately obtain
$V_1 \lesssim \langle t \rangle^{-\frac{n}{2}(p-1) +1}$.
Interpolating these two cases, we conclude that
\[
	V_m \lesssim
	\langle t \rangle^{-\frac{n}{2}(1-\frac{1}{m}) -\frac{n}{2}(p-1) +1}
\]
for all $m\in [1, \frac{2n}{n-2s}]$.
The case $n \le 2s$ can be also proved by using a similar
interpolation as in the proof of Lemma \ref{lem_ip} (ii).

Next, we consider $I\!V_m$.
It follows from
$\| \mathcal{N}(u(\tau)) \|_{L^1}
\lesssim \langle \tau \rangle^{-\frac{n}{2}(p-1)}$
and the fundamental theorem of calculus
\[
	G(t-\tau,x) - G(t,x) = - \tau \int_0^1 (\partial_t G)(t-a \tau,x)\,da
\]
we deduce that
\begin{align*}
	I\!V_m
	&\lesssim
		\int_0^{\frac{t}{2}} \langle t - \tau \rangle^{%
			-\frac{n}{2}\left( 1-\frac{1}{m} \right)-1}%
			\langle \tau \rangle^{1-\frac{n}{2}(p-1)}\,d\tau 
	\lesssim \langle t \rangle^{-\frac{n}{2}\left( 1-\frac{1}{m} \right)}
	\begin{cases}
		\langle t \rangle^{-1}&(p>1+\frac4n),\\
		\langle t \rangle^{-1}\log \langle t \rangle&(p=1+\frac4n),\\
		\langle t \rangle^{-\frac{n}{2}(p-1)+1}&(p<1+\frac4n).
	\end{cases}
\end{align*}
Finally, we estimate $I\!I\!I_m$.
To prove \eqref{r1asy3}, noting
$\mathcal{N}(u(\tau)) \in L^{\infty}(0,\infty; L^1(\mathbb{R}^n))$
we can prove
$I\!I\!I_m = o( t^{-\frac{n}{2}\left(1-\frac{1}{m}\right)})$
as $t\to \infty$.

To show \eqref{r1asy2}, applying \eqref{r1asy1}, we have
\begin{align*}
	I\!I\!I_m &\lesssim
	\int_0^{\frac{t}{2}} \langle t -\tau \rangle^{%
		-\frac{n}{2}\left( 1-\frac{1}{m} \right) - \min\{\frac{1}{2}, \frac{\delta}{2}\} }%
		\| \langle \cdot \rangle^{\delta} \mathcal{N}(u) \|_{L^1} \,d\tau.
\end{align*}
Taking $\delta = \beta\ (n=2), \delta = 1\ (n\ge 3)$,
we conclude
\begin{align*}
	I\!I\!I_m &\lesssim
	 \langle t \rangle^{-\frac{n}{2}\left( 1 - \frac{1}{m} \right)%
	- \min\{ \frac12, \frac{n}{2}(p-1)-1 \}%
	+\varsigma}.%
\end{align*}
Putting together the estimates above, we reach \eqref{r1asy2}.
\end{proof}

From Propositions \ref{prop_asy} and \ref{prop_asy2},
we finish up the proof of Theorem \ref{thm_asy}.
\begin{proof}[Proof of Theorem \ref{thm_asy}]
We start with the case $r>1$.
Let $m$ be a real number satisfying \eqref{m}.
First, we assume that $m\ge 2$.
In this case, the Sobolev embedding theorem implies
$\| \psi \|_{L^m} \le \| |\nabla|^{s_m} \psi \|_{L^2}$
with
$s_m = n ( \frac{1}{2} - \frac{1}{m} ) \in [0,s]$.
This and the interpolation inequality
\[
	\| |\nabla|^{s_m} \psi \|_{L^2}
		\le \| |\nabla|^s \psi \|_{L^2}^{\frac{s_m}{s}}
			\| \psi \|_{L^2}^{1-\frac{s_m}{s}}
\]
lead to
\begin{align*}
	& \| u(t) - \varepsilon \mathcal{G}(t)(u_0+u_1) \|_{L^m} \\
	&\quad \le \| |\nabla|^s (u(t) - \varepsilon \mathcal{G}(t)(u_0+u_1)) \|_{L^2}^{\frac{s_m}{s}}
		\| u(t) - \varepsilon \mathcal{G}(t)(u_0+u_1) \|_{L^2}^{1-\frac{s_m}{s}}.
\end{align*}
Using Propositions \ref{prop_asy} and \ref{prop_asy2},
we calculate
\begin{align*}
	&\| |\nabla|^s (u(t) - \varepsilon \mathcal{G}(t)(u_0+u_1)) \|_{L^2} \\
	&\lesssim
		\| |\nabla|^s (u(t) -v(t)) \|_{L^2}
		+ \| |\nabla|^s (v(t) - \varepsilon \mathcal{G}(t)(u_0+u_1)) \|_{L^2} \\
	&\lesssim
		\langle t \rangle^{%
		-\frac{n}{2} \left( \frac{1}{r}-\frac{1}{2}\right)-\frac{s}{2}%
		- \min\{ 1 , \frac{n}{2r}(p-1) - \frac{1}{2} \}}
		+
		\langle t \rangle^{-\frac{n}{2}\left( \frac{1}{r}-\frac{1}{2} \right) - \frac{s}{2}
		-\min\{ \frac{n}{2}\left(1-\frac{1}{r}\right), \frac12, \frac{n}{2r}(p-1)-1 \} %
		+ \varsigma}\\
	&\lesssim
		\langle t \rangle^{-\frac{n}{2}\left( \frac{1}{r}-\frac{1}{2} \right) - \frac{s}{2}
		-\min\{ \frac{n}{2}\left(1-\frac{1}{r}\right), \frac12, \frac{n}{2r}(p-1)-1 \} %
		+ \varsigma}.
\end{align*}
Similarly, we can see that
\[
	\| u(t) - \varepsilon \mathcal{G}(t)(u_0+u_1) \|_{L^2} \\
	\lesssim \langle t \rangle^{-\frac{n}{2}\left( \frac{1}{r}-\frac{1}{2} \right)
	-\min\{ \frac{n}{2}\left(1-\frac{1}{r}\right), \frac12, \frac{n}{2r}(p-1)-1 \} %
		+ \varsigma}.
\]
These two estimates yield
\begin{align*}
	 \| u(t) - \varepsilon \mathcal{G}(t)(u_0+u_1) \|_{L^m}
	 &\lesssim
	 \langle t \rangle^{-\frac{n}{2}\left( \frac{1}{r}-\frac{1}{2} \right) -\frac{s_m}{2}
	 -\min\{ \frac{n}{2}\left(1-\frac{1}{r}\right), \frac12, \frac{n}{2r}(p-1)-1 \} %
		+ \varsigma}\\
	 &=
	  \langle t \rangle^{-\frac{n}{2}\left( \frac{1}{r}-\frac{1}{m} \right)
	  -\min\{ \frac{n}{2}\left(1-\frac{1}{r}\right), \frac12, \frac{n}{2r}(p-1)-1 \} %
		+ \varsigma},
\end{align*}
which shows the desired estimate.

When $m\in [r,2]$, by the interpolation
\begin{align*}
	\| \psi \|_{L^m} \le
	\| \psi \|_{L^r}^{\frac{r(2-m)}{m(2-r)}} \| \psi \|_{L^2}^{\frac{2(m-r)}{m(2-r)}} 
	\le \| |\cdot|^{\beta} \psi \|_{L^2}^{\frac{r(2-m)}{m(2-r)}}
		\| \psi \|_{L^2}^{\frac{2(m-r)}{m(2-r)}}
\end{align*}
with $\beta$ satisfying
$n(\frac{1}{r}-\frac{1}{2}) < \beta \le \alpha$
and
$(\beta - n(\frac{1}{r}-\frac{1}{2}))\frac{r(2-m)}{m(2-r)} < \varsigma$.
Thus, it suffices to estimate the right-hand side with
$\psi = u(t) - \varepsilon \mathcal{G}(t)(u_0+u_1)$.
In the same way as before, we can deduce that
\[
	\| |\cdot|^{\beta} (u(t) - \varepsilon \mathcal{G}(t)(u_0+u_1)) \|_{L^2}
	\lesssim
	\langle t \rangle^{-\frac{n}{2}\left( \frac{1}{r}-\frac{1}{2} \right) +\frac{\beta}{2}
	 -\min\{ \frac{n}{2}\left(1-\frac{1}{r}\right), \frac12, \frac{n}{2r}(p-1)-1 \} %
		+ \varsigma}
\]
and hence, we have
\begin{align*}
	\| u(t) - \varepsilon \mathcal{G}(t)(u_0+u_1) \|_{L^m}
	 &\lesssim
	\langle t \rangle^{%
		-\frac{n}{2}\left( \frac{1}{r}-\frac{1}{2} \right)%
		+\frac{\beta}{2}\frac{r(2-m)}{m(2-r)}%
		 -\min\{ \frac{n}{2}\left(1-\frac{1}{r}\right), \frac12, \frac{n}{2r}(p-1)-1 \} %
		+ \varsigma}\\
	&=\langle t \rangle^{%
		-\frac{n}{2}\left( \frac{1}{r}-\frac{1}{m} \right)%
		+\left( \frac{\beta}{2} - \frac{n}{2}\left(\frac{1}{r}-\frac{1}{2} \right) \right)%
			 \frac{r(2-m)}{m(2-r)}
		 -\min\{ \frac{n}{2}\left(1-\frac{1}{r}\right), \frac12, \frac{n}{2r}(p-1)-1 \} %
		+ \varsigma}\\
	&= \langle t \rangle^{%
		-\frac{n}{2}\left( \frac{1}{r}-\frac{1}{m} \right)%
		 -\min\{ \frac{n}{2}\left(1-\frac{1}{r}\right), \frac12, \frac{n}{2r}(p-1)-1 \} %
		+ \frac{3\varsigma}{2}}.
\end{align*}
Rewriting $\frac{3\varsigma}{2}$ as $\varsigma$,
we complete the proof.

When $r=1$, as before, we first divide
\begin{align*}
	\| u(t) - \theta G(t)\|_{L^m}
	&\lesssim
	\| u(t) - v(t) \|_{L^m} + \| v(t) - \theta G(t) \|_{L^m}.
\end{align*}
Let
$\beta$ satisfy $\frac{n}{2}<\beta \le \alpha$
and $(\beta - \frac{n}{2})(\frac1m-\frac12) < \varsigma$.
Then, the first term is estimated as
\begin{align*}
	\| u(t) - v(t) \|_{L^m}
	&\lesssim
	\langle t \rangle^{-\frac{n}{2}\left(1-\frac{1}{m}\right)}
	\begin{cases}
		\langle t \rangle^{-\min\{ 1, \frac{n}{2}(p-1)-\frac{1}{2}\} }
		&(m\ge 2),\\
		\langle t \rangle^{\left(\beta-\frac{n}{2}\right)\left(\frac{1}{m}-\frac{1}{2}\right)%
			-\min\{ 1, \frac{n}{2}(p-1)-\frac{1}{2}\}}%
		&(1\le m <2)
	\end{cases} \\
	&\lesssim
	\langle t \rangle^{-\frac{n}{2}\left(1-\frac{1}{m}\right)%
	-\min\{ 1, \frac{n}{2}(p-1)-\frac{1}{2}\} + \varsigma}.
\end{align*}
For the second term, we apply Proposition \ref{prop_asy2}
and comparing the decay rate leads to the conclusion.
\end{proof}

\section{Blow-up and estimates of the lifespan}
In this section, we give a proof of Theorems \ref{thm_lflow} and \ref{thm_lfupp}.

\subsection{Estimates of the lifespan from below}
At first, we prove Theorem \ref{thm_lflow}.

\begin{proof}[Proof of Theorem \ref{thm_lflow}]
We first consider Case 1, namely, we assume
$r\in [1,2]$ and
$\min\{ 1+\frac{r}{2}, 1+\frac{r}{n} \} \le p < 1+\frac{2r}{n}$.

If $T(\varepsilon) = \infty$, then the assertion of the theorem is obviously true.
Hence, in what follows we assume $T(\varepsilon) < \infty$.
Then, Theorem \ref{thm_lwp} shows that
\begin{align}
\label{lfbu}
	\lim_{t \to T(\varepsilon)} \| u(t) \|_{H^{s,0}\cap H^{0,\alpha}} = \infty.
\end{align}
On the other hand, by Lemmas \ref{lem_duh} and \ref{lem_nl}, we have
\begin{align*}
	\left\| \int_0^t \mathcal{D}(t-\tau) \mathcal{N}(u(\tau))\,d\tau \right\|_{X(T)}
	&\lesssim
	\| \mathcal{N}(u) \|_{Y(T)}
	\langle T \rangle^{1-\frac{n}{2r}(p-1)} \\
	&\lesssim
	\| u \|_{X(T)}^{p}
	\langle T \rangle^{1-\frac{n}{2r}(p-1)}.
\end{align*}
This implies
\begin{align}
\label{ap1}
	\| u \|_{X(T)}
	\le \varepsilon C_0I_0 + C_1\| u \|_{X(T)}^{p}
	\langle T \rangle^{1-\frac{n}{2r}(p-1)}
\end{align}
with some constant $C_0, C_1 > 0$,
where
$I_0 = \| u_0 \|_{H^{s,0}\cap H^{0,\alpha}} + \| u_1 \|_{H^{s-1,0}\cap H^{0,\alpha}}$.
Now, by \eqref{lfbu}
and the continuity of $\| u \|_{X(T)}$ with respect to $T$,
there exists the smallest time
$\widetilde{T}(\varepsilon)$ such that
$\| u \|_{X(\widetilde{T}(\varepsilon))} = 2C_0I_0 \varepsilon$ holds.
Then, letting $T=\widetilde{T}(\varepsilon)$ in the above inequality gives
\begin{align*}
	2C_0I_0 \varepsilon
	\le C_0I_0 \varepsilon 
	+ C_1 (2C_0I_0 \varepsilon)^p
	\langle \widetilde{T}(\varepsilon) \rangle^{1-\frac{n}{2r}(p-1)}.
\end{align*}
We rewrite it as
\[
	C_0I_0 \varepsilon^{-(p-1)}
	\le C_1 (2C_0I_0)^p
	\langle \widetilde{T}(\varepsilon) \rangle^{1-\frac{n}{2r}(p-1)},
\]
which implies
\[
	\langle \widetilde{T}(\varepsilon) \rangle
	\gtrsim
	\varepsilon^{-1/\omega}
\]
with $\omega = \frac{1}{p-1} - \frac{n}{2r}$.
Thus, taking $\varepsilon$ sufficiently small,
we easily see that $\langle \widetilde{T}(\varepsilon) \rangle \ge 1$,
which enables us to replace
$\langle \widetilde{T}(\varepsilon) \rangle$
by $\widetilde{T}(\varepsilon)$ in the above estimate.
This and $\widetilde{T}(\varepsilon) < T(\varepsilon)$ imply the desired estimate.

In Case 2, that is, when $r=1$ and $p=1+\frac{2}{n}$,
the first author and Ogawa \cite{IkeOg} have already
proved the conclusion of Theorem \ref{thm_lflow}
by a slightly different argument.

Instead of \eqref{ap1}, we obtain
\[
	\| u \|_{X(T)}
	\le \varepsilon C_0I_0 + C_1\| u \|_{X(T)}^{p}
	\log \langle T \rangle.
\]
In the same manner as above, we can see that
\[
	\langle \widetilde{T}(\varepsilon) \rangle
	\ge
	\exp \left( C\varepsilon^{-(p-1)} \right),
\]
which gives the desired conclusion.
\end{proof}

\subsection{Estimates of the lifespan from above}
In this subsection, we give a proof of Theorem \ref{thm_lfupp}.
The proof is based on the test function method
introduced by Zhang \cite{Zh01} and
refined to estimate the lifespan by
Kuiper \cite{Ku03}, Sun \cite{Su10} and \cite{IkeWa15}
in which initial data belonging to $L^1$ are treated.
Here we further adapt their method to fit
initial data not belonging to $L^1$ by the argument in
\cite{IkeIn15}.


Before proving Theorem \ref{thm_lfupp},
we introduce the definition of weak solutions of \eqref{nldw}.
Let $T > 0$ and $u_0, u_1 \in L^1_{loc}(\mathbb{R}^n)$.
We say that a function $u \in L^p_{loc}([0,T)\times \mathbb{R}^n)$
is a weak solution
of \eqref{nldw} on the interval $[0,T)$ if the identity
\begin{align*}
	&\int_{[0,T)\times \mathbb{R}^n} u
	\left( \psi_{tt} - \Delta \psi - \psi_t \right) \,dxdt \\
	&= \varepsilon \int_{\mathbb{R}^n}
		\left( (u_0 + u_1)\psi(0,x) - u_0 \psi_t(0,x) \right)\,dx
		+ \int_{[0,T)\times \mathbb{R}^n} \mathcal{N}(u) \psi \,dxdt
\end{align*}
is valid for any
$\psi \in C_0^{\infty}([0,T)\times \mathbb{R}^n)$.
We also define the lifespan of the weak solution:
\begin{align*}
	T_w(\varepsilon):= \sup\{
		T \in (0,\infty) ; \mbox{there exists a weak solution}\ u\ 
						\mbox{on}\ [0,T) \}.
\end{align*}

\begin{proposition}\label{prop_bu}
Under the assumptions of Theorem \ref{thm_lfupp},
we have
\[
	T_w(\varepsilon)
	\lesssim \varepsilon^{-1/\kappa}
\]
for $\varepsilon \in (0,1]$,
where $\kappa = \frac{1}{p-1} - \frac{\lambda}{2}$
and
$\lambda$ satisfies $\frac{n}{2} + \alpha < \lambda < \frac{2}{p-1}$.
\end{proposition}

\begin{proof}[Proof of Proposition \ref{prop_bu}]
We assume that
$\mathcal{N}(u) = |u|^p$.
The case $\mathcal{N}(u) = -|u|^p$ is reduces to the case above by
considering $-u$.

First, we may assume that $T_w(\varepsilon) \ge 4$.
Because, if $T_w(\varepsilon) \le 4$,
then we immediately obtain
$T_w(\varepsilon) \le 4\varepsilon^{-1/\kappa}$
for any $\varepsilon \in (0,1]$.

Let
$\eta = \eta(t) \in C_0^{\infty}([0,\infty))$ be a test function satisfying
\[
	0\le \eta (t) \le 1,\quad
	\eta(t) = \begin{cases}
		1& (0\le t \le 1/2),\\
		0& (t \ge 1)
	\end{cases}
\]
and let $\phi(x) := \eta (|x|)$.
Then, we have (see for example, \cite{NiWa14})
\begin{align}
\label{t1}
	|\eta^{(j)}(t)| \lesssim \eta(t)^{1/p}\ (j=1,2),\quad
	|\Delta \phi (x) | \lesssim \phi(x)^{1/p}.
\end{align}
Let
$\tau \in [1, T_w(\varepsilon))$ and $R \in [2,\infty)$ be parameters.
We define
$\eta_{\tau}(t) = \eta (\frac{t}{\tau})$,
$\phi_R(x) = \phi(\frac{x}{R})$
and
\[
	\psi = \psi_{\tau,R}(t,x) = \eta_{\tau}(t) \phi_R(x).
\]
We also put
\begin{align*}
	I_{\tau,R} := \int_{[0,\tau)\times B_R} |u|^p \psi_{\tau,R}\,dxdt,\quad
	J_R := \varepsilon \int_{B_R} (u_0+u_1) \phi_R(x) \,dx,
\end{align*}
where $B_R = \{ x\in \mathbb{R}^n ;  |x| < R \}$.
Taking $\psi = \psi_{\tau,R}$ in the definition of the weak solution,
we have the identity
\begin{align*}
	I_{\tau,R} + J_R
	&= \int_{[0,\tau)\times B_R} u \partial_t^2 \psi_{\tau,R}\,dxdt
		- \int_{[0,\tau)\times B_R} u \Delta \psi_{\tau,R}\,dxdt
		- \int_{[0,\tau)\times B_R} u \partial_t \psi_{\tau,R}\,dxdt \\
	&=: K_1 + K_2 + K_3.
\end{align*}
By the H\"{o}lder inequality and \eqref{t1}, we have
\begin{align}
\label{k1}
	K_1&\lesssim
		\tau^{-2} \int_{[0,T)\times B_R} |u| |\eta^{\prime\prime}(t/\tau)| |\phi_R(x)|\,dxdt \\
\nonumber
		&\lesssim
		\tau^{-2}
		\int_{[0,T)\times B_R} |u| \eta_{\tau}(t)^{1/p} |\phi_R(x)|\,dxdt \\
\nonumber
		&\lesssim
		\tau^{-2} I_{\tau,R}^{1/p}
		\left( \int_{\frac{\tau}{2}}^{\tau} \int_{B_R} \phi_R \,dxdt \right)^{1/p^{\prime}} \\
\nonumber
		&\lesssim \tau^{-2+1/p^{\prime}} R^{n/p^{\prime}} I_{\tau,R}^{1/p},
\end{align}
where
$p^{\prime} = \frac{p}{p-1}$.
In a similar way, we see that
\begin{align}
\label{k2}
	K_2 \lesssim  \tau^{1/p^{\prime}} R^{-2+n/p^{\prime}} I_{\tau,R}^{1/p}
\end{align}
and
\begin{align}
\label{k3}
	K_3 \lesssim \tau^{-1+1/p^{\prime}} R^{n/p^{\prime}} I_{\tau,R}^{1/p}.
\end{align}
It follows from \eqref{k1}--\eqref{k3} that
\begin{align*}
	I_{\tau,R} + J_R
	\lesssim \left( \tau^{1/p^{\prime}} R^{-2+n/p^{\prime}}
				+ \tau^{-1+1/p^{\prime}} R^{n/p^{\prime}} \right) I_{\tau,R}^{1/p}.
\end{align*}
By the Young inequality, the right-hand side is bounded by
\[
	\frac{1}{2} I_{\tau,R}
	+ C ( \tau R^{-2 p^{\prime} +n}
		+ \tau^{-p^{\prime} +1} R^{n} ).
\]
Therefore, we have
\[
	J_R \lesssim \tau R^{-2 p^{\prime} +n} + \tau^{-p^{\prime} +1} R^{n}.
\]
From the assumption on the initial data,
we see that
\begin{align*}
	J_R &\ge \varepsilon \int_{|x|>1} (u_0(x)+u_1(x)) \phi \left( \frac{x}{R} \right)\,dx \\
		&\ge \varepsilon \int_{|x|>1} |x|^{-\lambda} \phi \left( \frac{x}{R} \right)\,dx \\
		&\ge \varepsilon R^{-\lambda+n} \int_{|y|>1/R} |y|^{-\lambda} \phi (y) \,dy\\
		&\ge \varepsilon R^{-\lambda+n} \int_{|y|>1/2} |y|^{-\lambda} \phi (y) \,dy,
\end{align*}
since $R \ge 2$.
Thus, we have
\begin{align*}
	\varepsilon \lesssim
	\tau R^{-2 p^{\prime} +\lambda } + \tau^{-p^{\prime} +1} R^{\lambda}.
\end{align*}
By taking $R = \tau^{1/2}$, we obtain
$\varepsilon \lesssim \tau^{-( \frac{1}{p-1} - \frac{\lambda}{2})}$,
which implies
\[
	\tau \lesssim \varepsilon^{-1/\kappa},
\]
with $\kappa = \frac{1}{p-1} - \frac{\lambda}{2}$.
Since $\tau$ is arbitrary in $[1,T_w(\varepsilon))$, we conclude
\[
	T_w(\varepsilon) \lesssim \varepsilon^{-1/\kappa}.
\]
This finishes the proof.
\end{proof}

\begin{proposition}\label{prop_sol}
Under the assumptions of Theorem \ref{thm_lwp},
the mild solution $u$ on $[0,T)$ is also a weak solution on $[0,T)$.
\end{proposition}
\begin{proof}
Let
$\psi \in C_0^{\infty}([0,T)\times \mathbb{R}^n)$
be a test function and we assume
${\rm supp\,} \psi \in [0,T_1) \times \mathbb{R}^n$,
where $T_1 < T$.
Then, Theorem \ref{thm_lwp} or Lemma \ref{lem_nl} yields
$\mathcal{N}(u(\tau)) \in L^{\infty}(0,T_1; L^{q}(\mathbb{R}^n))$.
Let
$F_j \in C_0^{\infty}([0,\infty)\times\mathbb{R}^n)\ (j=1,2,\ldots)$
be a sequence such that
\[
	\lim_{j\to \infty} \| F_j - \mathcal{N}(u) \|_{L^{\infty}(0,T_1; L^q(\mathbb{R}^n))} = 0.
\]
We also set
\[
	v_j := \int_0^t \mathcal{D}(t-\tau) F_j(\tau)\,d\tau,\quad
	u^{nl} := \int_0^t \mathcal{D}(t-\tau) \mathcal{N}(u(\tau))\,d\tau.
\]
Then, taking
$\psi = F_j - \mathcal{N}(u)$
in the proof of Lemma \ref{lem_duh}, we deduce
\begin{align*}
	\| v_j - u^{nl} \|_{L^{\infty}(0,T_1; L^2(\mathbb{R}^n))}
	\lesssim \| F_j - \mathcal{N}(u) \|_{L^{\infty}(0,T_1; L^q(\mathbb{R}^n))}
	\to 0\quad (j\to \infty)
\end{align*}
(indeed, in the proof of Lemma \ref{lem_duh}, take $\gamma = q$
and use the Sobolev embedding theorem).
Therefore, we have
\[
	\lim_{j\to\infty} \int_0^{T_1}\int_{\mathbb{R}^n}
		v_j ( \partial_t^2 - \Delta - \partial_t )\psi\,dxdt
	=
	\int_0^{T_1}\int_{\mathbb{R}^n}
		u^{nl} ( \partial_t^2 - \Delta - \partial_t )\psi\,dxdt.
\]
On the other hand, since $F_j$ is smooth and compactly supported with respect to $x$,
so is $v_j$ and hence,
using integration by parts we easily compute
\[
	\int_0^{T_1}\int_{\mathbb{R}^n}
		v_j ( \partial_t^2 - \Delta - \partial_t )\psi\,dxdt
	= \int_0^{T_1}\int_{\mathbb{R}^n} F_j \psi \,dxdt.
\]
Taking the limit $j\to \infty$ in the right-hand side
and noting $F_j \to \mathcal{N}(u)$ in $L^{\infty}(0,T_1; L^q(\mathbb{R}^n))$,
we have
\[
	\lim_{j\to\infty} \int_0^{T_1}\int_{\mathbb{R}^n}
		v_j ( \partial_t^2 - \Delta - \partial_t )\psi\,dxdt
	= \int_0^{T_1}\int_{\mathbb{R}^n} \mathcal{N}(u) \psi \,dxdt.
\]
Thus, we conclude
\[
	\int_0^{T_1}\int_{\mathbb{R}^n}
		u^{nl} ( \partial_t^2 - \Delta - \partial_t )\psi\,dxdt
	=\int_0^{T_1}\int_{\mathbb{R}^n} \mathcal{N}(u) \psi \,dxdt.
\]
In a similar way, approximating
$u_0, u_1$ by smooth functions and using the integration by parts,
we find that
\begin{align*}
	\varepsilon \int_0^{T_1}\int_{\mathbb{R}^n}
	\left( \tilde{\mathcal{D}}(t) u_0 + \mathcal{D}(t) u_1 \right)
	\left( \partial_t^2 - \Delta - \partial_t \right)\psi \,dxdt =
	\varepsilon \int_{\mathbb{R}^n}
		\left( (u_0 + u_1)\psi(0,x) - u_0 \psi_t(0,x) \right)\,dx.
\end{align*}
This completes the proof.
\end{proof}

Now, we are in the position to prove Theorem \ref{thm_lfupp}.
\begin{proof}[Proof of Theorem \ref{thm_lfupp}]
By Proposition \ref{prop_sol}, we see that
a mild solution is also a weak solution.
Therefore, we have
$T(\varepsilon) \le T_w(\varepsilon)$.
Hence, Proposition \ref{prop_bu} implies the desired estimate.
\end{proof}

\appendix
\section{Some auxiliary estimates}

\begin{lemma}\label{lem_fd}
Let
$\omega \in (0,1)$.
Then, there exists a constant
$C=C(\omega)>0$
such that we have
\[
	|\partial_j|^{\omega} \phi(x)
	= C \int_{\mathbb{R}} ( \phi (x- y) - \phi(x) ) \frac{dy_j}{|y_j|^{1+\omega}},
\]
where
$y=(0,\ldots,y_j,\ldots,0)$.
In particular, it follows that
\begin{align*}
	|\partial_j|^{\omega} (\phi \psi)(x)
	&=C \int_{\mathbb{R}}
		( \phi(x-y)\psi (x- y) - \phi(x)\psi(x) )
		\frac{dy_j}{|y_j|^{1+\omega}} \\
	&= \phi(x) |\partial_j|^{\omega} \psi(x)
	+ C \int_{\mathbb{R}}
		(\phi (x-y) - \phi(x) )\psi(x-y)
		\frac{dy_j}{|y_j|^{1+\omega}}.
\end{align*}
\end{lemma}
\begin{proof}
Recalling the identity
\[
	\mathcal{F}^{-1}[\phi] (x-y) = \mathcal{F}^{-1} [ e^{iy_j \xi_j} \phi ](x),
\]
we have
\begin{align*}
	\int_{\mathbb{R}} ( \phi (x- y) - \phi(x) ) \frac{dy_j}{|y_j|^{1+\omega}}
	&= \int_{\mathbb{R}}
		\mathcal{F}^{-1} \left[
			\left( e^{iy_j\xi_j} - 1 \right) \hat{\phi} \right](x) \frac{dy_j}{|y_j|^{1+\omega}} \\
	&= \frac{1}{(2\pi)^{n/2}} \int_{\mathbb{R}} \int_{\mathbb{R}^n}
		e^{ix\xi} \left( e^{iy_j\xi_j} - 1 \right) \hat{\phi} d\xi \frac{dy_j}{|y_j|^{1+\omega}} \\
	&= \frac{1}{(2\pi)^{n/2}} \int_{\mathbb{R}^n} e^{ix\xi} T(\xi_j) \hat{\phi}\,d\xi,
\end{align*}
where
\[
	T(\xi_j) = \int_{\mathbb{R}} \left( e^{iy_j\xi_j} - 1 \right) \frac{dy_j}{|y_j|^{1+\omega}}
	= 2 |\xi_j|^{\omega} \int_0^{\infty} (\cos y_j - 1 ) \frac{dy_j}{y_j^{1+\omega}}
	= C(\omega) |\xi_j|^{\omega}.
\]
Therefore, we conclude
\[
	\int_{\mathbb{R}} ( \phi (x- y) - \phi(x) ) \frac{dy_j}{|y_j|^{1+\omega}}
	= \frac{C}{(2\pi)^{n/2}} \int_{\mathbb{R}^n} e^{ix\xi} |\xi_j|^{\omega} \hat{\phi}\,d\xi
	= C |\partial_j|^{\omega} \phi (x).
\]
\end{proof}

\begin{lemma}[Sobolev-type inequality]\label{lem_sob}
Let $\beta \ge 0$. Then, we have
\[
	\| \langle \cdot \rangle^{\beta} \langle \nabla \rangle^{-1} \psi \|_{L^2}
	\lesssim \| \langle \cdot \rangle^{\beta} \psi \|_{L^q},
\]
where $q$ is defined in Section 1.2.
\end{lemma}

\begin{proof}
By the Plancherel theorem, it suffices to estimate
$\| \langle \nabla \rangle^{\beta} \langle \xi \rangle^{-1} \hat{\psi} \|_{L^2}$.
First, we note that
\[
	\| \langle \nabla \rangle^{\beta} \langle \xi \rangle^{-1} \hat{\psi} \|_{L^2}
	\lesssim
	\| \langle \xi \rangle^{-1} \hat{\psi} \|_{L^2}
	+ \| | \nabla |^{\beta} \langle \xi \rangle^{-1} \hat{\psi} \|_{L^2}.
\]
The Sobolev embedding implies
$\| \langle \nabla \rangle^{-1} \psi \|_{L^2} \lesssim \| \psi \|_{L^q}$
with $q= \max\{ 1, \frac{2n}{n+2} \}$,
which gives the estimate for the first term of the right-hand side.
Let us estimate the second term.
When $\beta \in \mathbb{Z}$, the Leibniz rule and the Sobolev embedding
and noting that $|\partial_j^k \langle \xi \rangle^{-1}| \lesssim \langle \xi \rangle^{-1}$
immediately imply
$\| | \nabla |^{\beta} \langle \xi \rangle^{-1} \hat{\psi} \|_{L^2}
\lesssim \| \langle x \rangle^{\beta} \psi \|_{L^q}$.
When $\beta \notin \mathbb{Z}$,
letting $\omega = \beta - [\beta] >0$ and
using Lemma \ref{lem_fd}, we have
\begin{align}
\nonumber
	\| |\nabla|^{\beta} ( \langle \xi \rangle^{-1} \hat{\psi} ) \|_{L^2}
	&\lesssim
	\sum_{j=1}^n \sum_{k=0}^{[\beta]}
		\left\| \left( \partial_j^{[\beta]-k} \langle \xi \rangle^{-1} \right)
			|\partial_j|^{k+\omega} \hat{\psi} \right\|_{L^2} \\
\label{a1}
	&\quad  + \sum_{j=1}^n \sum_{k=0}^{[\beta]}
		\left\| \int_{\mathbb{R}} \left(
			\partial_j^{[\beta]-k}\langle \xi \rangle^{-1}
				- \partial_j^{[\beta]-k} \langle \xi+\eta \rangle^{-1} \right)
			\partial_j^k \hat{\psi} (\xi + \eta ) \frac{d\eta_j}{|\eta_j|^{1+\omega}}
		\right\|_{L^2},
\end{align}
where
$\eta = (0,\ldots,\eta_j,\ldots,0)$.
Then, a straight forward calculation shows
\begin{align}
\label{a2}
	\left\| \left( \partial_j^{[\beta]-k} \langle \xi \rangle^{-1} \right)
			|\partial_j|^{k+\omega} \hat{\psi}  \right\|_{L^2}
	\lesssim
	\| \langle \xi \rangle^{-1} |\partial_j|^{k+\omega} \hat{\psi} \|_{L^2}
	\lesssim
	\| |x_j|^{k+\omega} \psi \|_{L^q},
\end{align}
where we have used the Sobolev embedding again.
Moreover, we claim that
\begin{align}
\label{a3}
	\left\| \int_{\mathbb{R}} \left(
			\partial_j^{[\beta]-k}\langle \xi \rangle^{-1}
				- \partial_j^{[\beta]-k} \langle \xi+\eta \rangle^{-1} \right)
			\partial_j^k \hat{\psi} (\xi + \eta ) \frac{d\eta_j}{|\eta_j|^{1+\omega}}
		\right\|_{L^2}
	\lesssim 
	\| |x_j|^k \psi \|_{L^q}.
\end{align}
Indeed, we first note that the left-hand side is bounded by
\begin{align*}
	&\left\| \int_{|\eta_j|\le 1} \left(
			\partial_j^{[\beta]-k}\langle \xi \rangle^{-1}
				- \partial_j^{[\beta]-k} \langle \xi+\eta \rangle^{-1} \right)
			\partial_j^k \hat{\psi} (\xi + \eta ) \frac{d\eta_j}{|\eta_j|^{1+\omega}}
		\right\|_{L^2}\\
	&\quad +\left\| \int_{|\eta_j|>1} \left(
			\partial_j^{[\beta]-k}\langle \xi \rangle^{-1}
				- \partial_j^{[\beta]-k} \langle \xi+\eta \rangle^{-1} \right)
			\partial_j^k \hat{\psi} (\xi + \eta ) \frac{d\eta_j}{|\eta_j|^{1+\omega}}
		\right\|_{L^2} =: I + I\!I.
\end{align*}
For $I$, the fundamental theorem of calculus implies
\begin{align*}
	I &\lesssim \int_{|\eta_j|\le 1}
		\left\| \int_0^1 \frac{d}{da} 
					\partial_j^{[\beta]-k}\langle \xi + a \eta \rangle^{-1}
				\,da  \cdot 
			\partial_j^k \hat{\psi} (\xi + \eta ) \right\|_{L^2}
			\frac{d\eta_j}{|\eta_j|^{1+\omega}} \\
	&\lesssim
		\int_{|\eta_j|\le 1}
		\left\| \int_0^1
					\partial_j^{[\beta]-k+1} \langle \xi + a\eta \rangle^{-1}
			\,da  \cdot 
			\partial_j^k \hat{\psi} (\xi + \eta ) \right\|_{L^2}
			\frac{d\eta_j}{|\eta_j|^{\omega}} \\
	&=
		\int_{|\eta_j|\le 1}
		\left\| \int_0^1
					\partial_j^{[\beta]-k+1} \langle \xi + a \eta \rangle^{-1}
				\,da  \cdot 
			\partial_j^k \hat{\psi} (\xi ) \right\|_{L^2}
			\frac{d\eta_j}{|\eta_j|^{\omega}} \\
	&\lesssim
		\int_{|\eta_j|\le 1}
		\left\|  \langle \xi \rangle^{-1} \partial_j^k \hat{\psi} (\xi ) \right\|_{L^2}
			\frac{d\eta_j}{|\eta_j|^{\omega}} \\
	&\lesssim \left\|  \langle \xi \rangle^{-1} \partial_j^k \hat{\psi} (\xi ) \right\|_{L^2}.
\end{align*}
This and the Sobolev embedding lead to
$I \lesssim \| |x_j|^k \psi \|_{L^q}$.
For $I\!I$, we first have
\begin{align*}
	I\!I &\lesssim
		\int_{|\eta_j|>1}\left\|  \left(
			\partial_j^{[\beta]-k}\langle \xi \rangle^{-1}
				- \partial_j^{[\beta]-k} \langle \xi+\eta \rangle^{-1} \right)
			\partial_j^k \hat{\psi} (\xi + \eta ) \right\|_{L^2}
			\frac{d\eta_j}{|\eta_j|^{1+\omega}} \\
		&\lesssim
		\int_{|\eta_j|>1}\left\| 
			\partial_j^{[\beta]-k}\langle \xi -\eta \rangle^{-1}
			\partial_j^k \hat{\psi} (\xi ) \right\|_{L^2(|\xi - \eta| \le |\xi|/2)}
			\frac{d\eta_j}{|\eta_j|^{1+\omega}} \\
		&\quad
		+\int_{|\eta_j|>1}\left\| 
			\partial_j^{[\beta]-k}\langle \xi-\eta \rangle^{-1}
			\partial_j^k \hat{\psi} (\xi ) \right\|_{L^2(|\xi - \eta| > |\xi|/2)}
			\frac{d\eta_j}{|\eta_j|^{1+\omega}} \\
		&\quad
		+\int_{|\eta_j|>1}\left\|
			\partial_j^{[\beta]-k} \langle \xi+\eta \rangle^{-1}
			\partial_j^k \hat{\psi} (\xi + \eta ) \right\|_{L^2(\mathbb{R}^n)}
			\frac{d\eta_j}{|\eta_j|^{1+\omega}}
		=: I\!I_1 + I\!I_2 + I\!I_3.
\end{align*}
The third term $I\!I_3$ is easily estimated as
\begin{align}
\label{a5}
	I\!I_3 \lesssim \left\|  \langle \xi \rangle^{-1} \partial_j^k \hat{\psi} (\xi ) \right\|_{L^2}
	\lesssim \| |x_j|^k \psi \|_{L^q}.
\end{align}
For $I\!I_2$, we note that
\[
	\left| \partial_j^{[\beta]-k}\langle \xi-\eta \rangle^{-1} \right|
	\lesssim \langle \xi-\eta \rangle^{-1} \lesssim \langle \xi \rangle^{-1} 
\]
holds for $|\xi - \eta| > |\xi|/2$ and hence,
\begin{align}
\label{a6}
	I\!I_2 &\lesssim
		\int_{|\eta_j|>1}\left\| 
			\langle \xi \rangle^{-1}
			\partial_j^k \hat{\psi} (\xi ) \right\|_{L^2(|\xi - \eta| > |\xi|/2)}
			\frac{d\eta_j}{|\eta_j|^{1+\omega}}
		\lesssim \left\| 
			\langle \xi \rangle^{-1}
			\partial_j^k \hat{\psi} (\xi ) \right\|_{L^2(\mathbb{R}^n)}
		\lesssim \| |x_j|^k \psi \|_{L^q}.
\end{align}
Finally, for $I\!I_1$, we first remark that
$|\eta| \ge |\xi| - |\xi-\eta| \ge \frac12 |\xi|$
holds.
Noting again that
$| \partial_j^{[\beta]-k}\langle \xi -\eta \rangle^{-1} | \lesssim \langle \xi - \eta \rangle^{-1}$
and applying the H\"{o}lder inequality, we see that
\begin{align}
\label{a4}
	I\!I_1&\lesssim
		\int_{|\eta_j|>1}
			\left\| \langle \xi -\eta \rangle^{-1} \right\|_{L^n(|\xi - \eta| \le |\xi|/2)}
			\left\|
				\partial_j^k \hat{\psi} (\xi )
			\right\|_{L^{q^{\prime}}(|\xi - \eta| \le |\xi|/2)}
			\frac{d\eta_j}{|\eta_j|^{1+\omega}}
\end{align}
for $n\ge 2$, where $q^{\prime}$ is the conjugate of $q$, and
\begin{align*}
	I\!I_1&\lesssim
		\int_{|\eta_j|>1}
			\left\| \langle \xi -\eta \rangle^{-1} \right\|_{L^2(|\xi - \eta| \le |\xi|/2)}
			\left\|
				\partial_j^k \hat{\psi} (\xi )
			\right\|_{L^{\infty}(|\xi - \eta| \le |\xi|/2)}
			\frac{d\eta_j}{|\eta_j|^{1+\omega}}
\end{align*}
for $n=1$.
In what follows, we only consider the case $n\ge 2$,
because the case $n=1$ is similar.
Using $|\eta_j| = |\eta| > \frac12 |\xi|$, we have
\[
	\left\| \langle \xi -\eta \rangle^{-1} \right\|_{L^n(|\xi - \eta| \le |\xi|/2)}
	\le \left\| \langle \xi -\eta \rangle^{-1} \right\|_{L^n(|\xi - \eta| \le |\eta_j|)}
	= \left\| \langle \Xi \rangle^{-1} \right\|_{L^n(|\Xi| \le |\eta_j|)}
	\lesssim (\log |\eta_j| + 1)^{\frac1n}.
\]
Also, we immediately obtain
\[
	\left\| \partial_j^k \hat{\psi} (\xi ) \right\|_{L^{q^{\prime}}(|\xi - \eta| \le |\xi|/2)}
	\lesssim
	\left\| \partial_j^k \hat{\psi} (\xi ) \right\|_{L^{q^{\prime}}(\mathbb{R}^n)}
	\lesssim
	\left\| |x_j|^k \psi \right\|_{L^q(\mathbb{R}^n)}.
\]
Combining them with \eqref{a4}, we conclude
\begin{align}
\label{a7}
	I\!I_1
	\lesssim \left\| |x_j|^k \psi \right\|_{L^q}
		\int_{|\eta_j|>1} (\log |\eta_j| + 1)^{\frac1n}\frac{d\eta_j}{|\eta_j|^{1+\omega}}
	\lesssim \left\| |x_j|^k \psi \right\|_{L^q}.
\end{align}
By \eqref{a5}, \eqref{a6} and \eqref{a7}, we have
$I\!I \lesssim \| |x_j|^k \psi \|_{L^q}$.
Consequently, we obtain
\begin{align*}
	\| \langle \cdot \rangle^{\beta} \langle \nabla \rangle^{-1} \psi \|_{L^2} \lesssim
	\| \psi \|_{L^q}
	+ \sum_{j=1}^n \sum_{k=0}^{[\beta]}
	\left( 
	\| | x_j|^k \psi \|_{L^q}
	+ \|  |x_j |^{k+\omega} \psi \|_{L^q} \right)
	\lesssim \| \langle x \rangle^{\beta} \psi \|_{L^q},
\end{align*}
which shows the assertion.
\end{proof}

\section*{Acknowledgments}
The authors are deeply grateful to
Professor Mitsuru Sugimoto for useful suggestions and comments.
This work was supported in part by
Grant-in-Aid for JSPS Fellows 15J01600 (Y. Wakasugi),
15J02570 (T. Inui),
26.1884 (M. Ikeda),
and
Grant-in-Aid for Young Scientists (B) 15K17571 (M. Ikeda)
of Japan Society for the Promotion of Science.







%
%
%
\end{document}